\documentclass[11pt,reqno,a4]{amsart}

\let\Setlength\setlength 
\usepackage{calc}
\newlength{\arrayrulewidthOriginal}
\newcommand{\Cline}[2]{%
  \noalign{\global\Setlength{\arrayrulewidthOriginal}{\arrayrulewidth}}%
  \noalign{\global\Setlength{\arrayrulewidth}{#1}}\cline{#2}%
  \noalign{\global\Setlength{\arrayrulewidth}{\arrayrulewidthOriginal}}}

\usepackage[english]{babel}
\usepackage{amsthm}
\usepackage{graphicx}
\usepackage{xcolor} 
\usepackage{transparent}
\usepackage{amssymb}
\usepackage{amsmath}
\usepackage{mathtools}
\usepackage[arrow, matrix, curve]{xy}
\usepackage{units} 
\usepackage{MnSymbol} 
\usepackage{here} 
\usepackage{enumerate} 
\usepackage{comment} 
\usepackage{listings} 
\usepackage{pdfpages} 
\usepackage{multirow} 
\usepackage{bigdelim}
\usepackage{multicol} 
\usepackage{cancel} 
\usepackage{wrapfig} 
\usepackage{float}
\usepackage{hyperref}
\usepackage{tikz-cd}
\usepackage{enumitem} 
\usepackage{bm}
\usepackage{colortbl} 
\usepackage{bm} 
\usepackage{dsfont}

\newcommand*\circled[1]{\tikz[baseline=(char.base)]{
    \node[shape=circle,draw,inner sep=2pt] (char) {#1};}} 

\usetikzlibrary{%
  matrix,%
  calc,%
  arrows%
}

\restylefloat{table}

\hypersetup{
    colorlinks,
    linkcolor={red!70!black},
    citecolor={green!60!black},
    urlcolor={blue!80!black}
}

\theoremstyle{dotless}
\newtheorem{corollary}{Corollary}[section]
\newtheorem{lemma}[corollary]{Lemma}
\newtheorem{theorem}[corollary]{Theorem}
\newtheorem{proposition}[corollary]{Proposition}
\newtheorem*{theorem*}{Theorem}
\theoremstyle{definition}
\newtheorem{construction}[corollary]{Construction}
\newtheorem{definition}[corollary]{Definition}
\newtheorem{remark}[corollary]{Remark}
\newtheorem{example}[corollary]{Example}

\newtheorem{notation}[corollary]{Notation}

\newcommand{\CR}{\operatorname{cr}}

\newcommand{\mult}{\operatorname{mult}}

\newcommand{\ft}{\operatorname{ft}}
\newcommand{\ev}{\operatorname{ev}}

\newcommand{\val}{\operatorname{val}}

\newcommand{\inc}{\operatorname{inc}}
\newcommand{\out}{\operatorname{out}}
\newcommand{\flow}{\operatorname{flow}}
\newcommand{\leak}{\operatorname{leak}}
\newcommand{\rec}{\operatorname{rec}}

\newcommand{\floor}{\operatorname{floor}}
\newcommand{\la}{\operatorname{lab}}
\newcommand{\sgn}{\operatorname{sgn}}

\begin{document}
\pagenumbering{arabic}

\author{Christoph Goldner}
\address{ Eberhard Karls Universit\"{a}t T\"{u}bingen, Germany}
\email{\href{mailto:christoph.goldner@math.uni-tuebingen.de}{christoph.goldner@math.uni-tuebingen.de}}

\title{Counting tropical rational space curves with cross-ratio constraints}

\keywords{Enumerative geometry, tropical geometry, cross-ratios, tropical cross-ratios, space curves, floor diagrams, cross-ratio floor diagrams}
\subjclass[2010]{14N10, 14T05}

\begin{abstract}
This is a follow-up paper of \cite{CR1}, where rational curves in surfaces that satisfy general positioned point and cross-ratio conditions were enumerated. A suitable correspondence theorem provided in \cite{IlyaCRC} allowed us to use tropical geometry, and, in particular, a degeneration technique called \textit{floor diagrams}. This correspondence theorem also holds in higher dimension.

In the current paper, we introduce so-called \textit{cross-ratio floor diagrams} and show that they allow us to determine the number of rational space curves that satisfy general positioned point and cross-ratio conditions. Moreover, \textit{graphical contributions} are introduced which provide a novel and structured way of understanding multiplicities of floor decomposed curves in $\mathbb{R}^3$. Additionally, so-called \textit{condition flows} on a tropical curve are used to reflect how conditions imposed on a tropical curve yield different types of edges. This concept is applicable in arbitrary dimension.
\end{abstract}

\maketitle


\section*{Introduction}

A \textit{cross-ratio} is, like a point condition, a condition that can be imposed on curves. More precisely, a cross-ratio is an element of the ground field associated to four collinear points. It encodes the relative position of these four points to each other. It is invariant under projective transformations and can therefore be used as a constraint that four points on $\mathbb{P}^1$ should satisfy. So a cross-ratio can be viewed as a condition on elements of the moduli space of $n$-pointed rational stable maps to a toric variety. In case of rational plane curves, a cross-ratio condition appears as a main ingredient in the proof of the famous Kontsevich's formula.
Thus the following enumerative problem naturally comes up:

\begin{enumerate}[label=(1),ref=(1)]
\item \label{Question1}
\fbox{\parbox{14.5cm}{Determine the number of rational curves in $\mathbb{P}^m_{\mathbb{C}}$ of a given degree that satisfy general positioned point conditions and cross-ratio conditions.}}
\end{enumerate}

In the past, tropical geometry proved to be an effective tool to answer enumerative questions. To successfully apply tropical geometry to an enumerative problem, a so-called \textit{correspondence theorem} is required.
A correspondence theorem states that an enumerative number equals its tropical counterpart, where in tropical geometry we have to count each tropical object with a suitable multiplicity reflecting the number of classical objects in our counting problem that tropicalize to the given tropical object.
The first celebrated correspondence theorem was proved by Mikhalkin \cite{MikhalkinFundamental}.
Tyomkin \cite{IlyaCRC} proved a correspondence theorem that involves cross-ratios. Using Tyomkin's correspondence theorem, question \ref{Question1} can be rephrased as

\begin{enumerate}[label=(2),ref=(2)]
\item \label{Question2}
\fbox{\parbox{14.5cm}{Determine the weighted number of rational tropical curves in $\mathbb{R}^m$ of a given degree that satisfy general positioned point conditions and tropical cross-ratio conditions.}}
\end{enumerate}

Notice that in the rephrased question \textit{tropical cross-ratios} are considered. Tropical cross-ratios are the tropical counterpart to non-tropical cross-ratios.
Mikhalkin \cite{MikhalkinCRC} introduced a tropical version of cross-ratios under the name ``tropical double ratio'' to embed the moduli space of $n$-marked abstract rational tropical curves $\mathcal{M}_{0,n}$ into $\mathbb{R}^N$ in order to give it the structure of a balanced fan.
Roughly speaking, a tropical cross-ratio fixes the sum of lengths of a collection of bounded edges of a rational tropical curve.

\begin{example}\label{ex:introduction_trop_CR_3_dim}
Figure \ref{Figure36} shows a rational tropical degree $2$ curve $C$ in $\mathbb{R}^3$. The curve $C$ satisfies three point conditions with its contracted ends labeled by $1,2,3$, and is satisfies one tangency condition (which is a line that is indicated by dots) with its end labeled with $4$. Moreover, $C$ satisfies the tropical cross-ratio $\lambda'=(12|34)$ which determines the bold red length.

\begin{figure}[H]
\centering
\def\svgwidth{360pt}
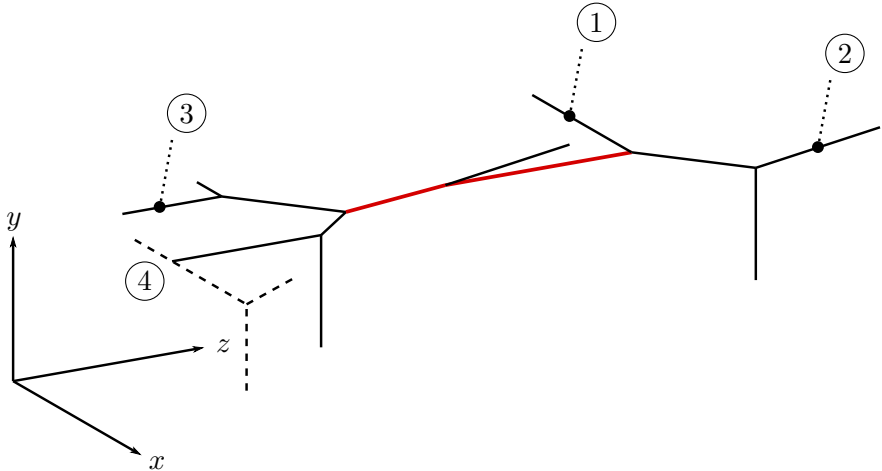
\caption{The tropical curve $C$ from Example \ref{ex:introduction_trop_CR_3_dim} satisfying three point conditions and one tropical cross-ratio.}
\label{Figure36}
\end{figure}

\end{example}

In case of plane curves (i.e. $m=2$) question \ref{Question2} (and therefore question \ref{Question1}) was answered exhaustively by giving a recursive formula \cite{GeneralKontsevich} and by explicitly constructing plane rational tropical curves that satisfy the given conditions \cite{CR1}.
The present paper determines the numbers of questions \ref{Question1}, \ref{Question2} in case of space curves, i.e. $m=3$. To do so, we combine methods used in \cite{CR1} with ideas developed in \cite{GeneralKontsevich}.\\

The main tool we use are so-called floor diagrams. Floor diagrams are graphs that arise from so-called floor decomposed tropical curves by forgetting some information. Floor diagrams were introduced by Mikhalkin and Brugall\'{e} in \cite{MikhalkinBrugalleFloorDiagIntroduction, MikhalkinBrugalle} to give a combinatorial description of Gromov-Witten invariants of Hirzebruch surfaces. Floor diagrams have also been used to establish polynomiality of the node polynomials \cite{MikhalkinFomin} and to give an algorithm to compute these polynomials in special cases, see \cite{BlockNodePoly}. Moreover, floor diagrams have been generalized, for example in case of $\Psi$-conditions, see \cite{PsiFloorDiagrams}, or for counts of curves relative to a conic \cite{Bru}. A generalization of floor diagrams that includes tropical cross-ratios was used in \cite{CR1} to answer question \ref{Question2} combinatorially for $m=2$, i.e. for plane curves.

In the present paper, we want to extend the so-called \textit{cross-ratio floor diagrams} of \cite{CR1} to give a combinatorial answer to question \ref{Question2} in case of $m=3$, i.e. for space curves.
So our aim is to define cross-ratio floor diagrams in case of $m=3$ from which floor decomposed curves can be reconstructed.
Defining cross-ratio floor diagrams associated to floor decomposed curves is a two step problem. First, we need to define the cross-ratio floor diagram itself, i.e. a degeneration of a floor decomposed curve. The difficulties lie in the second step, which is to define the multiplicity of a cross-ratio floor diagram. Multiplicities are necessary since different floor decomposed curves may degenerate to the same cross-ratio floor diagram. Moreover, we want \dots
\begin{itemize}
\item[(a)] \dots our multiplicity to be local, i.e. we want to define the multiplicity of a cross-ratio floor diagram as a product over vertex multiplicities of that floor diagram.
\item[(b)] \dots to make sure that such a local vertex multiplicity encodes the number of floors degenerating to that vertex such that we can glue the curve pieces degenerating to each vertex in the cross-ratio floor diagram to a curve that degenerates to the whole floor diagram.
\end{itemize}

The main result of this paper is to extend current degeneration techniques, i.e. to introduce suitable cross-ratio floor diagrams that yield a combinatorial solution of question \ref{Question2}, see Theorem \ref{thm:counting_CR_floor_diagrams=counting_curves}. It turns out that the multiplicities needed for such cross-ratio floor diagrams are exactly the numbers general Kontsevich's formula \cite{GeneralKontsevich} provides, i.e. in order to answer question \ref{Question2} for $m=3$, it is necessary to know its answer in case of $m=2$.
Defining the multiplicity of such a cross-ratio floor diagram is done via a similar approach as in \cite{GeneralKontsevich}. In fact, concepts of \cite{GeneralKontsevich} are systematically generalized to make them applicable in higher dimensions, as we hope that they might be useful in answering question \ref{Question2} for $m>3$. One of these concepts is that of \textit{condition flows} on tropical curves. Recently, a similar concept was independently introduced in \cite{MandelRuddat2019}. Condition flows reflect how conditions imposed on a tropical curve yield different types of edges. These flows generalize different ad hoc methods that appeared in different contexts \cite{BroccoliCurves, RanganathanCavalieriMarkwigJohnson, CR1}, see Remark \ref{remark:condition_flows_generalize_ad_hoc_methods}. Moreover, \textit{graphical contributions} are introduced which provide a novel and structured way of understanding multiplicities of floor decomposed curves in $\mathbb{R}^3$, see Proposition \ref{prop:graphical_contributions}.

Applying Tyomkin's correspondence theorem \cite{IlyaCRC} to the main result (Theorem \ref{thm:counting_CR_floor_diagrams=counting_curves}) then immediately yields a combinatorial answer to question \ref{Question1}, see Corollary \ref{cor:count_CR_floor_diag=classical_count}. We remark that our tropical approach is capable of not only determining the algebro-geometric numbers we are looking for in question \ref{Question1}, but also of determining further tropical numbers involving tangency conditions of codimension one and two. Moreover, our approach is not limited to tropical curves corresponding to curves in $\mathbb{P}^3$, but also yields curve counts in other toric varieties.

\subsection*{Future research}
Although concepts like condition flows and floor decomposition of tropical curves carry over to higher dimension, the problem whether there is a cross-ratio floor diagram approach to determine the numbers of Question \ref{Question1} for $m>3$ is still open. How multiplicities of floor decomposed tropical curves behave under cutting the tropical curves into pieces is currently unknown. This prevents us from extending the cross-ratio floor diagram approach to higher dimensions.

\subsection*{Organization of the paper}
The preliminary section introduces notation, collects background on tropical moduli spaces, tropical intersection theory, tropical cross-ratios and recalls previous results.
Right after the preliminary section the new and general concept of \textit{condition flows} is established. Floor decomposed curves and their multiplicities are studied via \textit{graphical contributions} in the following section.
Cross-ratio floor diagrams and their multiplicities are then introduced. The main result relating counts of cross-ratio floor diagrams to counts of rational tropical space curves of a given degree satisfying given conditions is proved in the last section.

\subsection*{Acknowledgements}
The author would like to thank Hannah Markwig for valuable feedback and helpful discussions. The author gratefully acknowledges partial support by DFG-collaborative research center TRR 195 (INST 248/237-1).

\section{Preliminaries}

We recall some standard notations and definitions from tropical geometry \cite{MikhalkinCRC, KontsevichPaper,GathmannKerberMarkwig} and give a brief overview of the necessary tropical intersection theory.

Besides this, we try to make notations used as clear as possible by introducing notations in separate blocks to which we refer later.

\begin{notation}\label{notation:underlined_symbols}
We write $[m]:=\lbrace 1,\dots,m \rbrace$ if $0\neq m\in\mathbb{N}$, and if $m=0$, then define $[m]:=\emptyset$.
Underlined symbols indicate a set of symbols, e.g. $\underline{n}\subset [m]$ is a subset $\lbrace 1,\dots, m\rbrace$. We may also use sets $S$ of symbols as an index, e.g. $p_S$, to refer to the set of all symbols $p$ with indices taken from $S$, i.e. $p_S:=\lbrace p_i\mid i\in S \rbrace$. The $\#$-symbol is used to indicate the number of elements in a set, for example $\#[m]=m$.
\end{notation}

\subsection*{Tropical intersection theory}
This subsection collects intersection theoretic background. For more details of tropical intersection theory see \cite{FultonSturmfels, Rau,Allermann,FirstStepsIntersectionTheory, Katz2012, IntersectionMatroidalFans,  AllermannHampeRau, JohannesIntersectionsonTropModuliSpaces}. 

\subsection{Tropical intersection theory}
\begin{definition}[Normal vectors and balanced fans]
Let $V:=\Gamma\otimes_{\mathbb{Z}}\mathbb{R}$ be the real vector space associated to a given lattice $\Gamma$ and let $X$ be a fan in $V$. The lattice generated by $\operatorname{span}(\kappa)\cap\Gamma$, where $\kappa$ is a cone of $X$, is denoted by $\Gamma_\kappa$. Let $\sigma$ be a cone of $X$ and $\tau$ be a face of $\sigma$ of dimension $\dim(\tau)=\dim(\sigma)-1$ (we write $\tau<\sigma$). A vector $u_{\sigma}\in\Gamma_\sigma$ that generates $\Gamma_\sigma / \Gamma_\tau$ such that $u_\sigma+\tau\subset\sigma$ defines a class $u_{\sigma / \tau}:=[u_\sigma]\in\Gamma_\sigma / \Gamma_\tau$ that does not depend on the choice of $u_\sigma$. This class is called \textit{normal vector of $\sigma$ relative to $\tau$}.

$X$ is a \textit{weighted fan of dimension $k$} if $X$ is of pure dimension $k$ and there are weights on its facets (i.e. its $k$-dimensional faces), that is there is a map $\omega_X:X^{(k)}\to\mathbb{Z}$. The number $\omega_X(\sigma)$ is called \textit{weight} of the facet $\sigma$ of $X$. To simplify notation, we write $\omega(\sigma)$ if $X$ is clear. Moreover, a weighted fan $(X,\omega_X)$ of dimension $k$ is called a \textit{balanced} fan of dimension $k$ if
\begin{align*}
\sum_{\sigma \in X^{(k)}, \tau < \sigma} \omega(\sigma)\cdot u_{\sigma / \tau} = 0
\end{align*}
holds in $V/\langle\tau\rangle_{\mathbb{R}}$ for all faces $\tau$ of dimension $\dim(\tau)=\dim(\sigma)-1$.
\end{definition}

\begin{definition}[Affine cycles]
Let $V:=\Gamma\otimes_{\mathbb{Z}}\mathbb{R}$ be the real vector space associated to a given lattice $\Gamma$. A \textit{tropical fan $X$ (of dimension $k$)} is a balanced fan of dimension $k$ in $V$ and $[(X,\omega_X)]$ denotes the refinement class of $X$ with weights $\omega_X$ (see Definition 2.8 and Construction 2.10 of \cite{FirstStepsIntersectionTheory}). Such a class is also called an \textit{affine (tropical) $k$-cycle} in $V$. Denote the set of all affine $k$-cycles in $V$ by $Z^{\textrm{aff}}_k(V)$. 
For a fan $X$ in $V$, we may also define an \textit{affine $k$-cycle} in $X$ as an element $[(Y,\omega_Y)]$ of $Z^{\textrm{aff}}_k(V)$ such that the support of $Y$ with nonzero weights lies in the support of $X$ (see Definition 2.15 of \cite{FirstStepsIntersectionTheory}). Define 
$|[(X,\omega_X)]|:=X^*$, where $X^*$ denotes the support of $X$ with nonzero weights.

The set $Z^{\textrm{aff}}_k(V)$ (resp. $Z^{\textrm{aff}}_k([(X,\omega_X)])$) can be turned into an abelian group by taking unions while refining appropriately.
\end{definition}

\begin{definition}[Rational functions]
Let $[(X,\omega_X)]$ be an affine $k$-cycle. A \textit{(nonzero) rational function on $[(X,\omega_X)]$} is a continuous piecewise linear function $\varphi:|[(X,\omega_X)]|\to\mathbb{R}$, i.e. there exists a representative $(X,\omega_X)$ of $[(X,\omega_X)]$ such that on each cone $\sigma\in X$ the map $\varphi$ is the restriction of an integer affine linear function. The set of (nonzero) rational functions of $[(X,\omega_X)]$ is denoted by $\mathcal{K}^*([(X,\omega_X)])$.

Define $\mathcal{K}\left([(X,\omega_X)]\right):=\mathcal{K}^*([(X,\omega_X)])\cup \{-\infty\}$ such that $(\mathcal{K}([(X,\omega_X)]),\operatorname{max},+)$ is a semifield, where the constant function $-\infty$ is the ``zero" function.
\end{definition}

\begin{definition}[Divisor associated to a rational function]\label{definition:Assoc_Weil_Div}
Let $[(X,\omega_X)]$ be an affine $k$-cycle in $V=\Gamma\otimes_{\mathbb{Z}}\mathbb{R}$ and $\varphi\in\mathcal{K}^*([(X,\omega_X)])$ a rational function on $[(X,\omega_X)]$. Let $(X,\omega)$ be a representative of $[(X,\omega_X)]$ on whose cones $\varphi$ is affine linear and denote these linear pieces by $\varphi_\sigma$. We denote by $X^{(i)}$ the set of all $i$-dimensional cones of $X$. We define $\operatorname{div}(\varphi):=\varphi\cdot [(X,\omega_X)]:= [(\bigcup_{i=0}^{k-1}X^{(i)},\omega_{\varphi})]\in Z^{\textrm{aff}}_{k-1}([(X,\omega_X)])$, where
\begin{align*}
\omega_\varphi : X^{(k-1)} &\to \mathbb{Z}\\
\tau &\mapsto\sum_{\sigma \in X^{(k)}, \tau < \sigma} \varphi_\sigma(\omega(\sigma)v_{\sigma/\tau})-\varphi_\tau \left( \sum_{\sigma \in X^{(k)}, \tau < \sigma}\omega(\sigma)v_{\sigma/\tau}\right)
\end{align*}
and the $v_{\sigma/\tau}$ are arbitrary representatives of the normal vectors $u_{\sigma/\tau}$. If $[(Y,\omega_Y)]$ is an affine $k$-cycle in $[(X,\omega_X)]$, we define $\varphi\cdot [(Y,\omega_Y)]:=\varphi\mid_{|[(Y,\omega_Y)]|}\cdot [(Y,\omega_Y)]$.
\end{definition}

\begin{definition}[Affine intersection product]
Let $[(X, \omega_X)]$ be an affine $k$-cycle. The subgroup of globally linear functions in $\mathcal{K}^*([(X, \omega_X)])$ with respect to $+$ is denoted by $\mathcal{O}^*([(X, \omega_X)])$. We define the \textit{group of affine Cartier divisors of $[(X, \omega_X)]$} to be the quotient group $\operatorname{Div}([(X, \omega_X)]):=\mathcal{K}^*([(X, \omega_X)])/\mathcal{O}^*([(X, \omega_X)])$. Let $[\varphi]\in\operatorname{Div}([(X, \omega_X)])$ be a Cartier divisor. The divisor associated to this function is denoted by $\operatorname{div}([\varphi]):=\operatorname{div}(\varphi)$ and is well-defined. The following bilinear map is called \textit{affine intersection product}
\begin{align*}
\cdot\, : \operatorname{Div}([(X, \omega_X)])\times Z^{\textrm{aff}}_k([(X, \omega_X)]) &\to Z^{\textrm{aff}}_{k-1}([(X, \omega_X)])\\
([\varphi],[(Y, \omega_Y)]) &\mapsto [\varphi]\cdot [(Y, \omega_Y)] := \varphi\cdot [(Y, \omega_Y)].
\end{align*}
\end{definition}

\begin{definition}[Morphisms of fans]
Let $X$ be a fan in $V=\Gamma\otimes_{\mathbb{Z}}\mathbb{R}$ and $Y$ a fan in $V'=\Gamma'\otimes_{\mathbb{Z}}\mathbb{R}$. A \textit{morphism} $f:X\to Y$ is a $\mathbb{Z}$-linear map from $|X|\subseteq V$ to $|Y|\subseteq V'$ induced by a $\mathbb{Z}$-linear map on the lattices. A \textit{morphism of weighted fans} is a morphism of fans. A \textit{morphism of affine cycles} $f:[(X,\omega_X)]\to [(Y,\omega_Y)]$ is a morphism of weighted fans $f:X^*\to Y^*$ that is independent of the choice of representatives, where $X^*$ (resp. $Y^*$) denotes the support of $X$ (resp. $Y$) with nonzero weight.
\end{definition}

\begin{definition}[Push-forward of affine cycles]
Let $V=\Gamma\otimes_{\mathbb{Z}}\mathbb{R}$ and $V'=\Gamma'\otimes_{\mathbb{Z}}\mathbb{R}$. Let $[(X, \omega_X)]\in Z^{\textrm{aff}}_m(V)$ and $[(Y, \omega_Y)]\in Z^{\textrm{aff}}_n(V')$ be cycles with representatives $(X,\omega_X)$ and $(Y,\omega_Y)$. Let $f:X\to Y$ be a morphism. Choosing a refinement of $(X,\omega_X)$, the set of cones
\begin{align*}
f_*X:=\{f(\sigma)\mid \sigma\in X \textrm{ contained in a maximal cone of $X$ on which $f$ is injective} \}
\end{align*}
is a tropical fan in $V'$ of dimension $m$ with weights
\begin{align*}
\omega_{f_*X}(\sigma'):=\sum_{\sigma\in X^{(m)}: \, f(\sigma)=\sigma'} \omega_X(\sigma)\cdot |\Gamma_{\sigma'}'/f(\Gamma_\sigma)|
\end{align*}
for all $\sigma'\in f_*X^{(m)}$. The equivalence class of $(f_*X,\omega_{f_*X})$ is uniquely determined by the equivalence class of $(X,\omega_X)$.
For $[(Z,\omega_Z)]\in Z^{\textrm{aff}}_k([(X,\omega_X)])$ we define
\begin{align*}
f_*[(Z,\omega_Z)]:=[(f_*(Z^*),\omega_{f_*(Z^*)})]\in Z^{\textrm{aff}}_k([(Y,\omega_Y)])
\end{align*}
The map
\begin{align*}
Z^{\textrm{aff}}_k([(X, \omega_X)])\to Z^{\textrm{aff}}_k([(Y, \omega_Y)]), \, [(Z, \omega_Z)]\mapsto f_*[(Z, \omega_Z)]
\end{align*}
is well-defined, $\mathbb{Z}$-linear and $f_*[(Z, \omega_Z)]$ is called \textit{push-forward of $[(Z, \omega_Z)]$ along $f$}.
\end{definition}

\begin{definition}[Pull-back of Cartier divisors]
Let $[(X, \omega_X)]\in Z^{\textrm{aff}}_m(V)$ and $[(Y, \omega_Y)]\in Z^{\textrm{aff}}_n(V')$ be cycles in $V=\Gamma\otimes_{\mathbb{Z}}\mathbb{R}$ and $V'=\Gamma'\otimes_{\mathbb{Z}}\mathbb{R}$. Let $f:[(X, \omega_X)]\to [(Y, \omega_Y)]$ be a morphism. The map
\begin{align*}
\operatorname{Div}([(Y, \omega_Y)])&\to\operatorname{Div}([(X, \omega_X)])\\
[h]&\mapsto f^*[h]:=[h\circ f]
\end{align*}
is well-defined, $\mathbb{Z}$-linear and $f^*[h]$ is called \textit{pull-back of $[h]$ along $f$}.
\end{definition}

\begin{proposition}[Projection formula]\label{prop:projection_formula}
Let $[(X, \omega_X)]\in Z^{\operatorname{aff}}_m(V)$ and $[(Y, \omega_Y)]\in Z^{\operatorname{aff}}_n(V')$ be cycles in $V=\Gamma\otimes_{\mathbb{Z}}\mathbb{R}$ and $V'=\Gamma'\otimes_{\mathbb{Z}}\mathbb{R}$. Let $f:[(X, \omega_X)]\to [(Y, \omega_Y)]$ be a morphism. Let $[(Z, \omega_Z)]\in Z^{\operatorname{aff}}_k([(X,\omega_X)])$ be a cycle and let $\varphi\in\operatorname{Div}\left([(Y, \omega_Y)]\right)$. Then the equality
\begin{align*}
\varphi\cdot\left( f_{*}[(Z, \omega_Z)] \right)=f_{*}\left( f^{*}\varphi\cdot [(Z, \omega_Z)] \right)\in Z^{\operatorname{aff}}_{k-1}([(Y, \omega_Y)])
\end{align*}
holds.
\end{proposition}

So far, we introducted affine cycles only. Affine cycles are building blocks of \textit{abstract} cycles. Since the whole ``affine-to-abstract"-procedure is quite technical, we omit it here and refer to section 5 of \cite{FirstStepsIntersectionTheory} instead. We want to remark that the projection formula (Proposition \ref{prop:projection_formula}) also holds for abstract cycles.
For our purposes the following definition of abstract cycles is sufficient:

\begin{definition}[Abstract cycles]
An \textit{abstract $k$-cycle} $C$ is a class under a refinement relation of a balanced polyhedral complex of pure dimension $k$ which is locally isomorphic to tropical fans.
\end{definition}

\begin{remark}[Rational functions on abstract cycles]
In the same way rational functions on affine cycles led to an affine intersection product, one can also consider \textit{rational functions on abstract cycles} to obtain a intersection product.
Again, we want to omit technicalities and refer to Definition 6.1 of \cite{FirstStepsIntersectionTheory} instead. The main point of considering rational functions on abstract cycles is that they are no longer piecewiese linear but pieceweise affine linear.

As we see below, it happens that we start with an affine cycle $[(X,\omega_X)]$ and want to intersect it with a rational function $f$ that is pieceweise affine linear. In order to do so, we need to refine $[(X,\omega_X)]$ in such a way that $f$ is linear on faces. Hence $[(X,\omega_X)]$ becomes a polyhedral complex which is a representative of an abstract cycle $C$. Then we can intersect $f$ with $C$.
\end{remark}

In the following we want to restrict to tropical intersection theory on $\mathbb{R}^n$.

\begin{definition}[Degree map]\label{definition:degree_map}
Let $A_0(\mathbb{R}^n)$ denote the set of abstract $0$-cycles in $\mathbb{R}^n$ up to rational equivalence. The map
\begin{align*}
\operatorname{deg}:A_0(\mathbb{R}^n) &\to \mathbb{Z}\\
[\omega_1 P_1+\dots+\omega_r P_r] &\mapsto \sum_{i=1}^r \omega_i
\end{align*}
is a well-defined morphism and for $D\in A_0(\mathbb{R}^n)$ the number $\operatorname{deg}(D)$ is called the \textit{degree of $D$}. 
\end{definition}

\subsection{Tropical moduli spaces}
This subsection collects background from \cite{MikhalkinCRC, KontsevichPaper,GathmannKerberMarkwig}.

\begin{definition}[Moduli space of abstract rational tropical curves]
We use Notation \ref{notation:underlined_symbols}. An \textit{abstract rational tropical curve} is a metric tree $\Gamma$ with unbounded edges called \textit{ends} and with $\val(v)\geq 3$ for all vertices $v\in\Gamma$. It is called $N$-\textit{marked abstract tropical curve} $(\Gamma,x_{[N]})$ if $\Gamma$ has exactly $N$ ends that are labeled with pairwise different $x_1,\dots,x_N\in\mathbb{N}$. Two $N$-marked tropical curves $(\Gamma,x_{[N]})$ and $(\tilde{\Gamma},\tilde{x}_{[N]})$ are isomorphic if there is a homeomorphism $\Gamma\to \tilde{\Gamma}$ mapping $x_i$ to $\tilde{x}_i$ for all $i$ and each edge of $\Gamma$ is mapped onto an edge of $\tilde{\Gamma}$ by an affine linear map of slope $\pm 1$. The set $\mathcal{M}_{0,N}$ of all $N$-marked tropical curves up to isomorphism is called \textit{moduli space of $N$-marked abstract tropical curves}. Forgetting all lengths of an $N$-marked tropical curve gives us its \textit{combinatorial type}.
\end{definition}

\begin{remark}[$\mathcal{M}_{0,N}$ is a tropical fan]
The moduli space $\mathcal{M}_{0,N}$ can explicitly be embedded into a $\mathbb{R}^t$ such that $\mathcal{M}_{0,N}$ is a tropical fan of pure dimension $N-3$ with its fan structure given by combinatorial types and all its weights are one, i.e. $\mathcal{M}_{0,n}$ represents an affine cycle in $\mathbb{R}^t$. This allows us to use tropical intersection theory on $\mathcal{M}_{0,n}$. For an example, see Figure \ref{Example_M_0_4}.
\end{remark}

\begin{figure}[H]
\centering
\def\svgwidth{150pt}
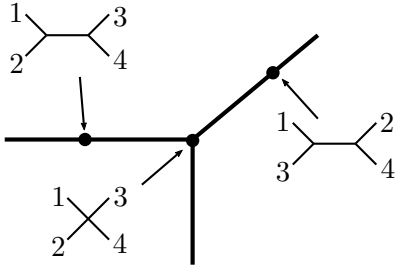
\caption{One way of embedding the moduli space $\mathcal{M}_{0,4}$ into $\mathbb{R}^2$ centered at the origin of $\mathbb{R}^2$. The length of a bounded edge of an abstract tropical curve depicted above is given by the distance of the point in $\mathcal{M}_{0,4}$ corresponding to this curve from the origin of $\mathbb{R}^2$. The ends of $\mathcal{M}_{0,4}$ correspond to different distributions of labels on ends of abstract tropical curves with four ends. All cases are $(12|34), (13|24), (14|23)$.}
\label{Example_M_0_4}
\end{figure}

\begin{definition}[Degree]\label{def:degree}
A tuple $(\Delta,l)$ consisting of a finite multiset $\Delta$ and a map $l$ is called a \textit{degree} in $\mathbb{R}^m$ if:
\begin{itemize}
\item[(1)] Each entry $v$ of $\Delta$ is a nonzero element of $\mathbb{Z}^m$ such that $\sum_{v\in\Delta} v=0$ and $\langle v \mid v\in \Delta \rangle =\mathbb{R}^m$.
\item[(2)] Each entry of $\Delta$ is equipped with a unique natural number called \textit{label}, i.e. the given map $l:\Delta\to [\#\Delta ]$ is a bijection.
\end{itemize}
Let $v=(v_1,\dots,v_m)\in\Delta$, then $\gcd(v_1,\dots,v_m)$ is called \textit{weight} of $v$. Most of the time the map $l$ is suppressed in the notation, i.e. we usually write $\Delta$ and assume that elements of $\Delta$ are labeled.
\end{definition}

\begin{notation}\label{notation:standard_directions_and_alpha_beta_degrees}
The following degrees are used often throughout the paper:
Let $e_i:=(\delta_{ti})_{t=1,\dots,m}$ for $i=1,\dots,m$ denote the vectors of the standard basis of $\mathbb{R}^m$, and define $e_0:=\sum_{i=1}^m e_i$. We call $e_0,-e_1,\dots,-e_m$ \textit{standard directions} of $\mathbb{R}^m$.

For $m\in\mathbb{N}_{>0}$ and $d\in\mathbb{N}$, we define the degree $\Delta_d^m$ to be the multiset consisting of $d$ copies of $e_0$ and $d$ copies of each $-e_i$ for $i=1,\dots,m$.

Let $\alpha:=(\alpha_i)_{i\in\mathbb{N}_{>0}}$ and $\beta:=(\beta_i)_{i\in\mathbb{N}_{>0}}$ be two sequences with $\alpha_i,\beta_i\in\mathbb{N}_{>0}$ such that $|\alpha|:=\sum_{i\in\mathbb{N}_{>0}}\alpha_i$ and $|\beta|:=\sum_{i\in\mathbb{N}_{>0}}\beta_i$ are finite.
Let $d\in\mathbb{N}$ such that $d-\sum_{i\in\mathbb{N}_{>0}}i\cdot\alpha_i+\sum_{i\in\mathbb{N}_{>0}}i\cdot\beta_i=0$ and define
\begin{align*}
\Delta_d^m(\alpha,\beta):= \Delta_d^m\backslash \underbrace{\lbrace -e_m,\dots,-e_m \rbrace}_{d \textrm{ many}} \cup
\bigcup_{i\in\mathbb{N}_{>0}} \underbrace{\lbrace i\cdot (-e_m),\dots, i\cdot (-e_m) \rbrace}_{\alpha_i \textrm{ many}} \cup
\bigcup_{i\in\mathbb{N}_{>0}} \underbrace{\lbrace i\cdot e_m,\dots, i\cdot e_m \rbrace}_{\beta_i \textrm{ many}},
\end{align*}
where unions are actually unions of multisets.
\end{notation}

\begin{definition}[Moduli space of rational tropical stable maps to $\mathbb{R}^m$]\label{definition:moduli_stable_maps}
Let $(\Delta,l)$ be a degree in $\mathbb{R}^m$ as in Definition \ref{def:degree} and let $n\in\mathbb{N}$. A \textit{rational tropical stable map of degree $(\Delta,l)$ to $\mathbb{R}^m$ with $n$ contracted ends} is a tuple $(\Gamma,x_{[N]},h)$, where $(\Gamma,x_{[N]})$ is an $N$-marked abstract tropical curve with $N=n+\#\Delta$, $x_{[N]}=[N]$ and a map $h:\Gamma\to\mathbb{R}^m$ that satisfies the following:
\begin{itemize}
\item[(a)]
Let $e\in\Gamma$ be an edge with length $l(e)\in [0,\infty]$, identify $e$ with $[0,l(e)]$ and denote the vertex of $e$ that is identified with $0\in [0,l(e)]=e$ by $V$. The map $h$ is integer affine linear, i.e. $h\mid_e:t\mapsto tv+a$ with $a\in\mathbb{R}^m$ and $v(e,V):=v\in\mathbb{Z}^m$, where $v(e,V)$ is called \textit{direction vector of $e$ at $V$} and the \textit{weight} of an edge (denoted by $\omega(e)$) is the $\gcd$ of the entries of $v(e,V)$. The vector $\frac{1}{\omega(e)}\cdot v(e,V)$ is called the \textit{primitive} direction vector of $e$ at $V$. If $e=x_i\in\Gamma$ is an end, then $v(x_i)$ denotes the direction vector of $x_i$ pointing away from its one vertex it is adjacent to.
\item[(b)]
The direction vector $v(x_i)$ of an end labeled with $x_i$ is given by $0\in\mathbb{R}^m$ if $x_i\in [n]$. Otherwise, $n<x_i\leq n+\#\Delta$ and $v(x_i)$ equals the unique $v\in\Delta$ with $l(v)+n=x_i$.
Ends with direction vector zero are called \textit{contracted ends}.
\item[(c)]
The \textit{balancing condition}
\begin{align*}
\sum_{\substack{e\in\Gamma\textrm{ an edge}, \\ V \textrm{ vertex of }e}}v(e,V)=0
\end{align*}
holds for every vertex $V\in\Gamma$.
\end{itemize}
Two rational tropical stable maps of degree $d$ with $n$ contracted ends, namely $(\Gamma , x_{[N]},h)$ and $(\Gamma' ,x'_{[N]},h')$, are isomorphic if there is an isomorphism $\varphi$ of their underlying $N$-marked tropical curves such that $h'\circ\varphi=h$.
The set $\mathcal{M}_{0,n}\left(\mathbb{R}^m,\Delta\right)$ of all (rational) tropical stable maps of degree $\Delta$ to $\mathbb{R}^m$ with $n$ contracted ends up to isomorphism is called \textit{moduli space of (rational) tropical stable maps of degree $\Delta$ to $\mathbb{R}^m$ (with $n$ contracted ends)}.
\end{definition}

\begin{notation}\label{notation:projection_moduli_spaces}
See Notation \ref{notation:standard_directions_and_alpha_beta_degrees} for the following: The projection
\begin{align*}
\pi:\mathbb{R}^m\to\mathbb{R}^{m-1}, (x_1,\dots,x_m)\mapsto (x_1,\dots,x_{m-1})
\end{align*}
induces a map 
\begin{align*}
\tilde{\pi}:\mathcal{M}_{0,1}\left(\mathbb{R}^m,\Delta_d^m(\alpha,\beta) \right)\to \mathcal{M}_{0,1+|\alpha|+|\beta|}\left(\mathbb{R}^{m-1},\Delta_d^{m-1} \right),
\end{align*}
where ends in $\Delta_d^m(\alpha,\beta)\backslash \left(\Delta_d^m\backslash \lbrace -e_m,\dots,-e_m \rbrace\right)$ are contracted. So $\tilde{\pi}$ induces labels on contracted ends by contracting labeled ends of direction parallel to $\pm e_m$. To emphasize how non-contracted ends are labeled, we write $\mathcal{M}_{0,1+|\alpha|+|\beta|}\left(\mathbb{R}^{m-1},\pi\left( \Delta^{m}_d(\alpha,\beta) \right) \right)$ instead of $\mathcal{M}_{0,1+|\alpha|+|\beta|}\left(\mathbb{R}^{m-1},\Delta_d^{m-1} \right)$. Moreover, we write $\alpha^{\la}$ (resp. $\beta^{\la}$) to refer to the set of labels associated to ends in $\alpha$ (resp. $\beta$).
\end{notation}

\begin{remark}[$\mathcal{M}_{0,n}\left(\mathbb{R}^m,\Delta\right)$ is a fan]\label{remark:identification:stable_maps_abstract_maps}
The map
\begin{align*}
\mathcal{M}_{0,n}\left(\mathbb{R}^m,\Delta\right) &\to \mathcal{M}_{0,N}\times\mathbb{R}^m\\
(\Gamma,x_{[N]},h) &\mapsto \left(\left(\Gamma,x_{[N]}\right),h(x_1)\right)
\end{align*}
with $N=n+\#\Delta$ is bijective and $\mathcal{M}_{0,n}\left(\mathbb{R}^m,\Delta\right)$ is a tropical fan of dimension $\#\Delta+n-3+m$. Hence $\mathcal{M}_{0,n}\left(\mathbb{R}^m,\Delta\right)$ represents an affine cycle in a $\mathbb{R}^t$. This allows us to use tropical intersection theory on $\mathcal{M}_{0,n}\left(\mathbb{R}^m,\Delta\right)$.
\end{remark}

\begin{definition}[Evaluation maps]\label{def:ev_map}
For $i\in [n]$ the map
\begin{align*}
\operatorname{ev}_i:\mathcal{M}_{0,n}\left(\mathbb{R}^m,\Delta\right) &\to \mathbb{R}^m\\
(\Gamma,x_{[N]},h) &\mapsto h(x_i)
\end{align*}
is called \textit{$i$-th evaluation map}. Under the identification from Remark \ref{remark:identification:stable_maps_abstract_maps} the $i$-th evaluation map is a morphism of fans $\operatorname{ev}_i:\mathcal{M}_{0,N}\times\mathbb{R}^m \to \mathbb{R}^m$. This allows us to pull-back cycles via the evaluation map.
\end{definition}

\begin{definition}[Forgetful maps]\label{def:ft_map}
For $N\geq4$ the map
\begin{align*}
\operatorname{ft}_{x_{[N-1]}}:\mathcal{M}_{0,N}&\to\mathcal{M}_{0,N-1}\\
(\Gamma,x_{[N]}) &\mapsto (\Gamma',x_{[N-1]}),
\end{align*}
where $\Gamma'$ is the stabilization (straighten $2$-valent vertices) of $\Gamma$ after removing its end marked by $x_N$ is called the $N$-th \textit{forgetful map}. Applied recursively, it can be used to forget several ends with markings in $I^C\subset x_{[N]}$, denoted by $\operatorname{ft}_I$, where $I^C$ is the complement of $I\subset x_{[N]}$. With the identification from Remark \ref{remark:identification:stable_maps_abstract_maps}, and additionally forgetting the map $h$ to the plane, we can also consider 
\begin{align*}
\operatorname{ft}_I:\mathcal{M}_{0,n}\left(\mathbb{R}^m,\Delta\right) &\to\mathcal{M}_{0,|I|}\\
(\Gamma,x_{[N]},h) &\mapsto \operatorname{ft}_I(\Gamma,x_i \mid i\in I).
\end{align*}
\end{definition}
Any forgetful map is a morphism of fans. This allows us to pull-back cycles via the forgetful map.

\begin{definition}[Tropical curves and multi lines]\label{def:tropical_line}
A \textit{tropical curve $C$ of degree $\Delta$} is the abstract $1$-dimensional cycle a rational tropical stable map of degree $\Delta$ gives rise to, i.e. $C$ is an embedded $1$-dimensional polyhedral complex in $\mathbb{R}^m$. A \textit{(tropical) multi line} $L$ is a tropical rational curve in $\mathbb{R}^{m}$ with $m+1$ ends such that the primitive direction of each of this ends is one of the \textit{standard directions} of $\mathbb{R}^m$, see Notation \ref{notation:standard_directions_and_alpha_beta_degrees}. The weight with which an end of $L$ appears is denoted by $\omega(L)$.
\end{definition}

\subsection{Enumerative meaning of tropical intersection products}
As indicated in the last section, tropical intersection theory can be applied to the tropical moduli spaces that are interesting for us.
In the present paper tropical intersection theory provides the overall framework in which we work but the most important aspects from this machinary are the following:

\begin{remark}[Enumerative meaning of our tropical intersection products]\label{remark:enumerative_meaning_intersection_product}
Throughout this paper, we consider intersection products of the form $\varphi_1^*(Z_1)\cdots\varphi_r^*(Z_r)\cdot \mathcal{M}_{0,n}\left(\mathbb{R}^m,\Delta\right)$, where $\varphi_i$ is either an evaluation map $\ev_i$ from Definition \ref{def:ev_map} or a forgetful map $\ft_I$ to $\mathcal{M}_{0,4}$ from Definition \ref{def:ft_map}, and $Z_i$ is a cycle we want to pull-back via $\varphi_i$ for $i\in [r]$.
Notice that $\ev_i$ is a map to $\mathbb{R}^m$ while $\ft_I$ is a map to $\mathcal{M}_{0,4}$.
Using a projection $\tilde{\pi}:\mathcal{M}_{0,4}\to\mathbb{R}$ as in Remark 2.2 of \cite{CR1} and considering $\tilde{\pi}\circ\ft_I$ instead of $\ft_I$ does not affect $\varphi_1^*(Z_1)\cdots\varphi_r^*(Z_r)\cdot \mathcal{M}_{0,n}\left(\mathbb{R}^m,\Delta\right)$ since
\begin{align*}
\left( \tilde{\pi}\circ\ft_I\right)^*(\tilde{Z}_i)&=\ft_I^*\left( \tilde{\pi}^*(\tilde{Z}_i)\right)\\
&=\ft_I^*(Z_i)
\end{align*}
holds for a suitable cycle $\tilde{Z}_i$. Thus all our maps can be treated as maps to $\mathbb{R}^t$ for suitable $t$. Hence Proposition 1.15 of \cite{JohannesIntersectionsonTropModuliSpaces} can be applied, and together with Proposition 1.12 of \cite{JohannesIntersectionsonTropModuliSpaces} and Lemma 2.11 of \cite{CR1} it follows that the support of the intersection product $\varphi_1^*(Z_1)\cdots\varphi_r^*(Z_r)\cdot \mathcal{M}_{0,n}\left(\mathbb{R}^m,\Delta\right)$ equals $\varphi_1^{-1}(Z_1)\cap\dots \cap \varphi_r^{-1}(Z_r)$. Hence this intersection product gains an enumerative meaning if it is $0$-dimensional. More precisely, each point in such an intersection product corresponds to a tropical stable map that satisfies certain conditions that are given by the cycles $Z_i$ for $i\in [r]$.
\end{remark}

The weights of such intersection products $\varphi_1^*(Z_1)\cdots\varphi_r^*(Z_r)\cdot \mathcal{M}_{0,n}\left(\mathbb{R}^m,\Delta\right)$ are discussed within the next section. Before proceeding with the next section, we want to briefly recall the concept of rational equivalence that is then frequently used in this paper.

\begin{remark}[Rational equivalence]\label{remark:facts_about_rational_equivalence}
When considering cycles $Z_i$ as in Remark \ref{remark:enumerative_meaning_intersection_product} that are conditions we impose on tropical stable maps, then we usually want to ensure that a $0$-dimensional cycle $\varphi_1^*(Z_1)\cdots\varphi_r^*(Z_r)\cdot \mathcal{M}_{0,n}\left(\mathbb{R}^m,\Delta\right)$ is independent of the exact positions of the conditions $Z_i$ for $i\in [r]$. This is where \textit{rational equivalence} comes into play. We usually consider cycles like $Z_i$ up to a rational equivalence relation. 
The most important facts about this relation are the following:
\begin{itemize}
\item[(a)]
Two cycles $Z,Z'$ in $\mathbb{R}^m$ that only differ by a translation are rationally equivalent.
\item[(b)]
Pull-backs $\varphi^*(Z),\varphi^*(Z')$ of rationally equivalent cycles $Z,Z'$ are rationally equivalent.
\item[(c)]
The \textit{degree} of a $0$-dimensional intersection product which is defined as the sum of all weights of all points in this intersection product is compatible with rational equivalence, i.e. if two $0$-dimensional intersection products are rationally equivalent, then their degrees are the same.
\end{itemize}
Notice that (a)-(c) allows us to ``move" all conditions we consider slightly without affecting a count of tropical stable maps we are interested in.
\end{remark}

Another fact about rational equivalence is the following:

\begin{remark}[Recession fan]\label{remark:recession_fan}
We use Notation \ref{notation:standard_directions_and_alpha_beta_degrees}. Each plane tropical curve $C$ of degree $\Delta_d^2$ is rationally equivalent to a multi line $L_C$ with weights $\omega(L_C)=d$. Hence pull-backs of $C$ and $L_C$ along the evaluation maps are rationally equivalent. The multi line $L_C$ is also called \textit{recession fan of $C$}.
\end{remark}

\subsection{Cross-ratios and their multiplicities}

\begin{definition}\label{def:trop_CR}
A \textit{(tropical) cross-ratio} $\lambda'$ is an unordered pair of pairs of unordered numbers $\left(\beta_1\beta_2|\beta_3\beta_4\right)$ together with an element in $\mathbb{R}_{>0}$ denoted by $|\lambda'|$, where $\beta_1,\dots,\beta_4$ are labels of pairwise distinct ends of a tropical stable map in $\mathcal{M}_{0,n}\left(\mathbb{R}^m,\Delta \right)$. We say that $C\in\mathcal{M}_{0,n}\left(\mathbb{R}^m,\Delta \right)$ satisfies the cross-ratio constraint $\lambda'$ if $C\in\ft^*_{\lambda'}\left(|\lambda'| \right)\cdot \mathcal{M}_{0,n}\left(\mathbb{R}^m,\Delta \right)$, where $|\lambda'|$ is the canonical local coordinate of the ray $\left(\beta_1\beta_2|\beta_3\beta_4\right)$ in $\mathcal{M}_{0,4}$.

A \textit{degenerated (tropical) cross-ratio} $\lambda$ is defined as a set $\lbrace \beta_1,\dots,\beta_4\rbrace$, where $\beta_1,\dots,\beta_4$ are pairwise distinct labels of ends of tropical stable map in $\mathcal{M}_{0,n}\left(\mathbb{R}^m,\Delta \right)$. We say that $C\in\mathcal{M}_{0,n}\left(\mathbb{R}^m,\Delta \right)$ satisfies the degenerated cross-ratio constraint $\lambda$ if $C\in\ft^*_\lambda\left(0 \right)\cdot \mathcal{M}_{0,n}\left(\mathbb{R}^m,\Delta \right)$. A degenerated cross-ratio arises from a non-degenerated cross-ratio by taking $|\lambda'|\to 0$ (see \cite{CR1} for more details). We refer to $\lambda$ as \textit{degeneration} of $\lambda'$ in this case.
\end{definition}


\begin{definition}\label{def:partial_ev}
Define the linear maps $\partial\ev_k:\mathcal{M}_{0,n}\left(\mathbb{R}^m,\Delta_d^m(\alpha,\beta)\right)\to\mathbb{R}^2$ by $\partial\ev_k:=\pi\circ\ev_k$, where $\pi$ is the projection from Notation \ref{notation:projection_moduli_spaces} and $k\in\alpha^{\la}$ or $k\in\beta^{\la}$.
\end{definition}

We use the maps $\partial\ev_k$ to either pull-back points (usually denoted by $P_f$) in $\mathbb{R}^{m-1}$, or tropical multi lines (usually denoted by $L_k$) in $\mathbb{R}^{m-1}$. If we pull-back conditions with $\partial\ev_k$, we refer to these conditions as \textit{tangency conditions}, where, in particular, we refer to $P_f$ as codimension one tangency condition and to $L_k$ as codimension two tangency condition.
All conditions we are interested in are point conditions, tangency conditions and cross-ratio conditions.

\begin{definition}[General position]\label{def:general_pos}
Let $\Delta^m_d(\alpha,\beta)$ be a degree as in Notation \ref{notation:standard_directions_and_alpha_beta_degrees}. Let $\lambda_{[\tilde{l}]}$ be degenerated tropical cross-ratios for some $\tilde{l}\in\mathbb{N}$, let $\mu'_{[l']}$ be non-degenerated tropical cross-ratios for some $l'\in\mathbb{N}$, and let $p_{[n]}\in\mathbb{R}^m$ be points for some $n\in\mathbb{N}_{>0}$. Let $\underline{\eta}^\gamma\subset\gamma^{\la}$ and $\underline{\kappa}^\gamma\subset\gamma^{\la}$ for $\gamma=\alpha,\beta$ be pairwise disjoint sets of labels. Let $P_{\underline{\eta}^\gamma}\in\mathbb{R}^{m-1}$ be points for $\gamma=\alpha,\beta$ and let $L_{\underline{\kappa}^\gamma}$ be tropical multi lines in $\mathbb{R}^{m-1}$ for $\gamma=\alpha,\beta$ such that
\begin{align}\label{eq:def:gen_pos_dimension}
\#\Delta^m_d(\alpha,\beta)-3+m=(m-1)n+\tilde{l}+l'+(m-1)\cdot\#\underline{\eta}+(m-2)\cdot\#\underline{\kappa}.
\end{align}
holds, we say that these conditions are in \textit{general position} if
\begin{align*}
&Z_{\Delta_d^m(\alpha,\beta)}\left(p_{[n]},L_{\underline{\kappa}^\alpha},L_{\underline{\kappa}^\beta},P_{\underline{\eta}^\alpha},P_{\underline{\eta}^\beta},\lambda_{[\tilde{l}]},\mu'_{[l']} \right)
:=\\
&\prod_{k\in \underline{\kappa}^\alpha \cup \underline{\kappa}^\beta} \partial\ev_k^{*}(L_k)
\cdot
\prod_{f\in \underline{\eta}^\alpha\cup \underline{\eta}^\beta} \partial\ev_f^{*}(P_f)
\cdot
\prod_{i=1}^n \ev_i^*\left( p_i\right)
\cdot
\prod_{j'=1}^{l'} \ft_{\mu_{j'}}^*\left( |\mu'_{j'}|\right)
\cdot
\prod_{\tilde{j}=1}^{\tilde{l}} \ft_{\lambda_{\tilde{j}}}^*\left( 0\right)
\cdot
\mathcal{M}_{0,n}\left(\mathbb{R}^m,\Delta_d^m(\alpha,\beta)\right)
\end{align*}
is a zero-dimensional nonzero cycle that lies inside top-dimensional cells of
\begin{align*}
X:=\prod_{\tilde{j}=1}^{\tilde{l}} \ft_{\lambda_{\tilde{j}}}^*\left( 0\right)\cdot\mathcal{M}_{0,n}\left(\mathbb{R}^m,\Delta_d^m(\alpha,\beta)\right).
\end{align*}
\end{definition}

Roughly speaking, $n$ point conditions, $\kappa$ tangency conditions which arise as pull-backs of tropical curves, $\eta$ tangency conditions which arise as pull-backs of points and $l$ (degenerated) tropical cross-ratio conditions are in general position if there are finitely many tropical stable maps of degree $\Delta$ with $c$ contracted ends in $\mathbb{R}^m$ satisfying them and
\begin{align}\label{eq:general_dimension_count_arbitrary_contracted_ends}
\#\Delta+c-3+m=mn+l+(m-1)\eta+(m-2)\kappa
\end{align}
holds.
Often, there are precisely as many contracted ends as point conditions such that \eqref{eq:general_dimension_count_arbitrary_contracted_ends} becomes
\begin{align}\label{eq:general_dimension_count}
\#\Delta-3+m=(m-1)n+l+(m-1)\eta+(m-2)\kappa,
\end{align}
which is exactly \eqref{eq:def:gen_pos_dimension}.

\begin{notation}\label{notation:convention_conditions_related_to_ends}
Notice that in Definition \ref{def:general_pos} a convention is used to which we stick from now on: Given a degree $\Delta^m_d(\alpha,\beta)$ and general positioned conditions, we know which conditions we expect to be satisfied by which labeled ends as we use the same index for conditions and $\ev$ (resp. $\partial\ev$) maps. In particular, we may e.g. consider a submultiset of $\Delta^m_d(\alpha,\beta)$ which contains all ends satisfying the tangency conditions $L_{\underline{\kappa}^\alpha}\cup L_{\underline{\kappa}^\beta}$.
\end{notation}

\begin{remark}
Given an intersection product as in Definition \ref{def:general_pos}, where $L_k$ is a rational tropical curve in $\mathbb{R}^{m-1}$ whose ends are of standard direction, we can pass to the recession fan $\rec(L_k)$ of $L_k$ and obtain an intersection product that is rationally equivalent to the one we started with \cite{Allermann}. Therefore we can always assume that $L_k$ is in fact a tropical multi line in $\mathbb{R}^{m-1}$, see Definition \ref{def:tropical_line}.
\end{remark}

We assume in the following that all conditions are in general position, and if we refer to a set of conditions to be in general condition although this set has not enough elements, then we mean that there are some conditions that we can add to this set such that all together these conditions are in general position.

\subsection{Correspondence Theorem and previous results}

\begin{definition}\label{def:numbers_of_interest}
We use Notation \ref{notation:underlined_symbols}, \ref{notation:standard_directions_and_alpha_beta_degrees}. For general positioned condition as in Definition \ref{def:general_pos}, where we additionally require from the tropical cross-ratios $\lambda_{[\tilde{l}]},\mu'_{[l']}$ that each entry of a tropical cross-ratio is a label of a contracted end or a label of an end whose primitive direction is $\pm e_m\in\mathbb{R}^m$. We define
\begin{align*}
N_{\Delta_d^m(\alpha,\beta)}\left(p_{[n]},L_{\underline{\kappa}^\alpha},L_{\underline{\kappa}^\beta},P_{\underline{\eta}^\alpha},P_{\underline{\eta}^\beta},\lambda_{[\tilde{l}]},\mu'_{[l']} \right)
:=
\deg\left( Z_{\Delta_d^m(\alpha,\beta)}\left(p_{[n]},L_{\underline{\kappa}^\alpha},L_{\underline{\kappa}^\beta},P_{\underline{\eta}^\alpha},P_{\underline{\eta}^\beta},\lambda_{[\tilde{l}]},\mu'_{[l']} \right) \right),
\end{align*}
where $\deg$ is the degree function that sums up all multiplicites of the points in the intersection product, see Definition \ref{definition:degree_map}. In other words, $N_{\Delta_d^m(\alpha,\beta)}\left(p_{[n]},L_{\underline{\kappa}^\alpha},L_{\underline{\kappa}^\beta},P_{\underline{\eta}^\alpha},P_{\underline{\eta}^\beta},\lambda_{[\tilde{l}]},\mu'_{[l']} \right)$ is the number of rational tropical stable maps to $\mathbb{R}^m$ (counted with multiplicity) of degree $\Delta_d^m(\alpha,\beta)$ satisfying the tropical cross-ratios $\lambda_{[\tilde{l}]},\mu'_{[l']}$, the tangency conditions $L_{\underline{\kappa}^\alpha},L_{\underline{\kappa}^\beta},P_{\underline{\eta}^\alpha},P_{\underline{\eta}^\beta}$ and point conditions $p_{\underline{n}}$. If we write $N_{\Delta_d^m(\alpha,\beta)}\left(p_{[n]},\lambda_{[\tilde{l}]},\mu'_{[l']}\right)$, we mean that there are no tangency conditions in the set of given conditions.
\end{definition}

\begin{remark}
The numbers of Definition \ref{def:numbers_of_interest} are independent of the exact positions of points $p_{[n]}$, points $P_{\underline{\eta}^\alpha},P_{\underline{\eta}^\beta}$ and lines $L_{\underline{\kappa}^\alpha},L_{\underline{\kappa}^\beta}$ as long as all conditions are in general position. Moreover, 
they are independent of the exact nonzero lengths of the non-degenerated tropical cross-ratios $\mu'_{[l']}$.
\end{remark}

\begin{theorem}[Correspondence Theorem 5.1 of \cite{IlyaCRC}]\label{thm:correspondence_thm}
Let $\Delta_d^m(\alpha,\beta)$ be a degree as in Notation \ref{notation:standard_directions_and_alpha_beta_degrees}. Consider rational algebraic curves in the toric variety associated to the fan $\Delta_d^m(\alpha,\beta)$ as in \cite{BrugalleMarkwig2016}.
Let $N^{\operatorname{alg}}_{\Delta_d^m(\alpha,\beta)}\left(p_{[n]},\mu_{[l]}\right)$ denote the number of those curves that additionally satisfy point conditions $p_{[n]}$ and non-tropical cross-ratios $\mu_{[l]}$ such that all conditions are in general position. Then
\begin{align*}
N_{\Delta_d^m(\alpha,\beta)}\left(p_{[n]},\lambda'_{[l]}\right)
=
N^{\operatorname{alg}}_{\Delta_d^m(\alpha,\beta)}\left(p_{[n]},\mu_{[l]}\right)
\end{align*}
holds, where $\lambda'_j$ is the tropical cross-ratio associated to $\mu_j$ for $j\in [l]$ in the sense of \cite{IlyaCRC}.
\end{theorem}

Given a tropical stable map $C$ that satisfies a tropical cross-ratio condition $\lambda'$, we can think of this condition as a path of fixed length inside this stable map. Thus a degenerated tropical cross-ratio condition $\lambda$ can be thought of as a path of length zero inside a tropical tropical stable map, i.e. there is a vertex of valence $>3$ in a stable map satisfying a degenerated tropical cross-ratio. Or in other words, there is a vertex $v\in C$ such that the image of $v$ under $\ft_{\lambda}$ is $4$-valent. We say that $\lambda$ is \textit{satisfied at $v$}. It is obvious that \textit{a tropical stable map $C$ satisfies a degenerated tropical cross-ratio condition if and only if there is a vertex of $C$ that satisfies the degenerated tropical cross-ratio}. We define the set $\lambda_v$ of tropical cross-ratios associated to a vertex $v$ that consists of all given tropical cross-ratios whose images of $v$ using the forgetful map are $4$-valent.

\begin{remark}\label{remark:path_criterion}
An equivalent and more descriptive way of saying that a tropical cross-ratio is satisfied at a vertex is the \textit{path criterion}:
Let $C$ be a rational tropical tropical stable map and let $\lambda=\lbrace \beta_1,\dots,\beta_4\rbrace$ be a tropical cross-ratio, then a pair $\left( \beta_i,\beta_j\right)$ induces a unique path in $C$. If the paths associated to $\left( \beta_{i_1},\beta_{i_2} \right)$ and $\left( \beta_{i_3},\beta_{i_4} \right)$ intersect in exactly one vertex $v$ of $C$ for all pairwise different choices of $i_1,\dots,i_4$ such that $\lbrace i_1,\dots,i_4 \rbrace =\lbrace 1,\dots,4 \rbrace$, then and only then the tropical cross-ratio $\lambda$ is satisfied at $v$. Note that ``for all choices" above is equivalent to ``for one choice".
\end{remark}

Let $v$ be a vertex of a rational abstract tropical curve underlying a rational tropical stable map as before. If $$\val(v)=3+\#\lambda_v$$ holds, then we say that $v$ is \textit{resolved according to $\lambda'_i$} (notation from Definition \ref{def:trop_CR} is used) if we replace $v$ by two vertices $v_1,v_2$ that are connected by a new edge such that
\begin{align*}
\lambda_v=\lbrace\lambda_i\rbrace\cup\lambda_{v_1}\cup\lambda_{v_2}
\end{align*}
is a union of pairwise disjoint sets and
\begin{align*}
\val(v_k)=3+\#\lambda_{v_k}
\end{align*}
holds for $k=1,2$.

Resolutions of vertices come into play when we want to determine the weight of a tropical stable map that contributes to $N_{\Delta_d^m(\alpha,\beta)}\left(p_{[n]},L_{\underline{\kappa}^\alpha},L_{\underline{\kappa}^\beta},P_{\underline{\eta}^\alpha},P_{\underline{\eta}^\beta},\lambda_{[l]} \right)$.

\begin{definition}[Cross-ratio multiplicity]\label{def:CR_mult}
We use notation from Definition \ref{def:trop_CR}. Let $v$ be a vertex of a rational abstract tropical curve underlying a tropical stable map with
\begin{align*}
\lambda_v=\lbrace \lambda_{i_1},\dots,\lambda_{i_r}\rbrace \quad \textrm{and} \quad \val(v)=3+r.
\end{align*}
Let $|\lambda'_{i_1}|>\dots >|\lambda'_{i_r}|$ be a total order. A \textit{total resolution} of $v$ is a $3$-valent labeled rational abstract tropical curve on $r$ vertices that arises from $v$ by resolving $v$ according to the following recursion. First, resolve $v$ according to $\lambda'_{i_1}$. The two new vertices are denoted by $v_1,v_2$. Choose $v_j$ with $\lambda_{i_2}\in\lambda_{v_j}$ and resolve it according to $\lambda_{i_2}'$ (this may not be unique, pick one resolution). Now we have $3$ vertices $v_1,v_2,v_3$ from which we pick the one with $\lambda_{i_3}\in\lambda_{v_j}$, resolve it and so on. We define the \textit{cross-ratio multiplicity} $\mult_{\CR}(v)$ of $v$ to be the number of total resolution of $v$.
\end{definition}

\begin{remark}
The cycle $X$ from Definition \ref{def:general_pos} was computed in \cite{CR1}, where it turned out that the weights of the top-dimensional cells of $X$ are precisely the cross-ratio multiplicities from Definition \ref{def:CR_mult}. In particular, $\mult_{\CR}(v)$ does not depend on the total order $|\lambda'_{i_1}|>\dots >|\lambda'_{i_r}|$.
\end{remark}

\begin{example}
Let $v$ be a $6$-valent vertex such that $\lambda_v=\lbrace \lambda_1,\lambda_2,\lambda_3\rbrace$ and the degenerated tropical cross-ratios are given by $\lambda'_1:=(12|56),\lambda'_2:=(34|56),\lambda'_3=(12|34)$. The following two $3$-valent trees are all the total resolutions of $v$ with respect to $|\lambda'_1|>|\lambda'_2|>|\lambda'_3|$.
\begin{figure}[H]
\centering
\def\svgwidth{300pt}
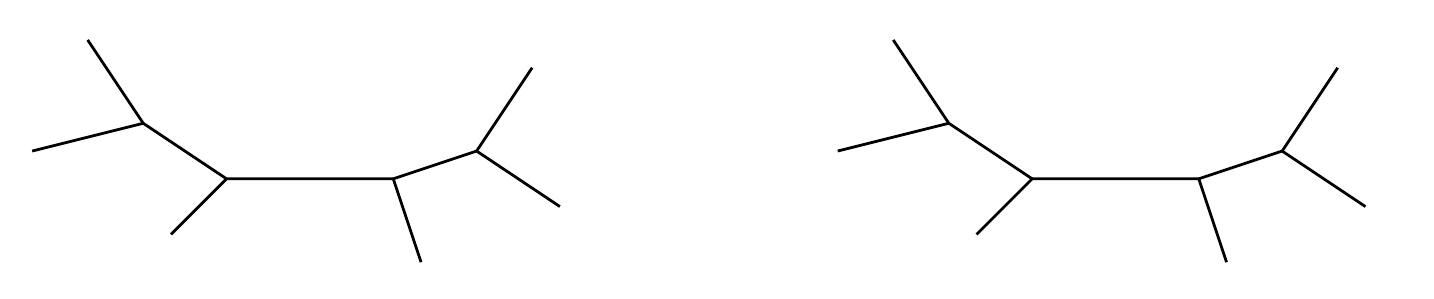
\label{Figure10}
\end{figure}
\end{example}

\begin{definition}(ev-matrix)\label{def:ev_matrix}
The contribution of the evaluation maps $\ev_i$ and $\partial\ev_k$ to the intersection theoretic multiplicity of a tropical cycle in $\mathcal{M}_{0,n}\left(\mathbb{R}^m,\Delta_d^m(\alpha,\beta)\right)$ can be calculated (see Lemma 1.2.9 of \cite{Rau}) by the absolute value of the index of the \textit{ev-matrix} which is given by the (locally around $C$) linear maps $\ev_i$ and $\partial\ev_j$ (for all possible $i,j$), where the coordinates on $X$ (see Definition \ref{def:general_pos}) are the bounded edges' lengths. In the special case that the cycle is $0$-dimensional, i.e. if $C$ contributes to the number $N_{\Delta_d^m(\alpha,\beta)}\left(p_{[n]},L_{\underline{\kappa}^\alpha},L_{\underline{\kappa}^\beta},P_{\underline{\eta}^\alpha},P_{\underline{\eta}^\beta},\lambda_{[l]} \right)$, then the contribution of the evaluation maps $\ev_i$ and $\partial\ev_k$ to the intersection theoretic multiplicity of $C$ can be calculated via the absolute value of the determinant of the $\ev$-matrix of $C$.
If $M(C)$ is the $\ev$-matrix of $C$ as above, then we define
\begin{align*}
\mult_{\ev}(C):=|\det\left(M(C)\right)|
\end{align*}
and refer to $\mult_{\ev}(C)$ as $\ev$-multiplicity of $C$.
\end{definition}

\begin{example}\label{ex:ev-mult}
Consider the tropical stable maps $C$ whose image in $\mathbb{R}^3$ is shown in Figure \ref{Figure28}.
The ends of $C$ are labeled by $1,\dots,6$. The labels are indicated with circled numbers in Figure \ref{Figure28}. The direction vectors of edges and ends of $C$ are shown in Figure \ref{Figure28}. Moreover, the lengths of the three bounded edges of $C$ are denoted by $l_1,l_2,l_3$. The end labeled with $1$ which is drawn dotted indicates a contracted end. The degree of $C$ is $\Delta_{1}^3\left( \alpha, \beta\right)$, where $\alpha=(0,1,0,\dots)$ and $\beta=(1,0,\dots)$ (see Notation \ref{notation:standard_directions_and_alpha_beta_degrees}), i.e. $C$ has one end of primitive direction $-e_3$ whose weight is $2$ and $C$ has one end of primitive direction $e_3$ whose weight is $1$.

The tropical stable map $C$ satisfies the following conditions by which it is fixed:
$p_1$ is a point condition to which the end labeled with $1$ is contracted. The end labeled with $3$ satisfies a codimension two tangency condition $L_3$, where $L_3$ is a multi line with ends of weight $1$ which is indicated by a dashed line in Figure \ref{Figure28}. Moreover, the end labeled with $6$ satisfies a codimension one tangency condition $P_6$. Notice that Notation \ref{notation:convention_conditions_related_to_ends} was used.
\begin{figure}[]
\centering
\def\svgwidth{300pt}
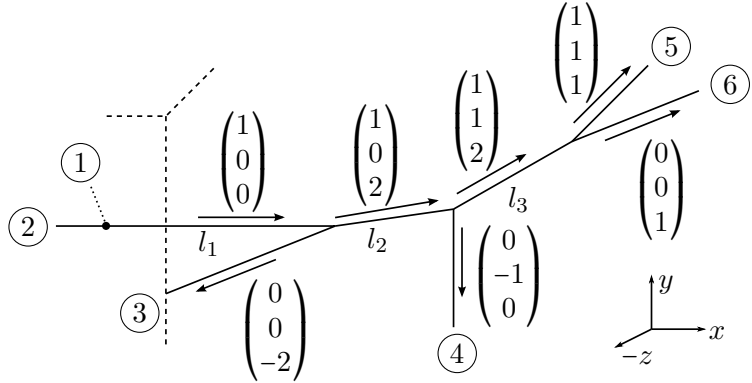
\caption{The tropical stable map $C$ that is fixed by a point $p_1$ and two tangency conditions $L_3,P_6$. The arrows and vectors indicate the directions of the edges.}
\label{Figure28}
\end{figure}

Then the ev-matrix $M(C)$ with respect to the base point $p_1$ of $C$ reads as
\begin{align*}
M(C)=
\begin{array}{c c cccccc c}
 && \multicolumn{3}{c}{\footnotesize \textrm{Base $p_1$}} & \footnotesize \textrm{$l_1$} & \footnotesize \textrm{$l_2$} & \footnotesize \textrm{$l_3$} &\\
\footnotesize \textrm{$\ev_1$} &\ldelim({6}{0.5em}& 1 & 0 & 0 & 0 & 0 & 0 & \rdelim){6}{0.5em} \\
&& 0 & 1 & 0 & 0 & 0 & 0 &\\
&& 0 & 0 & 1 & 0 & 0 & 0 &\\
\footnotesize \textrm{$\partial\ev_3$} && 1 & 0 & 0 & 1 & 0 & 0 & \\
\footnotesize \textrm{$\partial\ev_6$}&& 1 & 0 & 0 & 1 & 1 & 1 &\\
&& 0 & 1 & 0 & 0 & 0 & 1 &\\
\end{array}.
\end{align*}
The first $3$ rows describe the position of $p_1$. The fourth row describes the position of $L_3$ and the last two rows describe the position of $P_6$ using the coordinates $l_1,l_2,l_3$.
\end{example}

\begin{remark}
In case of tropical curves in $\mathbb{R}^2$, the $\ev$-multiplicity splits into a product of local vertex multiplicities. This property of the $\ev$-multiplicity does not hold for tropical space curves, see Example \ref{ex:ev-mult}.
\end{remark}

\begin{proposition}\label{prop:zsfssg_int_theory}
We use Notation \ref{notation:underlined_symbols}. Counting rational tropical stable maps satisfying degenerated tropical cross-ratios yields the same numbers as counting rational tropical stable maps satisfying non-degenerated ones, i.e.
\begin{align*}
N_{\Delta_d^m(\alpha,\beta)}\left(p_{[n]},L_{\underline{\kappa}^\alpha},L_{\underline{\kappa}^\beta},P_{\underline{\eta}^\alpha},P_{\underline{\eta}^\beta},\lambda'_{[l]} \right)
=
N_{\Delta_d^m(\alpha,\beta)}\left(p_{[n]},L_{\underline{\kappa}^\alpha},L_{\underline{\kappa}^\beta},P_{\underline{\eta}^\alpha},P_{\underline{\eta}^\beta},\lambda_{[l]} \right).
\end{align*}
If $C$ is a tropical stable map that contributes to $N_{\Delta_d^m(\alpha,\beta)}\left(p_{[n]},L_{\underline{\kappa}^\alpha},L_{\underline{\kappa}^\beta},P_{\underline{\eta}^\alpha},P_{\underline{\eta}^\beta},\lambda_{[l]} \right)$, then the multiplicity $\mult(C)$ with which $C$ contributes to this intersection product is given by
\begin{align*}
\mult(C)=\mult_{\ev}(C)\prod_{v\mid v \textrm{ vertex of }C}\mult_{\CR}(v).
\end{align*}
\end{proposition}

\begin{proof}
This follows immediately from \cite{CR1} since the arguments used there do not depend on our tropical curves lying in $\mathbb{R}^2$.
\end{proof}

The following corollary is a consequence of Proposition \ref{prop:zsfssg_int_theory} and was already proved in \cite{CR1}. It is a crucial observation and enables us to state that our tropical curves are floor decomposed later on, which in turn allows us to work with floor diagrams.

\begin{corollary}\label{cor:CR_pfade_ueber_alle_edges_an_vertex}
Let $C$ be a rational tropical stable map such that it contributes to the number $N_{\Delta_d^m(\alpha,\beta)}\left(p_{[n]},L_{\underline{\kappa}^\alpha},L_{\underline{\kappa}^\beta},P_{\underline{\eta}^\alpha},P_{\underline{\eta}^\beta},\lambda_{[l]} \right)$. Let $v\in C$ be a vertex of $C$ such that $\val(v)>3$. Then for every edge $e$ adjacent to $v$ in $C$ there is a $\beta_i$ in some $\lambda_j\in\lambda_v$ such that $e$ is in the shortest path from $v$ to $\beta_i$. 
\end{corollary}

\begin{proof}
Assume that there is a vertex $v$ of $C$ and an edge $e$ of $v$ such that $e$ does not appear in some shortest path to some $\beta_i$ in some $\lambda_j\in\lambda_v$. Then a total resolution of $v$ cannot have $3$-valent vertices only since each $3$-valent vertex arising from resolving a cross-ratio cannot be adjacent to $e$. This is (by Proposition \ref{prop:zsfssg_int_theory}) a contradiction to $C$ contributing to $N_{\Delta_d^m(\alpha,\beta)}\left(p_{[n]},L_{\underline{\kappa}^\alpha},L_{\underline{\kappa}^\beta},P_{\underline{\eta}^\alpha},P_{\underline{\eta}^\beta},\lambda_{[l]} \right)$.
\end{proof}

\section{Condition flows}

In the following \textit{condition flows} on a tropical curve are defined. The motivation is the following: In case of a tropical curve $C$ in $\mathbb{R}^2$ and some set $S$ of general positioned conditions, the following implication holds (see for example \cite{KontsevichPaper}): If $C$ satisfies all given conditions $S$ and $C$ has a \textit{string}, then $C$ is not fixed by the given conditions.
Now that we are in higher dimension (i.e. let $C$ be a tropical curve in $\mathbb{R}^m$), we aim for a generalization, namely the implication: If $C$ satisfies all given conditions $S$ and there is no \textit{conditions flow} of type $m$ on $C$, then $C$ is not fixed by the given conditions. A similar construction has been used in \cite{MandelRuddat2019} to study multiplicities of tropical curves.

\begin{definition}[Leaky]
A graph $G$ together with a function $leak:V(G)\to \mathbb{N}$ that assigns a natural number to each vertex of the graph is called \textit{leaky}.
\end{definition}

\begin{definition}[Flow]\label{def:flow}
Let $G$ be a tree, where we allow \textit{ends}, i.e. edges that are adjacent to a single vertex only without forming a loop.
Each edge $e$ that is adjacent to two vertices consists of two half-edges $e_1,e_2$. If $v$ is a vertex of $G$ that is adjacent to $e$, then we refer to the half-edge $e_i$ ($i=1,2$) of $e$ that is adjacent to $v$ as \textit{outgoing edge of $v$} and to the other half-edge of $e$ as \textit{incoming edge of $v$}. If $e$ is an edge that is adjacent to a single vertex $v$, then $e$ is considered to be an \textit{incoming edge of $v$}. A \textit{flow structure} on $G$ is given by a map $R$ that assigns to each half-edge and to each end of $G$ an element of $\mathbb{N}$. We refer to the image $R(e_i)$ of an end or a half-edge $e_i$ under the flow structure map $R$ as \textit{flow on $e_i$}. Moreover, the \textit{flow of a vertex $v$} is defined by
\begin{align*}
\flow(v):=\sum_{\textrm{$e$ incoming edge of $v$}} R(e).
\end{align*}
\end{definition}

\begin{figure}[H]
\centering
\def\svgwidth{230pt}
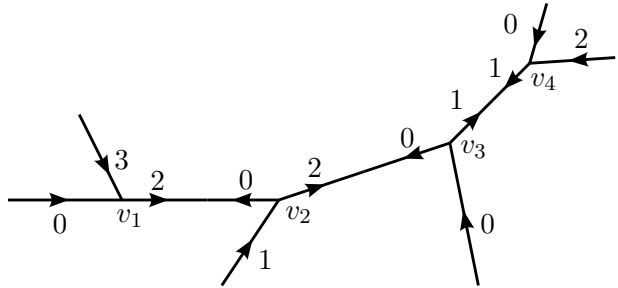
\caption{An example of a flow structure, where each half edge and if it is considered incoming or outgoing is indicated by an arrow (the numbers on the arrows are the numbers associated to the half edges by the flow structure). The vertices are denoted by $v_i$ for $i\in [4]$ and $\flow(v_i)=3$ for $i\in [4]$.}
\label{Figure29}
\end{figure}

\begin{definition}[Condition flow]\label{def:condition_flow}
Let $G$ be a leaky graph with a flow as defined in \ref{def:flow}. The flow structure on $G$ is called \textit{condition flow of type $m$} if it satisfies the following properties:
\begin{enumerate}[label=(P\arabic*),ref=(P\arabic*)]
\item \label{P1}
If $e$ is an edge of $G$ consisting of the two half-edges $e_1,e_2$, then
\begin{align*}
R(e_1)+R(e_2)=m-1.
\end{align*} 
\item \label{P2}
The flow is \textit{balanced} on each vertex, that is
\begin{align*}
\flow(v)-\leak(v)=0
\end{align*}
holds for all vertices $v$ of $G$.
\end{enumerate}
\end{definition}

Figure \ref{Figure29} provides an example of a condition flow of type $3$, where $\leak(v_i)=3$ for $i\in [4]$.

\begin{remark}\label{remark:condition_flows_generalize_ad_hoc_methods}
A condition flow of type $2$ is a flow structure on a graph $G$ such that for each edge consisting of two half-edges there is exactly one half-edge $e_1$ with $R(e_1)=1$ and another half-edge $e_2$ with $R(e_2)=0$.
There are different ways of encoding this condition flow of type $2$ into a graph $G$.
In \cite{BroccoliCurves} orientations on $G$ where used to indicate half-edges $e_1$ with $R(e_1)=1$, and in \cite{RanganathanCavalieriMarkwigJohnson} and \cite{CR1} ``thick'' half-edges were used.
One advantage of condition flows over these ad hoc constructions is that they are applicable in higher dimensions.
\end{remark}

\begin{lemma}\label{lemma:condition_flow_unique}
A condition flow of type $m$ on a tree is uniquely determined by its leak function and the flow on its ends.
\end{lemma}

\begin{proof}
Assume there are two condition flows of type $m$ with the same leak function on a tree $G$. First, note that the leak function determines the flows on the vertices by \ref{P2}. Assume there is at least one half-edge $e_1$ on which the flows differ. Since we assumed that the flows on the ends are equal, there is another half-edge $e_2$ adjacent to $e_1$. Thus the flows also differ on $e_2$ because of \ref{P1}. So there is an edge $e$ of $G$ on which the flows differ. Denote a vertex to which $e$ is adjacent by $v$. If $v$ is only adjacent to one bounded edge, namely $e$, then \ref{P2} yields a contradiction. Hence there is another edge $e'\neq e$ adjacent to $v$ on which the flows differ because of \ref{P2}. Since $G$ is a tree, there is a vertex $v'$ of $G$ which is only adjacent to one bounded edge such that the flows on this edge differ, which leads to the same contradiction as above.
\end{proof}

\begin{construction}\label{constr:condition_flow_to_trop_curve}
Let $G$ be a tree with fixed flows on its ends. We construct a flow structure on $G$ the following way.
Note that we can think of each bounded edge of $G$ as being glued from two half-edges (by cutting it into two halves). Set all flows on all half-edges that are no ends to be zero.
We use the following procedure to spread the flows of the ends to all half-edges:
Choose a vertex $v$ of $G$. Now spread the flows on $G$ according to the following rule.
If a vertex $v$ has $r$ outgoing edges $e_{\out,1},\dots,e_{\out,r}$ and $\tilde{r}$ incoming edges $e_{\inc,1},\dots,e_{\inc,\tilde{r}}$ such that $e_{\inc,i}$ and $e_{\out,i}$ form an edge for $i=1,\dots,r$, then
\begin{align}\label{eq:flow_spread}
R(e_{\out,i})=
\begin{cases}
\flow(v)-R(e_{\inc,i})-1 &\textrm{, if $\flow(v)-R(e_{\inc,i})>0$} \\
0 &\textrm{, if $\flow(v)-R(e_{\inc,i})=0$ and $\flow(v)\neq 0$}\\
\end{cases}
\end{align}
for $i=1,\dots,r$.  
  
Repeat with another vertex $v'$ of $G$. Notice that flows on half-edges can at most increase. Stop when the flows on all edges stay the same. This construction yields a unique flow on $G$.
\end{construction}

\begin{example}\label{ex:constructing_flow_on_tree}
We want to illustrate Construction \ref{constr:condition_flow_to_trop_curve}. Figure \ref{Figure30} provides an example of a tree $G$ on the four vertices $v_i$ for $i\in [4]$. The flows on ends of $G$ are indicated in Figure \ref{Figure30}. Before starting with the procedure, all flows of non-end half-edges are set to be zero as in Figure \ref{Figure30}.

\begin{figure}[H]
\centering
\def\svgwidth{200pt}
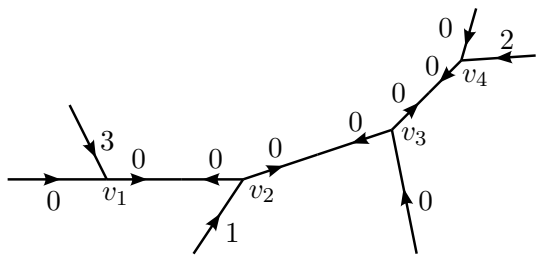
\caption{The graph $G$ to which Construction \ref{constr:condition_flow_to_trop_curve} should be applied.}
\label{Figure30}
\end{figure}

\noindent See Figure \ref{Figure31} for the following: In the first step, Construction \ref{constr:condition_flow_to_trop_curve} is applied to determine the flow on the outgoing half-edges of $v_1$. After that, Construction \ref{constr:condition_flow_to_trop_curve} is applied to determine the flows on the outgoing half-edges of $v_2$. If Construction \ref{constr:condition_flow_to_trop_curve} is then applied to $v_4$ and after that to $v_3$, then the procedure terminates. The fourth step in Figure \ref{Figure31} shows the resulting flow structure on $G$ which is the same as the one in Figure \ref{Figure29}.

\begin{figure}[H]
\centering
\def\svgwidth{380pt}
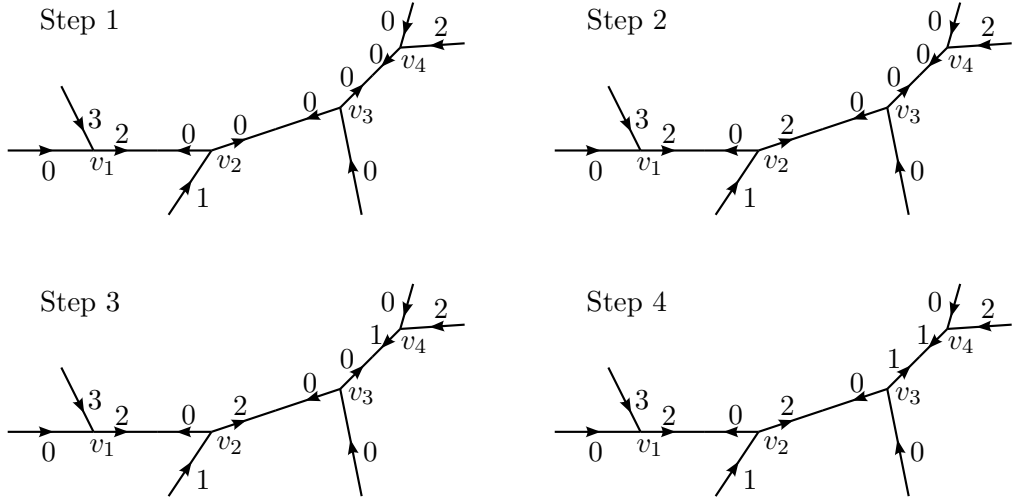
\caption{The progress of Construction \ref{constr:condition_flow_to_trop_curve} applied to $G$ and the condition flow of type $3$ it yields on $G$.}
\label{Figure31}
\end{figure}
\end{example}

\begin{proof}[Proof that Construction \ref{constr:condition_flow_to_trop_curve} terminates uniquely.]
We use induction on the number $N$ of vertices of $G$. If $N=1$, then there is nothing to show. So let $N=2$. Then the procedure of Construction \ref{constr:condition_flow_to_trop_curve} stops uniquely after at most 2 steps. For the induction step notice that $G$ is a tree, i.e. there is a vertex $v$ that is adjacent to exactly one edge $e$ that is no end. The flows on the ends of $G$ are given and
\begin{align*}
R(e_{\out,v})=\begin{cases}
0 & \textrm{if all ends adjacent to $v$ are of flow zero,}\\
\flow(v)-1 & \text{else}
\end{cases}
\end{align*}
is already determined since the maximum of flows assigned to the same half-edge is taken in each step. Let $v'$ be the vertex adjacent to $v$ via $e$. Consider the tree $G'$ that arises from $G$ the following way: forget $e,v$ and all ends adjacent to $v$, then attach a new end $e'$ to $v'$ and assign the flow $0$ (resp. $\flow(v)-1$) to $e'$. Now run the procedure of Contruction \ref{constr:condition_flow_to_trop_curve} on $G'$. By induction, this procedure terminates uniquely. Notice that the missing flow on $e$ associated to the outgoing half-edge $e_{\out,v'}$ of $v'$ is determined by \eqref{eq:flow_spread}. Moreover, the flow on $e_{\out,v'}$ does not affect the other flows on $G$ which the procedure generated on $G'$. Thus Construction \ref{constr:condition_flow_to_trop_curve} terminates uniquely.
\end{proof}

\begin{definition}[Induced flows]\label{def:induced_flows}
Consider a tropical stable map $C$ that contributes to the number $N_{\Delta_d^m(\alpha,\beta)}\left(p_{[n]},L_{\underline{\kappa}^\alpha},L_{\underline{\kappa}^\beta},P_{\underline{\eta}^\alpha},P_{\underline{\eta}^\beta},\lambda_{[l]} \right)$. Associate flows $R(e)$ to ends $e$ of $C$ the following way:
If $e$ is a non-contracted end of $C$ satisfying a codimension two tangency condition $L_{k}$ for some $k\in\underline{\kappa}^\gamma$, $\gamma=\alpha,\beta$, then $R(e):=m-2$. If $e$ satisfies a codimension one tangency condition $P_f$ for some $f\in\underline{\eta}^\gamma$, $\gamma=\alpha,\beta$, then $R(e):=m-1$. If $e$ is a contracted end of $C$ satisfying a point condition, then $R(e):=m$. Otherwise, set $R(e):=0$. We refer to these flows on the ends as \textit{induced flows} from the tangency and point conditions.
\end{definition}

\begin{example}
Let $C$ be the tropical stable map depicted in Figure \ref{Figure28} that contributes to the number $N_{\Delta_1^3\left( (0,1,0,\dots), (1,0,\dots)\right)}\left(p_{1},L_3,P_6\right)$ as in Example \ref{ex:ev-mult}. Step $1$ in Figure \ref{Figure31} shows the flows conditions induce on ends of $C$. Example \ref{ex:constructing_flow_on_tree} shows the flow structure Construction \ref{constr:condition_flow_to_trop_curve} assigns so $C$. Notice that the resulting flow structure is a condition flow of type $3$ and that the constructed flow structure does not depend on the order of the vertices from which the flows were spread.
\end{example}

\begin{proposition}\label{prop:condition_flow_to_trop_curve}
Let $C$ be a tropical stable map that contributes to the enumerative number $N_{\Delta_d^m(\alpha,\beta)}\left(p_{[n]},L_{\underline{\kappa}^\alpha},L_{\underline{\kappa}^\beta},P_{\underline{\eta}^\alpha},P_{\underline{\eta}^\beta},\lambda_{[l]} \right)$ such that flows on its ends are induced from the point and tangency conditions as in Definition \ref{def:induced_flows}.
Then Construction \ref{constr:condition_flow_to_trop_curve} associates the unique condition flow of type $m$ to $C$, where the leak function is given by $\leak(v)=m$ for all vertices $v$ of $C$.
\end{proposition}

\begin{proof}
Given a tropical stable map $C$ contributing to $N_{\Delta_d^m(\alpha,\beta)}\left(p_{[n]},L_{\underline{\kappa}^\alpha},L_{\underline{\kappa}^\beta},P_{\underline{\eta}^\alpha},P_{\underline{\eta}^\beta},\lambda_{[l]} \right)$, we give another interpretation of the flow constructed in \ref{constr:condition_flow_to_trop_curve}, namely in terms of spatial restrictions the vertices of $C$ impose on their neighbors. By restrictions we mean the following: Let $\Gamma$ be the combinatorial type of $C$, i.e. $C$ without its metric structure. Since $C$ fulfills all given conditions, we are able to re-embed $\Gamma$ into $\mathbb{R}^m$, i.e. we are able to reconstruct the lengths of all edges of $C$. To do so, we proceed in the following way: Let $v\in C$ be a vertex adjacent to a contracted end satisfying a point condition $p_i$ for some $i$, then choose $\Gamma\to\mathbb{R}^m$ in such a way that $v\mapsto p_i$. Let $e$ be a bounded edge adjacent to $v$ and some other vertex $v'$. Since $\Gamma$ knows the direction of $e$ in $\mathbb{R}^m$, fixing $v$ imposes an $(m-1)$-dimensional restriction on the position of $v'$. In other word, $v'$ can only move along the direction of $e$. We encode this restriction from $v$ to $v'$ into $\Gamma$ by interpreting $e$ as two glued half-edges $e_1,e_2$, where the half-edge $e_i$ adjacent to $v$ is equipped with a number $R(e_i)=m-1$. We refer to this half-edge as outgoing edge of $v$ or as incoming edge of $v'$. Iteratively, the restrictions spread along $\Gamma$, i.e. let $e'$ be another bounded edge adjacent to $v'$ and some other vertex $v''\neq v$. Since we know the direction of $e'$ in $\mathbb{R}^m$, the $1$-dimensional movement of $v'$ allows $v''$ to only move along two directions. More precisely, we may vary the length of $e$ and the length of $e'$. Said differently, $v'$ imposes at least an $(m-2)$-dimensional restriction on $v''$.

Obviously, we could also have started with a tangency condition, i.e. some other restriction incoming to a vertex via an end.

We claim that the flow structure constructed from restrictions passing from a vertex to another via half-edges fulfills the procedure equation \eqref{eq:flow_spread} describes in Construction \ref{constr:condition_flow_to_trop_curve}.

Denote the equations of \eqref{eq:flow_spread} by I and II, from top to bottom.
\begin{itemize}
\item
Let $v$ be a vertex that is adjacent to another vertex $v'$ via an edge $e$ such that $v$ gains all its spatial restrictions via the incoming half-edge $e_{\inc}$ of $e$. Then $v$ does not impose a spatial restriction to $v'$ via its outgoing half-edge $e_{\out}$ of $e$. Hence II holds. Said differently, a vertex cannot pass spatial directions back to an adjacent vertex from which they came.
\item
I holds since repeating the argument of II yields the summand $R(e_{\inc,i})$ of \eqref{eq:flow_spread}, and as we saw before, passing over a vertex lowers the number of restrictions in general by $1$, where in general means that edges $e,e'$ adjacent to the same vertex $v$ are usually not parallel -- if they are parallel, then there is (because $C$ is fixed by the general positioned conditions) a end adjacent to $v$ that either satisfies a point condition or some tangency condition. Notice that in both of these two special cases I holds.
\end{itemize}
Hence our flow structure on $\Gamma$ defined as restrictions passing from one vertex to another is governed by the same equations as the flow structure assigned to $\Gamma$ by Construction \ref{constr:condition_flow_to_trop_curve}.
Hence these two flow structures on $\Gamma$ coincide. Next, we claim that the flow structure on $\Gamma$ interpreted as spatial restrictions is a condition flow of type $m$, i.e. it satisfies \ref{P1} and \ref{P2} of Definition \ref{def:condition_flow}. Given a bounded edge $e$ of $C$, cut it and stretch it to infinity. Denote the two components of $C$ obtained that way by $C_1,C_2$, where $e_1$ is the end of $C_1$ that used to be $e$ and $e_2$ is the analogous end of $C_2$. We use the following notation: Let $\Delta_i$ be the degree of $C_i$, let $\underline{n_i}\subset [n]$ be the point conditions satisfied by $C_i$, let $\underline{l_i}\subset [l]$ be the degenerated cross-ratios satisfied by $C_i$, let $\underline{\kappa_i}\subset \underline{\kappa}^\alpha\cup\underline{\kappa}^\beta$ be the codimension two tangency conditions satisfied by $C_i$ and let $\underline{\eta_i} \subset \underline{\eta}^\alpha\cup\underline{\eta}^\beta$ be the codimension one tangency conditions satisfied by $C_i$ for $i=1,2$. Then
\begin{align}\label{eq:prop_proof1}
\#\Delta_1+\#\Delta_2-2=\#\Delta_d^m(\alpha,\beta)
\end{align}
and
\begin{align}\label{eq:prop_proof2}
\#\Delta_i-3+m=(m-1)\cdot \#\underline{n_i}+\#\underline{l_i}+(m-1)\cdot \#\underline{\eta_i}+(m-2)\cdot \#\underline{\kappa_i}+R(e_i)
\end{align}
hold for $i=1,2$. Adding \eqref{eq:prop_proof1} and \eqref{eq:prop_proof2}, and applying \eqref{eq:general_dimension_count} yields \ref{P1}. Moreover, \ref{P2} can be satisfied by defining the leak function this way. Then the leak function coincides with the one given in Proposition \ref{prop:condition_flow_to_trop_curve} since all conditions are in general position. Moreover, this condition flow is unique due to Lemma \ref{lemma:condition_flow_unique}.
\end{proof}

Proposition \ref{prop:condition_flow_to_trop_curve} allows us to think about condition flows the way we think about strings in tropical curves in $\mathbb{R}^2$: Proposition \ref{prop:condition_flow_to_trop_curve} is an exclusion criterion for stable maps not contributing to $N_{\Delta_d^m(\alpha,\beta)}\left(p_{[n]},L_{\underline{\kappa}^\alpha},L_{\underline{\kappa}^\beta},P_{\underline{\eta}^\alpha},P_{\underline{\eta}^\beta},\lambda_{[l]} \right)$ on the level of combinatorial types. If $C$ is the combinatorial type of a tropical stable map $C'$ and Construction \ref{constr:condition_flow_to_trop_curve} does not lead to a condition flow of type $m$ with the leak function given in Proposition \ref{prop:condition_flow_to_trop_curve}, then $C'$ cannot contribute to $N_{\Delta_d^m(\alpha,\beta)}\left(p_{[n]},L_{\underline{\kappa}^\alpha},L_{\underline{\kappa}^\beta},P_{\underline{\eta}^\alpha},P_{\underline{\eta}^\beta},\lambda_{[l]} \right)$.

\begin{remark}
Another way to think about flows is the following: each vertex $v$ of a tropical curve in $\mathbb{R}^m$ is a point in $\mathbb{R}^m$, i.e. the minimal number of affine linear equations needed to cut out $v$ is $m$. The flow of $v$ is the number of equations $v$ needs to satisfy. These equations arise from imposing conditions to our tropical stable map as in the proof of Proposition \ref{prop:condition_flow_to_trop_curve}, and these equations are affine linear since tropical stable maps are piecewise linear. Choosing all conditions in general conditions means to choose the minimal number of conditions needed to fix our tropical stable map, i.e. the matrix of affine linear equations associated to each vertex needs to have full rank, or in other words, the flow of each vertex needs to be $m$ if each vertex should be fixed.

If there are not enough conditions to fix a curve, then a parts of the curve are movable. These movable parts are encoded in the flow structure since all vertices with flow less than $m$ are movable.
The special case of one missing condition and one \textit{movable component} for curves in $\mathbb{R}^2$ was studied in \cite{GeneralKontsevich}. We remark here, that we also could have used flows there to describe which parts of a curve are movable.
\end{remark}

\section{Floor decomposition}

From now on we specialize to tropical space curves, i.e. tropical stable maps to $\mathbb{R}^3$.

\subsection{Floor decomposed tropical curves}
Our first aim it to show that we may assume that the tropical stable maps we need to consider are \textit{floor decomposed}, see Proposition \ref{prop:floor_decomposed}. We remark, that Proposition \ref{prop:floor_decomposed} can be generalized to tropical stable maps to $\mathbb{R}^m$.

\begin{definition}[Stretched configuration]\label{def:stretched_config}
Let $\pi:\mathbb{R}^3\to\mathbb{R}^2$ be the natural projection that forgets the last coodinate as in Notation \ref{notation:projection_moduli_spaces}. Let $\epsilon>0$ be a real number and let $B_\epsilon:=(-\epsilon,\epsilon)\times (-\epsilon,\epsilon)$ be a box in $\mathbb{R}^2$.
Let $p_{[n]},L_{\underline{\kappa}^\alpha},L_{\underline{\kappa}^\beta},P_{\underline{\eta}^\alpha},P_{\underline{\eta}^\beta},\lambda_{[l]}$ be general positioned conditions as in Definition \ref{def:general_pos}. These conditions are said to be in \textit{stretched configuration} if:
\begin{itemize}
\item
$\pi\left(P_{\underline{\eta}^\gamma}\right)\subset B_{\epsilon}$ for $\gamma=\alpha,\beta$,
\item
$L_k^{(0)}\in B_{\epsilon}$, where $L_k^{(0)}$ denotes the $0$-skeleton, i.e. the vertex of $L_k$ for $k\in \underline{\kappa}^\alpha\cup\underline{\kappa}^\beta$,
\item
$\pi(p_{[n]})\subset B_\epsilon$ and the distances of the $z$-coordinates $p_{i,z}$ of the points $p_i$ are large compared to the size of the box $B_\epsilon$, i.e.
$|p_{i+1,z}-p_{i,z}|>>\epsilon$ for $i=1,\dots,n-1$.
\end{itemize}
\end{definition}

\begin{remark}
Stretched configurations exist, because the set of all positions of general positioned conditions is dense in the set of positions of all possible conditions, i.e. the property of being in general position can be preserved when stretching the points $p_i$ in $z$-direction.
\end{remark}

\begin{definition}\label{def:floor_decomposed}
An \textit{elevator} of a tropical stable map $C$ of degree $\Delta^3_d\left(\alpha,\beta\right)$ is an edge whose primitive direction is $(0,0,\pm 1)$. A connected component $C_i$ of $C$ that remains if the interiors of the elevators are removed is called \textit{floor} of the curve $C$. The number $s_i\in\mathbb{N}$ of ends of $C_i$ that are of direction $(1,1,1)$ is called the \textit{size} of the floor $C_i$. A tropical stable map that is fixed by general positioned conditions as in Definition \ref{def:general_pos} is called \textit{floor decomposed} if each of the points $p_{[n]}$ lies on its own floor. Notice that floors can be of size zero, i.e. a floor can have exactly one vertex.

Later we equip floors with additional ends by cutting elevators (Construction \ref{constr:cutting_along_elevators}) and stretching them to infinity. By abuse of notation we refer to these tropical stable maps as floors as well when no confusion can occur.
\end{definition}

\begin{example}\label{ex:floor_decomposed_tropical_stable_map}
Figure \ref{Figure32} shows a floor decomposed tropical stable map $C$. The labels of some of its ends are indicated with circled numbers. The ends labeled with $1$ and $2$ are drawn dotted which indicates that these ends are contracted. The other labeled ends are of primitive direction $\pm e_3\in\mathbb{R}^3$ using Notation \ref{notation:standard_directions_and_alpha_beta_degrees}. The end labeled with $8$ is of weight two while all other ends are of weight one such that the degree of $C$ is $\Delta^3_4\left( (4,1,0,\dots), (2,0,\dots) \right)$, see Notation \ref{notation:standard_directions_and_alpha_beta_degrees}. The general positioned conditions $C$ satisfies are the following: The end labeled with $1$ (resp. $2$) satisfies a point condition $p_1$ (resp. $p_2$). The ends labeled with $i\in [9]\backslash [2]$ satisfy codimension one tangency conditions $P_i$ for $i\in [9]\backslash [2]$. Moreover, $C$ satisfies the degenerated tropical cross-ratio $\lambda_1=\lbrace 1,2,3,7 \rbrace$ at its only $4$-valent vertex.

The elevator of $C$ has weight two and is drawn dashed. Thus $C$ has two floors $C_i$ for $i=1,2$, where the point $p_i$ lies on $C_i$ for $i=1,2$.

\begin{figure}[H]
\centering
\def\svgwidth{450pt}
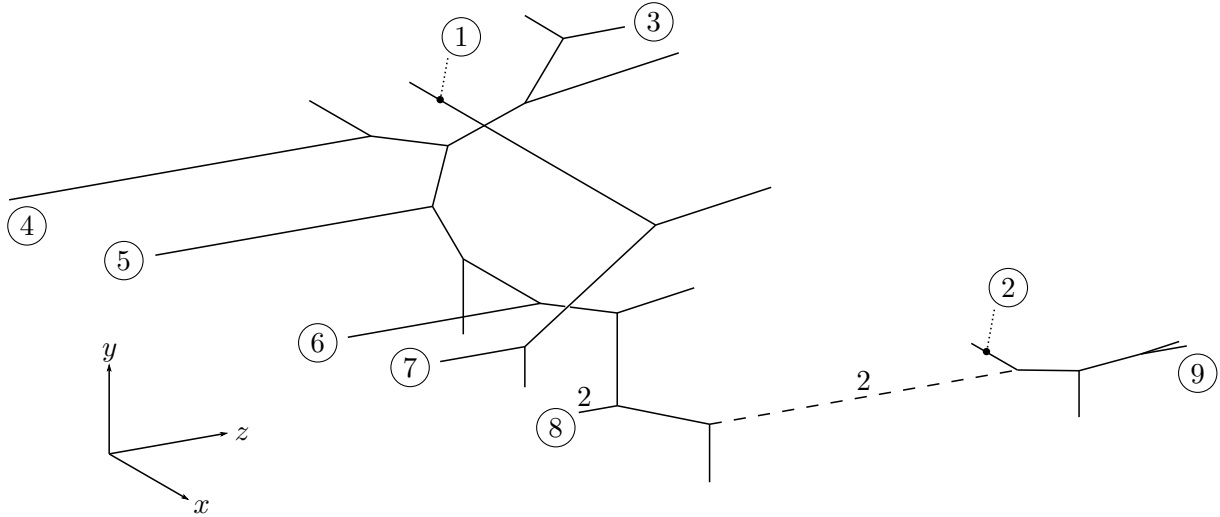
\caption{The tropical stable map $C$ from Example \ref{ex:floor_decomposed_tropical_stable_map} which is floor decomposed. It has two floors $C_i$ for $i=1,2$, where $C_1$ is of size $s_1=3$ and $C_2$ is of size $s_2=1$. The dashed edge is the elevator of weight two of $C$.}
\label{Figure32}
\end{figure}
\end{example}

\begin{proposition}\label{prop:floor_decomposed}
Let $p_{[n]},L_{\underline{\kappa}^\alpha},L_{\underline{\kappa}^\beta},P_{\underline{\eta}^\alpha},P_{\underline{\eta}^\beta},\lambda_{[l]}$ be conditions in a stretched configuration as in Definition \ref{def:stretched_config} such that each entry of each degenerated cross-ratio is a label of a contracted end or a label of an end whose primitive direction is $(0,0,\pm 1)\in\mathbb{R}^3$. Then every tropical stable map contributing to
$N_{\Delta_d^3(\alpha,\beta)}\left(p_{[n]},L_{\underline{\kappa}^\alpha},L_{\underline{\kappa}^\beta},P_{\underline{\eta}^\alpha},P_{\underline{\eta}^\beta},\lambda_{[l]} \right)$
is floor decomposed.
\end{proposition}

\begin{proof}
We follow arguments used in \cite{unpublishedFloorDiagrams,Torchiani}, where an analogous statement is proved for the case without tropical cross-ratios. To incorporate tropical cross-ratios, we use Corollary \ref{cor:CR_pfade_ueber_alle_edges_an_vertex} as we did in \cite{CR1}.

Let $C$ be a tropical stable map contributing to $N_{\Delta_d^3(\alpha,\beta)}\left(p_{[n]},L_{\underline{\kappa}^\alpha},L_{\underline{\kappa}^\beta},P_{\underline{\eta}^\alpha},P_{\underline{\eta}^\beta},\lambda_{[l]} \right)$.
The set of all possible bounded edges' directions is finite because of the balancing condition and the fixed directions of ends. If $\epsilon$ from Definition \ref{def:stretched_config} is sufficiently small compared to the distances between the points $p_{[n]}$ and all vertices of $C$ lie inside the box $B_\epsilon\times\mathbb{R}$, then $C$ decomposes into parts that are connected by horizontal edges. So it is sufficient to show that all vertices of $C$ lie inside $B_\epsilon\times\mathbb{R}$ from Definition \ref{def:stretched_config}.

Assume $v\in C$ is a vertex whose $x$-coordinate is maximal and $v$ lies outside of $B_\epsilon\times\mathbb{R}$. Since the $x$-coordinate of $v$ is maximal, there is an end $e$ of direction $(1,1,1)$ adjacent to $v$. If $v$ is not $3$-valent, then there is a $j$ such that $\lambda_j\in\lambda_v$ and the label of $e$ appears as an entry in $\lambda_j$ because of Corollary \ref{cor:CR_pfade_ueber_alle_edges_an_vertex}. Due to our assumptions on the tropical cross-ratios $\lambda_{[l]}$, the end $e$ cannot be an entry of any of these, which is a contradiction. Hence $v$ must be $3$-valent. Denote the edges adjacent to $v$ by $e_1,e_2,e$, where $e$ is, as before, an end of direction $(1,1,1)$. If $e_1$ is an end parallel to $(0,0,-1)$, then $v$ allows a $1$-dimensional movement in the direction of $e_2$, since $e_1$ either satisfies no condition or satisfies a codimension two tangency condition $L_k$ for some $k$ with $\pi(e_2),\pi(e)\subset L_k$, where $\pi$ is the natural projection that forgets the $z$-coordinate of $\mathbb{R}^3$. Thus $e_1,e_2$ are bounded edges. Since $e$ is an end of direction $(1,1,1)$ and thus of weight $1$, and $v$ is maximal with respect to its $x$-coordinate, it follows (without loss of generality) that the $x$-coordinate of the direction vector of $e_1$ is $0$ and the $x$-coordinate of the direction vector of $e_2$ is $-1$. Denote the vertex adjacent to $v$ via $e_1$ by $v'$. Notice that $v'$ is also $3$-valent, adjacent to an end $e'$ parallel to $e$ and a bounded edge $\tilde{e}\neq e_1$. By balancing, $e_2,e,e_1,e',\tilde{e}$ lie in the affine hyperplane $\langle e_1,e\rangle+v$ of $\mathbb{R}^3$. Thus $v$ allows a $1$-dimensional movement in the direction of $e_2$ which is a contradiction .

Notice that similar arguments hold if $v$ is chosen in such a way that its $y$-coordinate is maximal or its $x$-coordinate (resp. $y$-coordinate) is minimal. So in any case a $1$-dimensional movement leads to a contradiction. Hence all vertices of $C$ lie inside the box $B_\epsilon\times\mathbb{R}$. Therefore $C$ is floor decomposed.
\end{proof}

\begin{notation}\label{notation:1/1_edge_and_2/0_edge}
Whenever we refer to the condition flow of $C$, where $C$ is a floor decomposed tropical stable map contributing to $N_{\Delta_d^3(\alpha,\beta)}\left(p_{[n]},L_{\underline{\kappa}^\alpha},L_{\underline{\kappa}^\beta},P_{\underline{\eta}^\alpha},P_{\underline{\eta}^\beta},\lambda_{[l]} \right)$, we mean that $C$ is equipped with the condition flow of type $3$ associated to $C$ using Construction \ref{constr:condition_flow_to_trop_curve}. In particular, given a bounded edge $e$ of $C$ that consists of two half-edges $e_1,e_2$, we refer to $e$ as $1/1$ edge if $R(e_1)=R(e_2)=1$, and we refer to $e$ as $2/0$ edge if either $R(e_1)=2$ and $R(e_2)=0$ or $R(e_1)=0$ and $R(e_2)=2$.
\end{notation}

\begin{definition}[Floor graph]\label{def:floor_graph}
Let $p_{[n]},L_{\underline{\kappa}^\alpha},L_{\underline{\kappa}^\beta},P_{\underline{\eta}^\alpha},P_{\underline{\eta}^\beta},\lambda_{[l]}$ be conditions in a stretched configuration such that the $z$-coordinate of $p_i$ is greater than the $z$-coordinate of $p_j$ if $i>j$. Let $C$ be a tropical stable map that is fixed by these conditions. The tropical stable map $C$ is floor decomposed by Proposition \ref{prop:floor_decomposed}. Given $C$, we associate a so-called \textit{floor graph} $\Gamma_C$, i.e. a weighted graph on an ordered set of vertices with a flow structure, to $C$ the following way:
each vertex of $\Gamma_C$ corresponds to a floor of $C$, an edge of $\Gamma_C$ corresponds to an elevator of $C$ and connects the vertices of $\Gamma_C$ that correspond to the floors the elevator connects in $C$. Weights on the edges of $\Gamma_C$ are induced by weights on the elevators of $C$. The given point conditions $p_{[n]}$ are totally ordered according to their $z$-coordinates. Thus the floors of the floor decomposed tropical stable map $C$ are also totally ordered, i.e. the vertices $v_{[n]}$ of $\Gamma_C$ are ordered as well, namely $v_1<\cdots <v_n$.
Moreover, a flow structure on $\Gamma_C$ is induced by the flows on the elevators of $C$ (see Notation \ref{notation:1/1_edge_and_2/0_edge}), i.e. if an elevator is a $1/1$ (resp. $2/0$) elevator, then its associated edge in $\Gamma_C$ is a $1/1$ (resp. $2/0$) edge.
\end{definition}

\begin{example}\label{ex:floor_graph}
Figure \ref{Figure33} shows the floor graph $\Gamma_C$ associated to the floor decomposed tropical stable map $C$ from Example \ref{ex:floor_decomposed_tropical_stable_map}. Notice that the elevator of $C$ is a $1/1$ elevator.

\begin{figure}[H]
\centering
\def\svgwidth{150pt}
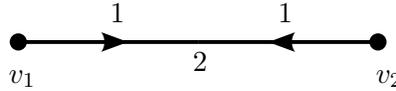
\caption{The floor graph $\Gamma_C$ associated to the floor decomposed tropical stable map $C$ from Example \ref{ex:floor_decomposed_tropical_stable_map}. The vertex $v_i$ of $\Gamma_C$ corresponds to the floor $C_i$ of $C$ for $i=1,2$.}
\label{Figure33}
\end{figure}
\end{example}

\subsection{Cutting elevators}
The following construction allows us to break floor decomposed tropical stable map into their parts by cutting elevators.

\begin{construction}[Cutting elevators]\label{constr:cutting_along_elevators}
Let $p_{[n]},L_{\underline{\kappa}^\alpha},L_{\underline{\kappa}^\beta},P_{\underline{\eta}^\alpha},P_{\underline{\eta}^\beta},\lambda_{[l]}$ be conditions in stretched configuration (with notation from Definition \ref{def:general_pos}) and let $C$ be a floor decomposed tropical stable map that is fixed by these conditions. If $e$ is an elevator of $C$, then we construct two tropical stable maps $C_1,C_2$ from $C$ by cutting $e$. The loose ends of $e$ are stretched to infinity. These ends (with its induced weights) are denoted by $e_i\in C_i$ for $i=1,2$ and the vertex adjacent to $e_i$ is denoted by $v_i$ for $i=1,2$. By abuse of notation we also refer to the label of $e_i$ by $e_i$.

The condition flow on $C$ induces flow structures on $C_i$, where the flow on $e_i$ is given by the flow on $e$ that is incoming to $v_i$ for $i=1,2$.

The degenerated cross-ratios are \textit{adapted} to the cutting the following way: If $\lambda_j$ is a degenerated cross-ratio that is satisfied at some vertex $v\in C_i$ for $i=1,2$, then, by the path criterion (Remark \ref{remark:path_criterion}), either all entries of $\lambda_j$ are labels of ends of $C_i$ or $3$ entries of $\lambda_j$ are labels of ends of $C_i$ and one entry $\beta$ is a label of an end of $C_t$ for $i\neq t\in\lbrace 1,2\rbrace$. In the first case, we do not change $\lambda_j$ and in the latter case, we replace the entry $\beta$ of $\lambda_j$ by $e_i$. We denote a degenerated cross-ratio that we adapted to $e_i$ by $\lambda_j^{\to e_i}$.

If $e$ is a $2/0$ elevator, then the component $C_i$ to which $2$ is the incoming flow along $e$ satisfies the codimension one tangency condition $P_{e_i}$, given by $\pi(v_t)$ for $i\neq t\in\lbrace 1,2\rbrace$, where $\pi$ is the projection from Notation \ref{notation:projection_moduli_spaces}.
If $e$ is a $1/1$ elevator, then each $C_i$ satisfies a codimension two condition $L_{e_i}$ for $i=1,2$, given by the projection $\pi$ of the movement of the vertices $v_i$. Notice that ends of $L_{e_i}$ are a priori not of standard direction. However, as we see with Corollary \ref{cor:cut_1/1_elevators_yield_standard_directions}, we can assume that $L_{e_i}$ for $i=1,2$ are --- like $L_{\underline{\kappa}^\alpha},L_{\underline{\kappa}^\beta}$ --- curves with ends of standard directions.

Denote the new sets of general positioned conditions each tropical stable map $C_i$ for $i=1,2$ satisfies by 
$p_{\underline{n_i}},L_{\underline{\kappa_i}^\alpha},L_{\underline{\kappa_i}^\beta},P_{\underline{\eta_i}^\alpha},P_{\underline{\eta_i}^\beta},\lambda^{\to e_i}_{\underline{l_i}}$.
Moreover, denote the degree of $C_i$ by $\Delta^3_{s_i}\left( \alpha^{i},\beta^{i} \right)$ for $i=1,2$ as in Notation \ref{notation:standard_directions_and_alpha_beta_degrees}.
\end{construction}

\begin{notation}\label{notation:adapt_CRs_to_multiple_cuts}
If Construction \ref{constr:cutting_along_elevators} is used to cut more than one elevator, it is can be necessary to adapt the cross-ratios $\lambda_{[l]}$ to more than one cut. This is denoted by $\lambda_j^{\to}$ for $\lambda_j\in\lambda_{[l]}$.
\end{notation}

\subsection{Multiplicities of floor decomposed curves}
Our next goal is to give a sufficiently local description of the multiplicity $\mult(C)$ (see Proposition \ref{prop:zsfssg_int_theory}) of a floor decomposed tropical stable map $C$, i.e. we aim for an expression of $\mult(C)$ which is a product of multiplicities, where each multiplicity is associated to a floor. The obvious approach of cutting elevator edges and determining multiplicities of the arising pieces works in case of $2/0$ elevators (see Notation \ref{notation:1/1_edge_and_2/0_edge}). It turns out that $1/1$ elevator edges that are adjacent to higher-valent vertices are more complicated. Here, we need to take the directions of the $1$-dimensional restrictions transported via a $1/1$ elevator into account.

A general tropical line $L\subset\mathbb{R}^2$ that is centered at $0$, with $3$ ends of standard directions and weight $1$ on each, is cut out by $\max_{(x,y)\in\mathbb{R}^2}(x,y,0)$. This allows us to look at degenerated lines as well.

\begin{definition}[Degenerated tropical lines]\label{def:degenerated_trop_lines}
The tropical intersections $L_{10}:=\max_{(x,y)\in\mathbb{R}^2}(x,0)\cdot \mathbb{R}^2$, $L_{01}:=\max_{(x,y)\in\mathbb{R}^2}(y,0)\cdot \mathbb{R}^2$ and $L_{1\text{-}1}:=\max_{(x,y)\in\mathbb{R}^2}(x,-y)\cdot \mathbb{R}^2$ and any translations thereof are called \textit{degenerated tropical lines}.
\end{definition}

\begin{figure}[H]
\centering
\def\svgwidth{350pt}
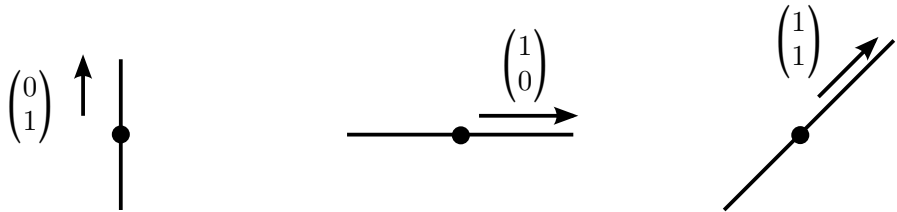
\caption{The degenerated tropical lines (from left to right) $L_{10},L_{01}$ and $L_{1\text{-}1}$ in $\mathbb{R}^2$ with ends of weight one.}
\label{Figure14}
\end{figure}

\begin{notation}[Replacing tangency conditions on $1/1$ edges]\label{notation:replacing_tangency_conditions_1/1-edge}
Let $C$ be a floor decomposed tropical stable map as in Construction \ref{constr:cutting_along_elevators} and let $e$ be a $1/1$ elevator. See Construction \ref{constr:cutting_along_elevators} for the following: cut $e$ and obtain two new tangency conditions $L_{e_1}$ (resp. $L_{e_2}$) that $C_1$ (resp. $C_2$) satisfy. Let $v_i$ be the vertex of $C_i$ that is adjacent to $e_i$ which satisfies $L_{e_i}$. Let $\pi(v_i)\in\mathbb{R}^2$ denote the projection of $v_i$ under $\pi$ along the elevator direction (see also Notation \ref{notation:projection_moduli_spaces}) for $i=1,2$. Let $L_{st}$ be a degenerated line of Definition \ref{def:degenerated_trop_lines} such that its vertex is translated to $\pi(v_1)$ (resp. $\pi(v_2)$). Let $C_{i,st}$ denote the tropical stable map that equals $C_i$, but where the $L_{e_i}$ tangency condition is replaced with $L_{st}$, i.e. $C_{i,st}$ satisfies $L_{st}$ instead of $L_{e_i}$ for $i=1,2$.

Notice that the multiplicities of $C_i$ and $C_{i,st}$ may differ. In particular, the multiplicity of $C_{i,st}$ may be zero, whereas the multiplicity of $C_i$ can be nonzero.
\end{notation}

\begin{example}\label{ex:replace_tangency_condition_floor_decomposed}
Consider the floor $C_2$ of $C$ from Example \ref{ex:floor_decomposed_tropical_stable_map}. The $\ev$-multiplicity of $C_{2,10}$ equals $1$ since it is the determinant of the $\ev$-Matrix of Example \ref{ex:ev-mult}. The $\ev$-multiplicity of $C_{2,01}$ is $0$ since $C_{2,01}$ is not fixed by its conditions.
\end{example}

\begin{lemma}\label{lemma:weights_cutting_elevators}
Let $C$ be a floor decomposed tropical stable map that contributes to the number $N_{\Delta_d^3(\alpha,\beta)}\left(p_{[n]},L_{\underline{\kappa}^\alpha},L_{\underline{\kappa}^\beta},P_{\underline{\eta}^\alpha},P_{\underline{\eta}^\beta},\lambda_{[l]} \right)$. Let $e$ be an elevator of weight $\omega(e)$ and cut $C$ along $e$ as in Construction \ref{constr:cutting_along_elevators} to obtain $C_1,C_2$.
\begin{itemize}
\item[(a)]
If $e$ is a $2/0$ elevator, then
\begin{align*}
\mult(C)=\omega(e)\cdot \mult(C_1)\cdot\mult(C_2).
\end{align*}
\item[(b)]
If $e$ is a $1/1$ elevator, then
\begin{align*}
\mult(C)=\omega(e)\cdot | \mult(C_{1,10})\cdot\mult(C_{2,01}) - \mult(C_{1,01})\cdot\mult(C_{2,10}) |,
\end{align*}
where tangency conditions are replaced as in Notation \ref{notation:replacing_tangency_conditions_1/1-edge}.
\end{itemize}
\end{lemma}

\begin{proof}
It is sufficient to prove Lemma \ref{lemma:weights_cutting_elevators} for $\ev$-multiplicities since the cross-ratio multiplicities can be expressed locally at vertices (see Proposition \ref{prop:zsfssg_int_theory}). Thus contributions from vertices to cross-ratio multiplicities do not depend on cutting edges.
\begin{itemize}
\item[(a)]
The proof of part (a) is basically the same as the one of part (a) of Proposition 3.4 of \cite{GeneralKontsevich}, and can easily be adapted to this situation.

\item[(b)]
The proof of part (b) follows ideas of \cite{GeneralKontsevich}.

We assume that the weights of each multi line $\omega(L_k)$ (see Definition \ref{def:tropical_line}) for $k\in\underline{\kappa}^\alpha\cup\underline{\kappa}^\alpha$ equals $1$ since we can pull out the factor $\omega(L_k)$ frome each row of the $ev$-matrix, apply all the following arguments and multiply with $\omega(L_k)$ later.

We use notation from Construction \ref{constr:cutting_along_elevators}, i.e. we denote the vertex of $C_1$ adjacent to the cut edge $e$ by $v_1$ and the other vertex adjacent to $e$ by $v_2$. The ev-matrix $M(C)$ of $C$ with respect to the base point $v_1$ is given by
\begin{align*}
M(C)=
\begin{array}{cc ccc|cccc !{\color{red!70!black}\vline width 1.5pt} c| cccc|cc}
  && \multicolumn{3}{c}{\footnotesize \textrm{Base $v_1$}}  & \multicolumn{5}{c}{\footnotesize \textrm{lengths in $C_1$}}&\multicolumn{4}{c}{\footnotesize \textrm{lengths in $C_2$}}& \footnotesize \textrm{$l_e$} &\\
\footnotesize\textrm{conditions in $C_1$} &\ldelim({6}{0.5em}& \multicolumn{3}{c|}{\multirow{3}*{*}}  & \multicolumn{4}{c!{\color{red!70!black}\vline width 1.5pt}}{\multirow{3}*{*}}&* & \multicolumn{4}{c|}{\multirow{3}*{0}} &0
  & \rdelim){6}{0.5em} \\
  & &&& &&&& &\vdots&&&& &\vdots&\\
  & &&& &&&& &*&&&& &0&\\
  \arrayrulecolor{red!70!black}\Cline{1.5pt}{2-16}\arrayrulecolor{black} \footnotesize\textrm{conditions in $C_2$} && \multicolumn{3}{c|}{\multirow{3}*{*}} & \multicolumn{4}{c!{\color{red!70!black}\vline width 1.5pt}}{\multirow{3}*{0}}& 0 & \multicolumn{4}{c|}{\multirow{3}*{*}} &*
  &  \\
  & &&& &&&& &\vdots&&&& &\vdots&\\
  & &&& &&&& &0&&&& &*&\\
\end{array}
\end{align*}
The bold red lines divide $M(C)$ into square pieces at the upper left and the lower right. This follows from similar arguments used in the proof of part (a). Let $M$ be the matrix consisting of the lower right block of $M(C)$ whose entries (see above) are indicated by $*$ and its columns are associated to lengths in $C_2$.
Let $A=(a_{ij})_{ij}$ be the submatrix of $M(C)$ given by the rows that belong to conditions of $C_1$ and by the base point's columns and the columns that are associated to lengths in $C_1$, i.e. $A$ consists of all the $*$-entries above the bold red line in $M(C)$.

Consider the Laplace expansion of the rightmost column of $A$. Recursively, use Laplace expansion on every column that belongs to the lengths in $C_1$ starting with the rightmost column.
Eventually, we end up with a sum in which each summand contains a factor $\det (N_i)$ for a matrix $N_i$, which is one of the following two matrices, namely
\begin{align*}
N_1=
\begin{array}{c cc!{\color{red!70!black}\vline width 1.5pt} c |c|ccc c}
  & \multicolumn{3}{c}{\footnotesize \textrm{Base $v_1$}} & \multicolumn{1}{c}{\footnotesize \textrm{$l_e$}} &&&&\\
  \ldelim({6}{0.5em} &*&*&0& 0& \multicolumn{3}{c}{\multirow{2}*{0}} & \rdelim){6}{0.5em} \\
  &*&*&0&0&&&&\\
  \arrayrulecolor{red!70!black}\Cline{1.5pt}{1-9}\arrayrulecolor{black} &  \multicolumn{2}{c!{\color{red!70!black}\vline width 1.5pt}}{\multirow{4}*{*}}& * & * & \multicolumn{3}{c}{\multirow{4}*{$M$}} & \\
&&&\vdots&\vdots &&&&\\
&&&\vdots&\vdots &&&&\\
&&&*&*&&&\\
\end{array} \textrm{\quad and\quad }
N_2=
\begin{array}{c cc!{\color{red!70!black}\vline width 1.5pt} c |c|ccc c}
  & \multicolumn{3}{c}{\footnotesize \textrm{Base $v_1$}} & \multicolumn{1}{c}{\footnotesize \textrm{$l_e$}} &&&&\\
  \ldelim({6}{0.5em} &*&*&b_1& 0& \multicolumn{3}{c}{\multirow{2}*{0}} & \rdelim){6}{0.5em} \\
  &*&*&b_2&0&&&&\\
  \arrayrulecolor{red!70!black}\Cline{1.5pt}{1-9}\arrayrulecolor{black} &  \multicolumn{2}{c!{\color{red!70!black}\vline width 1.5pt}}{\multirow{4}*{*}}& * & * & \multicolumn{3}{c}{\multirow{4}*{$M$}} & \\
&&&\vdots&\vdots &&&&\\
&&&\vdots&\vdots &&&&\\
&&&*&*&&&\\
\end{array}
\end{align*}
Since the $l_e$ column of $N_1$ equals $\omega(e)$ times the third column of $N_1$, the determinant $\det(N_1)$ is zero and thus does not occur in the Laplace expansion from above. In case of matrix $N_2$, at least one of the entries $b_1,b_2$ is $1$. Moreover, if $b_1$ or $b_2$ equals $1$, then this $1$ is the only nonzero entry in the whole row. Thus Laplace expanding this row and dividing the $l_e$ column by $\omega(e)$ to obtain the column $\tilde{l}_e$ (which gives the global factor of $\omega(e)$ in part (b) of Lemma \ref{lemma:weights_cutting_elevators}) yields the following $3$ cases.
\begin{align*}
M_{a_{r1}a_{r2}}:=
\begin{array}{c ccc|ccc c}
  & & & \multicolumn{1}{c}{\footnotesize \textrm{$\tilde{l}_e$}} &\multicolumn{3}{c}{\footnotesize \textrm{lenghts in $C_2$}}&\\
  \ldelim({6}{0.5em} &a_{r1}&a_{r2}&0&0&\dots&0& \rdelim){6}{0.5em}\\
   \cline{1-8}  &\multicolumn{3}{c|}{\multirow{5}*{$*$}}&\multicolumn{3}{c}{\multirow{5}*{$M$}} & \\
  &&&&&&&\\
  &&&&&&&\\
  &&&&&&&\\
  &&&&&&&\\
\end{array},
\end{align*}
where $(a_{r1},a_{r2})=(1,0)$, $(a_{r1},a_{r2})=(0,1)$ or $(a_{r1},a_{r2})=(1,-1)$ are the remaining entries of $A$ in its $r$-th row after the recursive procedure.
Notice that in each case the entries of the first $3$ columns are of such a from that $M_{st}$ for $st=10,01,1\text{-}1$ is the ev-matrix of $C_{2,st}$ (see Notation \ref{notation:replacing_tangency_conditions_1/1-edge}) with base point $v_2$.

We can group the summands according to the values $a_{r1},a_{r2}$ and obtain in total
\begin{align}\label{eq:proof5_multiplicities_of_1/1_and_2/0}
|\det(M(C))|=\omega(e)\cdot|F_{10}\cdot\det(M_{10})+F_{01}\cdot\det(M_{01})+F_{1\text{-}1}\cdot\det(M_{1\text{-}1})|,
\end{align}
where $F_{st}\in\mathbb{R}$ for $st=10,01,1\text{-}1$ are factors occuring due to the recursive Laplace expansion. 
More precisely, let $b'$ be the number of bounded edges in $C_1$ and define $b:=b'+1$, i.e. $b$ is the total number of Laplace expansions we applied. Then
\begin{align}\label{eq:proof_multiplicities_of_1/1_and_2/0_F_st}
F_{st}=\sum_{r:(a_{r1},a_{r2})=(s,t)}\sum_{\sigma} \sgn(\sigma) \prod_{j=3}^{3+b} a_{\sigma(j)j},
\end{align}
where the second sum goes over all bijections $\sigma:\lbrace 3,\dots,3+b \rbrace\to \lbrace 1,\dots,r-1,r+1,\dots,b+1\rbrace$, i.e. it goes over all possibilities of choosing for each column Laplace expansion was used on an entry in a row of $A$ which is not the $r$-th row.

Let $A_{10},A_{01},A_{1\text{-}1}$ be the square matrices obtained from $A$ by adding the new first row $(1,0,0,\dots,0)$, $(0,1,0\dots,0)$ or $(1,-1,0,\dots,0)$ to $A$. Again, notice that $A_{st}$ for $st=10,01,1\text{-}1$ is the ev-matrix of $C_{1,st}$ (see Notation \ref{notation:replacing_tangency_conditions_1/1-edge}, Definition \ref{def:ev_matrix}) with base point $v_1$.

We claim that
\begin{align}\label{eq:proof1_multiplicities_of_1/1_and_2/0}
\det(A_{10})=F_{01}-F_{1\text{-}1}
\end{align}
holds. Let $N$ be the number of columns and rows of $A_{st}$. Denote the entries of the $\ev$-matrix $M(C)$ by $m(C)_{ij}$. Define
\begin{align*}
S_{st}:=\lbrace r\in [N-1] \mid m(C)_{r1}=s,\; m(C)_{r2}=t \rbrace
\end{align*}
for $(s,t)=(1,0),(0,1),(1,-1)$ and notice that $\#S_{10}+\#S_{01}+\#S_{1\textrm{-}1}=N-1$. Denote the entries of $A_{10}$ by $a^{(10)}_{ij}$ and apply Leibniz' determinant formula to obtain
\begin{align*}
\det(A_{10})&=\sum_{\sigma\in\mathds{S}_N}\sgn(\sigma)\prod_{j=1}^N a^{(10)}_{\sigma(j)j}\\
&=\sum_{\substack{\sigma\in\mathds{S}_N\\ \sigma(2)\in S_{01}}}\sgn(\sigma)\prod_{j=1}^N a^{(10)}_{\sigma(j)j}
+
\sum_{\substack{\sigma\in\mathds{S}_N\\ \sigma(2)\in S_{1\textrm{-1}}}}\sgn(\sigma)\prod_{j=1}^N a^{(10)}_{\sigma(j)j}
=F_{01}-F_{1\text{-}1},
\end{align*}
where the second equality holds by definition of $S_{st}$ and the third equality holds by considering how contributions of $F_{01}$ and $F_{1\text{-}1}$ arise as choices of entries of $A$, see \eqref{eq:proof_multiplicities_of_1/1_and_2/0_F_st}. The minus sign comes from the factor $a^{(10)}_{\sigma(2),2}=-1$ in each product in the last sum. Thus \eqref{eq:proof1_multiplicities_of_1/1_and_2/0} holds.

We can show in a similar way that
\begin{align}
\det(A_{01})&=-\left(F_{10}+F_{1\text{-}1}\right)=-F_{10}-F_{1\text{-}1},\label{eq:proof2_multiplicities_of_1/1_and_2/0} \\
\det(A_{1\text{-}1})&=F_{10}+F_{1\text{-}1}+F_{01}-F_{1\text{-}1}=F_{10}+F_{01}\label{eq:proof3_multiplicities_of_1/1_and_2/0}
\end{align}
hold. Solving the system of linear equations \eqref{eq:proof1_multiplicities_of_1/1_and_2/0}, \eqref{eq:proof2_multiplicities_of_1/1_and_2/0}, \eqref{eq:proof3_multiplicities_of_1/1_and_2/0} for $F_{10},F_{01},F_{1\text{-}1}$ yields
\begin{align}\label{eq:proof4_multiplicities_of_1/1_and_2/0}
\left(\begin{array}{c}
F_{10}\\
F_{01}\\
F_{1\text{-}1}
\end{array}\right)\in
\left(\begin{array}{c}
-\det(A_{01})\\
\det(A_{10})\\
0
\end{array}\right)+
\langle\left(\begin{array}{c}
-1\\
1\\
1
\end{array}\right)\rangle,
\end{align}
where the $1$-dimensional part appears because of the relation
\begin{align*}
-\det(M_{10})+\det(M_{01})+\det(M_{1\text{-}1})=0.
\end{align*}
Combining \eqref{eq:proof5_multiplicities_of_1/1_and_2/0} with \eqref{eq:proof4_multiplicities_of_1/1_and_2/0} proves part (b) of Lemma \ref{lemma:weights_cutting_elevators}, where $A_{st}=C_{1,st}$ and $M_{st}=C_{2,st}$.
\end{itemize}
\end{proof}


Lemma \ref{lemma:weights_cutting_elevators} gives rise to a graphical interpretation of $\mult(C)$ if $C$ is floor decomposed. For that, we want to iteratively use part (b) of Lemma \ref{lemma:weights_cutting_elevators} with the following notation.

\begin{notation}[Iterating Notation \ref{notation:replacing_tangency_conditions_1/1-edge}]\label{notation:replacing_tangency_conditions_1/1-edge_iterated}
Let $C$ denote a floor decomposed tropical stable map as in Construction \ref{constr:cutting_along_elevators} and let $C_i$ denote a floor of $C$. The collection of labels of ends arising from cutting $1/1$ elevators adjacent to $C_i$ whose primitive direction is $(0,0,-1)\in\mathbb{R}^3$ (resp. $(0,0,1)$) is denoted by $\underline{{1/1}_i}^\alpha$ (resp. $\underline{{1/1}_i}^\beta$)
Let $L_{\underline{{1/1}_i}^\alpha}$ (resp. $L_{\underline{{1/1}_i}^\beta}$) denote the tangency conditions, that arose from cutting the $1/1$ elevators, and that ends of $C_i$ satisfy. 
Then let $C_{i,st_{\underline{{1/1}_i}^\alpha};st_{\underline{{1/1}_i}^\beta}}$ denote the floor $C_i$ where the tangency condition $L_k$ is replaced by a tangency condition $L_{st_{k}}$ that is a degenerated line for $k\in \underline{{1/1}_i}^\alpha \cup \underline{{1/1}_i}^\beta$ as in Notation \ref{notation:replacing_tangency_conditions_1/1-edge}.
\end{notation}

\begin{definition}[Graphical contribution]\label{def:graphical_contribution}
Let $C$ be a floor decomposed tropical stable map contributing to $N_{\Delta_d^3(\alpha,\beta)}\left(p_{[n]},L_{\underline{\kappa}^\alpha},L_{\underline{\kappa}^\beta},P_{\underline{\eta}^\alpha},P_{\underline{\eta}^\beta},\lambda_{[l]} \right)$ and let $\Gamma_C$ denote its floor graph, see Definition \ref{def:floor_graph}.
Cut the edges of $\Gamma_C$ the following way: If $e$ is a $2/0$ edge of $\Gamma_C$, just cut it. If $e$ is a $1/1$ edge of $\Gamma_C$, cut it and attach a small horizontal line segment to one of the loose ends and a small vertical line segment to the other loose end. Cutting all edges of $\Gamma_C$ this way gives a decorated graph, called \textit{graphical contribution} to $\mult(C)$. 

The \textit{multiplicity} of a graphical contribution is defined the following way. Given a graphical contribution $G$ to $\mult(C)$, we draw its vertices on a line in the plane such that vertices corresponding to points with smaller $z$-coordinate are more to the left than vertices corresponding to points with greater $z$-coordinates. Let $u$ be the number of cut $1/1$ edges of $G$ where the loose end with a horizontal line segment is attached to the left of the two vertices adjacent to this $1/1$ edge. Define the \textit{multiplicity} of the graphical contribution $G$ to $C$ as
\begin{align*}
\mult(G):=(-1)^{u}\cdot \prod_i \mult\left( C_{i,st_{\underline{{1/1}_i}^\alpha};st_{\underline{{1/1}_i}^\beta}} \right) \cdot \prod_e \omega(e),
\end{align*}
where the first product goes over all floors of $C$ (for notation, see Construction \ref{constr:cutting_along_elevators}). Here, every codimension two tangency condition was replaced by some degenerated line condition the following way (see also Notation \ref{notation:replacing_tangency_conditions_1/1-edge_iterated}): if $e$ is a $1/1$ elevator that used to connect the curves $C_i$ and $C_j$ (that we obtained from cutting $C$), and the loose end $e_i$ of $e$ (that is obtained by cutting $e$) that is adjacent to $C_i$ is equipped with a vertical (resp. horizontal) line segment, then replace the codimension two tangency condition $L_{e_i}$ that $e_i$ satisfies with the degenerated line condition $L_{10}$ (resp. $L_{01}$). More precisely (see Figure \ref{Figure14}), the line segments in the graphical contribution represent the degenerated tropical line conditions, i.e. vertical (resp. horizontal) segments represent vertical (resp. horizontal) degenerated tropical lines. The second product goes over all edges $e$ of $\Gamma_C$ and multiplies their weights $\omega(e)$.
\end{definition}

\begin{example}\label{ex:graphical_contributions}
Let $\Gamma_C$ be the weighted graph shown in Figure \ref{Figure15}, where we suppress the weight of an edge if it is one.

\begin{figure}[H]
\centering
\def\svgwidth{400pt}
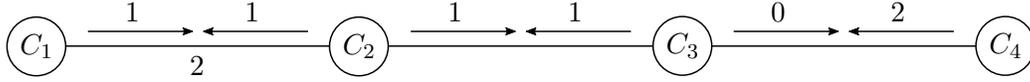
\caption{$\Gamma_C$ with its flow structure and weights.}
\label{Figure15}
\end{figure}

\noindent There are $4$ different graphical contributions that $\Gamma_C$ gives rise to, see Figure \ref{Figure16}. Denote the graphical contributions shown in Figure \ref{Figure16} from top to bottom by $G_1,\dots,G_4$. The multiplicities can be read off as:
\begin{align*}
\mult(G_1)&=(-1)^0\cdot 2\cdot\mult(C_{1,\emptyset;10})\cdot\mult(C_{2,01;10})\cdot\mult(C_{3,01;\emptyset})\cdot \mult(C_{4})\\
\mult(G_2)&=(-1)^1\cdot2\cdot\mult(C_{1,\emptyset;10})\cdot\mult(C_{2,01;01})\cdot\mult(C_{3,10;\emptyset})\cdot \mult(C_{4})\\
\mult(G_3)&=(-1)^1\cdot2\cdot\mult(C_{1,\emptyset;01})\cdot\mult(C_{2,10;10})\cdot\mult(C_{3,01;\emptyset})\cdot \mult(C_{4})\\
\mult(G_4)&=(-1)^2\cdot 2\cdot\mult(C_{1,\emptyset;01})\cdot\mult(C_{2,10;01})\cdot\mult(C_{3,10;\emptyset})\cdot \mult(C_{4})
\end{align*}

\begin{figure}[H]
\centering
\def\svgwidth{400pt}
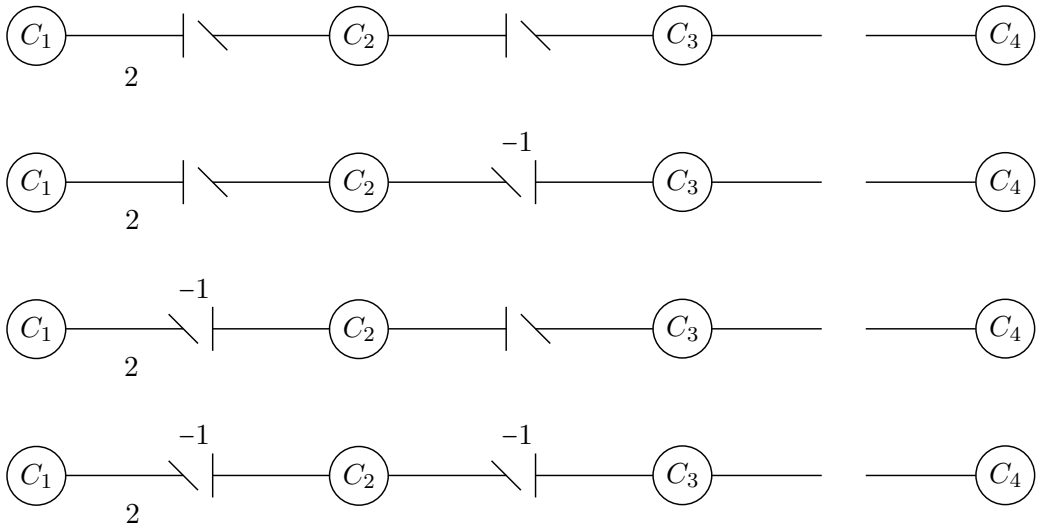
\caption{All graphical contributions associated to $\Gamma_C$, where the vertical and horizontal line segments and configurations of them that contribute to $u$ from Definition \ref{def:graphical_contribution} are indicated by $-1$.}
\label{Figure16}
\end{figure}

\end{example}

\begin{proposition}\label{prop:graphical_contributions}
Let $C$ be a floor decomposed tropical stable map that contributes to the number $N_{\Delta_d^3(\alpha,\beta)}\left(p_{[n]},L_{\underline{\kappa}^\alpha},L_{\underline{\kappa}^\beta},P_{\underline{\eta}^\alpha},P_{\underline{\eta}^\beta},\lambda_{[l]} \right)$. Then
\begin{align*}
\mult(C)=|\sum_G \mult(G)|,
\end{align*}
where the sum goes over all graphical contributions to $C$, see Definition \ref{def:graphical_contribution}.
\end{proposition}

\begin{proof}
Iterate Lemma \ref{lemma:weights_cutting_elevators}, where, in part (b) of Lemma \ref{lemma:weights_cutting_elevators}, we use the convention that $C_1$ lies on the left of $C_2$ with respect to the order used in Definition \ref{def:graphical_contribution}.
\end{proof}

\subsection{Pushing forward conditions along elevators}
The aim of this subsection is to prove the following proposition, which determines how the $1$-dimensional conditions a floor decomposed tropical stable map exchanges via its $1/1$ elevators look like. More precisely, cutting a $1/1$ elevator $q$ adjacent to the floors $C_i,C_j$ leads to loose edges that can move in a $1$-dimensional way, i.e. the floor $C_i$ adjacent to $q$ gives rise to a $1$-dimensional cycle $Y_{i,q}$ that can be pushed forward to $\mathbb{R}^2$ using $\partial\ev_{q}$. The cycle $\partial\ev_{q,*}(Y_{i,q})$ is the $1$-dimensional restriction $C_i$ imposes on $C_j$ via the elevator $q$.

\begin{proposition}\label{prop:push-forward}
For notation, see Notation \ref{notation:underlined_symbols}, \ref{def:general_pos}, \ref{notation:adapt_CRs_to_multiple_cuts} and Construction \ref{constr:cutting_along_elevators}. Let $C_i$ be a floor of a floor decomposed tropical stable map $C\in\mathcal{M}_{0,n}\left(\mathbb{R}^3,\Delta_d^3(\alpha,\beta)\right)$ which satisfies general positioned conditions $p_{[n]},L_{\underline{\kappa}^\alpha},L_{\underline{\kappa}^\beta},P_{\underline{\eta}^\alpha},P_{\underline{\eta}^\beta},\lambda_{[l]}$. Let $\Delta^3_{s_i}\left( \alpha^{i},\beta^{i} \right)$ be the degree of $C_i$ and let $q\in\Delta^3_{s_i}\left( \alpha^{i},\beta^{i} \right)$ be the label of an end whose primitive direction is $(0,0,\pm 1)$. The cycle
\begin{align*}
Y_{i,q}
:=
\prod_{k\in \underline{\kappa_{i}}^{\alpha}\cup \underline{\kappa_{i}}^{\beta}} \partial\ev_k^{*}(L_k)& 
\cdot
\prod_{f\in \underline{\eta_{i}}^{\alpha}\cup \underline{\eta_{i}}^{\beta}} \partial\ev_f^{*}(P_f)
\cdot
\prod_{j\in \underline{l_i}} \ft_{\lambda_j^{\to}}^*\left( 0\right)
\cdot
\ev_i^*\left( p_i\right)
\cdot
\mathcal{M}_{0,1}\left(\mathbb{R}^3, \Delta^3_{s_i}\left( \alpha^{i},\beta^{i} \right)\right)
\end{align*}
has the following properties.
\begin{itemize}
\item[(1)] The recession fan of the push-forward $\partial\ev_{q,*}(Y_{i,q})$ does only contain ends of standard directions.
\item[(2)] Each unbounded cell $\sigma\in Y_{i,q}$ that is mapped to an end of the recession fan of $\partial\ev_{q,*}(Y_{i,q})$ under the push-forward $\partial\ev_{q,*}$ satisfies the following: If $C_\sigma\in\sigma$ is a tropical stable map in the interior of $\sigma$, then $q$ is adjacent to a $3$-valent vertex, which is adjacent to another end $E\neq q$ such that $\pi(E)\subset \mathbb{R}^2$ is an end of standard direction.
\end{itemize}
\end{proposition}

An immediate consequence of Proposition \ref{prop:push-forward} is the following corollary, which yields that all restrictions exchanged via a $1/1$ elevator are in fact tropical curves with ends of standard direction.

\begin{corollary}\label{cor:cut_1/1_elevators_yield_standard_directions}
Let $C$ be a floor decomposed tropical stable map that contributes to the number $N_{\Delta_d^3(\alpha,\beta)}\left(p_{[n]},L_{\underline{\kappa}^\alpha},L_{\underline{\kappa}^\beta},P_{\underline{\eta}^\alpha},P_{\underline{\eta}^\beta},\lambda_{[l]} \right)$. Then the codimension two tangency conditions each $1/1$ elevator passes on to its neighbors have ends of standard direction only. In particular, we can assume that if we cut all elevators as in Construction \ref{constr:cutting_along_elevators}, then the appearing codimension two tangency conditions have ends of standard directions.
\end{corollary}

\begin{proof}
Apply part (1) of Proposition \ref{prop:push-forward} inductively by cutting one $1/1$ elevator after another.
\end{proof}

\begin{remark}
Notice that Proposition \ref{prop:push-forward} can also be shown the way Corollary 2.31 in \cite{GeneralKontsevich} was shown, where Corollary 2.31 follows from Proposition 2.1 of \cite{GeneralKontsevich}. Since  Proposition 2.1 of \cite{GeneralKontsevich} is actually a stronger statement than Proposition \ref{prop:push-forward} there is is no need to evoke the machinery developed in \cite{GeneralKontsevich}.
\end{remark}

\begin{proof}[Proof of Proposition \ref{prop:push-forward}]
Let $L_{10}$ be a degenerated tropical line in $\mathbb{R}^2$ which is parallel to the $y$-axis as in Definition \ref{def:degenerated_trop_lines}. The projection formula (Proposition \ref{prop:projection_formula} in case of abstract cycles) yields
\begin{align}\label{eq:alternative_proof_push-forward:proj_formula}
L_{10}\cdot \partial\ev_{q,*}(Y_{i,q})=\partial\ev_{q,*}\left( \partial\ev_{q}^*(L_{10})\cdot Y_{i,q} \right).
\end{align}
Assume that the degenerated line $L_{10}$ is shifted in the direction $(-1,0)\in\mathbb{R}^2$ such that $L_{10}$ intersects $\partial\ev_{q,*}(Y_{i,q})\subset\mathbb{R}^2$ only in $1$-dimensional ends of $\partial\ev_{q,*}(Y_{i,q})$.

Let $\pi:\mathbb{R}^3\to\mathbb{R}^2$ be the projection that forgets the $z$-coordinate and let
\begin{align*}
\tilde{\pi}:
\mathcal{M}_{0,1}\left(\mathbb{R}^3, \Delta^3_{s_i}\left( \alpha^{i},\beta^{i} \right)\right)
\to
\mathcal{M}_{0,1+|\alpha^{i}|+|\beta^{i}|}\left(\mathbb{R}^2, \pi\left(\Delta^3_{s_i}\left( \alpha^{i},\beta^{i} \right)\right)\right)
\end{align*}
be its induced map $\pi$ on the moduli spaces as in Notation \ref{notation:projection_moduli_spaces}.

Each tropical stable map corresponding to a point of $\tilde{\pi}_*(Y_{i,q})$ can be lifted uniquely to a tropical stable map corresponding to a point in $Y_{i,q}$ as in the proof of Proposition \ref{prop:mult_CR_floor_diagram_2_dim}. Thus for
\begin{align*}
&Y_{\pi,i,q}:=\\
&\prod_{k\in \underline{\kappa_{i}}^{\alpha}\cup \underline{\kappa_{i}}^{\beta}} \ev_k^{*}(L_k)& 
\cdot
\prod_{f\in \underline{\eta_{i}}^{\alpha}\cup \underline{\eta_{i}}^{\beta}} \ev_f^{*}(P_f)
\cdot
\prod_{j\in \underline{l_i}} \ft_{\lambda_j^{\to q}}^*\left( 0\right)
\cdot
\ev_i^*\left( \pi(p_i)\right)
\cdot
\mathcal{M}_{0,1+|\alpha^{i}|+|\beta^{i}|}\left(\mathbb{R}^2, \pi\left(\Delta^3_{s_i}\left( \alpha^{i},\beta^{i} \right)\right)\right)
\end{align*}
the equality 
\begin{align}\label{eq:proof_push_forward_conditions_standard_directions}
\tilde{\pi}_*(Y_{i,q})=Y_{\pi,i,q}
\end{align}
holds on the level of sets. To see that \eqref{eq:proof_push_forward_conditions_standard_directions} also holds on the level of cycles, multiplicities are compared. Notice that each multiplicity of a top-dimensional cell of $\tilde{\pi}_*(Y_{i,q})$ (resp. $Y_{\pi,i,q}$) arises as a product of a cross-ratio multiplicity and an index of an $\ev$-matrix, see Definition \ref{def:ev_matrix}. The lifting of the proof of Proposition \ref{prop:mult_CR_floor_diagram_2_dim} guarantees that the cross-ratio multiplicity part coincides. Let $\sigma$ be a top-dimension cell of $Y_{i,q}$ the $\ev$-multiplicity part of $\sigma$ is given by the absolute value of the index of the $\ev$-Matrix $M(\sigma)$ associated to $\sigma$, see Definition \ref{def:ev_matrix}. We choose $p_i$ as base point for the local coordinates used for $M(\sigma)$. Then
\begin{align*}
M(\sigma)=
\begin{array}{c ccc| ccc c}
  & \multicolumn{3}{c}{\footnotesize \textrm{Base $p_i$}} & &&&\\
  \ldelim({6}{0.5em} &1&0&0& 0& \dots & 0& \rdelim){6}{0.5em} \\
  &0&1&0&0&\dots&0&\\
  &0&0&1&0&\dots&0&\\
  \cline{1-8} & \multicolumn{2}{c}{\multirow{3}*{$*$}} & \multicolumn{1}{|c|}{\multirow{1}*{$*$}} &  \multicolumn{3}{c}{\multirow{3}*{$*$}} & \\
&&&\multicolumn{1}{|c|}{\multirow{1}*{\vdots}}&&&&\\
&&&\multicolumn{1}{|c|}{\multirow{1}*{$*$}}&&&&\\
\end{array}.
\end{align*}
The $\ev$-matrix $M(\pi(\sigma))$ is obtained from $M(\sigma)$ by erasing the third column and row which intersect in the $z$-coordinate of the base point, which is $1$.
Recall that $\partial\ev$ is by definition $\ev\circ\pi$, i.e. the indices of $M(\sigma)$ and $M(\pi(\sigma))$ are equal. Therefore \eqref{eq:proof_push_forward_conditions_standard_directions} holds on the level of cycles.

By definition of $\partial\ev_q$ and the arguments from before, the right-hand side of \eqref{eq:alternative_proof_push-forward:proj_formula} is equal to $\ev_{q,*}\left( \ev_{q}^*(L_{10})\cdot Y_{\pi,i,q} \right)$. Moreover, by shifting $L_{10}$ to the left as before, we can assume that $\ev_{q}^*(L_{10})$ intersects $Y_{\pi,i,q}$ only in ends of $Y_{\pi,i,q}$.

We claim that each tropical stable map $C$ that contributes to the $0$-dimensional cycle $\ev_{q}^*(L_{10})\cdot Y_{\pi,i,q}$ has an end of direction $(-1,0)\in\mathbb{R}^2$ that is adjacent to a $3$-valent vertex $v$ which in turn is adjacent to the contracted end $q$. To prove the claim, it is sufficient to show that $C$ has no vertex that is not adjacent to $q$ whose $x$-coordinate is smaller or equal to the one of $L_{10}$. Assume that there is a vertex $v$ of $C$ that is not adjacent to $q$ and that the $x$-coordinate of $v$ is minimal. Assume also that $v$ is adjacent to an end $e$ of direction $(-1,0)\in\mathbb{R}^2$. Since each entry of a given tropical cross-ratio is a contracted end, Corollary \ref{cor:CR_pfade_ueber_alle_edges_an_vertex} yields that $v$ is $3$-valent. If $v$ is adjacent to a contracted end $e$, then this end needs to satisfy a condition, otherwise $v$ allows a $1$-dimensional movement which is a contradiction. Since $L_{10}$ was moved sufficiently into the direction of $(-1,0)$ in $\mathbb{R}^2$, we know that the condition $e$ satisfies can only be a multi line condition that locally around $v$ is parallel to the $x$-axis of $\mathbb{R}^2$. Hence $v$ allows again a $1$-dimensional movement which is a contradiction. In total, $v$ is $3$-valent, adjacent to an end of direction $(-1,0)$ and is not adjacent to a contracted end.

Assume additionally that the $y$-coordinate of $v$ is minimal among the vertices with minimal $x$-coordinate that are not adjacent to $q$ and that are adjacent to an end of direction $(-1,0)$.
We distinguish two cases:
\begin{itemize}
\item
In the fist case, the $x$-coordinate of $v$ is strictly smaller than the one of $L_{10}$. Hence $v$ is by balancing (all ends have weight $1$) adjacent to an edge of direction $(0,-1)$. If this edge is an end, then $v$ gives rise to a $1$-dimensional movement which is a contradiction. If this edge leads to a vertex $v'$ that is adjacent to a contracted end $t$, then $t$ can only satisfy a multi line condition that locally around $t$ is parallel to the $x$-axis of $\mathbb{R}^2$. By our assumption, there is no vertex with the same $x$-coordinate as $v'$ below $v'$ that is adjacent to an end of direction $(0,-1)$.
Hence $v$ allows a $1$-dimensional movement which is a contradiction.
\item
In the second case, the $x$-coordinate of $v$ equals the $x$-coordinate of $L_{10}$ and $v$ is adjacent to a vertex $v'$ that is in turn adjacent to the contracted end $q$ such that the $y$-coordinate of $v'$ is smaller than the one of $v$ (if there is no such vertex $v'$, then we end up with the same contradiction as in case one). By Corollary \ref{cor:CR_pfade_ueber_alle_edges_an_vertex}, $v'$ cannot be adjacent to an end of direction $(-1,0)$ and by minimality of the $y$-coordinate of $v$, there is an end $e'$ of direction $(0,-1)\in\mathbb{R}^2$ adjacent to $v'$ which is parallel to $L_{10}$. Thus $v'$ allows a $1$-dimensional movement which is a contradiction.
\end{itemize}
Thus the claim is true.

Since the $x$-coordinate of $L_{10}$ is so small that $\ev_{q}^*(L_{10})$ intersectes $Y_{\pi,i,q}$ in ends only, we can use the claim from above to determine the directions of those ends: If we consider a point in $\ev_{q}^*(L_{10})\cdot Y_{\pi,i,q}$, forget $\ev_{q}^*(L_{10})$ and move $q$ a bit, the primitive direction of movement of $q$ is the direction $(-1,0)\in\mathbb{R}^2$, which is a standard direction. Moving $L_{10}$ slightly and applying \eqref{eq:alternative_proof_push-forward:proj_formula} yields that $L_{10}$ intersects the push-forward $\partial\ev_{q,*}(Y_{i,q})$ in ends of direction $(-1,0)\in\mathbb{R}^2$ only.

We can use similar arguments for $L_{10}$ if the $x$-coordinate of $L_{10}$ is so large that it intersects $\partial\ev_{q,*}(Y_{i,q})$ in ends only, and we can use similar arguments for $L_{01}$ with small (resp. large) $y$-coordinate. In total, it follows that ends of $\partial\ev_{q,*}(Y_{i,q})$ are of standard direction and that their weights are given as in Proposition \ref{prop:push-forward}.
\end{proof}

\section{Cross-ratio floor diagrams}

\begin{definition}[Cross-ratio floor diagram]\label{def:CR_floor_diagram}
Let $p_{[n]},L_{\underline{\kappa}^\alpha},L_{\underline{\kappa}^\beta},P_{\underline{\eta}^\alpha},P_{\underline{\eta}^\beta},\lambda_{[l]}$ be general positioned conditions as in Definition \ref{def:general_pos} with respect to the degree $\Delta^3_d(\alpha,\beta)$, i.e. $\#\Delta^3_d(\alpha,\beta)=2n+l+2\cdot \#\underline{\eta}+\#\underline{\kappa}$. Moreover, each entry of a degenerated tropical cross-ratio is a label of a contracted end or a label of an end of primitive direction $(0,0,\pm 1)\in\mathbb{R}^3$.

A \textit{cross-ratio floor diagram} $\mathcal{F}$ (that satisfies the given conditions) is a tree without ends on a totally ordered set of vertices $v_1<\dots <v_n$ with a flow structure that is a condition flow of type $3$ such that $\mathcal{F}$ satisfies the following properties:
\begin{enumerate}[label=(\arabic*),ref=(\arabic*)]
\item \label{item_1:def:cr_floor_diagram}
Each vertex is labeled with a possibly empty set of labeled ends $\delta_{v_i}\subset\Delta^3_d(\alpha,\beta)$ such that $\delta_{v_i}\cap \delta_{v_j}=\emptyset$ for all $i\neq j$ and $\bigcup_{i=1}^n \delta_{v_i}=\Delta^3_d(\alpha,\beta)$.
\item \label{item_2:def:cr_floor_diagram}
Each edge $e$ of $\mathcal{F}$ (consisting of two half-edges) is equipped with a \textit{weight} $\omega(e)\in\mathbb{N}_{>0}$ such that vertices $v_i$ of $\mathcal{F}$ are \textit{balanced} with respect to these weights, i.e.
\begin{align*}
\#\delta_{v_i}^{(1,1,1)}+\sum_{e'\in\delta_{v_i}^\beta}\omega(e')+\sum_{\substack{\textrm{$e$ an edge}\\\textrm{between $v_i<v_j$}}} \omega(e) -\sum_{e'\in\delta_{v_i}^\alpha}\omega(e')- \sum_{\substack{\textrm{$e$ an edge}\\\textrm{between $v_j<v_i$}}} \omega(e)=0
\end{align*}
holds for all $i\in [n]$, where $\delta_{v_i}^{(1,1,1)}$ is the subset of $\delta_{v_i}$ that contains all ends of primitive direction $(1,1,1)$, $\delta_{v_i}^\beta$ is the subset of $\delta_{v_i}$ that contains all ends of primitive direction $(0,0,1)$ and $\delta_{v_i}^\alpha$ is the subset of $\delta_{v_i}$ that contains all ends of direction $(0,0,-1)$.
\item \label{item_3:def:cr_floor_diagram}
The graph $\mathcal{F}$ satisfies the given degenerated tropical cross-ratios $\lambda_{[n]}$. More precisely, $\mathcal{F}$ satisfies a degenerated tropical cross-ratio $\lambda_j$ if $\mathcal{F}$ satisfies the path criterion (see Remark \ref{remark:path_criterion}) for $j\in [l]$, where the paths' end points are given as follows. If $\beta_t$ is an entry of $\lambda_j$ that is the label of a contracted end $x_i$ that satisfies a point condition $p_i$, then the end point associated to $\beta_t$ is the vertex $v_i$. If $\beta_t$ is the label of a non-contracted end instead, i.e. this end appears in $\delta_{v_i}^\alpha$ or $\delta_{v_i}^\beta$ for some vertex $v_i$, then the end point associated to $\beta_t$ is $v_i$. We say that the degenerated tropical cross-ratio $\lambda_j$ is satisfied at a vertex $v_i$ of $\mathcal{F}$ if the paths associated to $\lambda_j$ intersect only in $v_i$. The set of all tropical cross-ratios satisfied at a vertex $v_i$ is denoted by $\lambda_{v_i}$ and $\#\lambda_{v_i}\in\mathbb{N}$ is called the \textit{number of tropical cross-ratios at $v_i$}.
\item \label{item_4:def:cr_floor_diagram}
Define
\begin{align*}
A(v_i):=3\cdot\#\delta_{v_i}^{(1,1,1)}+ \#\delta_{v_i}^{\alpha}+ \#\delta_{v_i}^{\beta}+\val(v_i) -2 -\#\lambda_{v_i}-2( \#\delta_{v_i}^{\alpha,P}+ \#\delta_{v_i}^{\beta, P})- \#\delta_{v_i}^{\alpha,L}- \#\delta_{v_i}^{\beta,L}
\end{align*}
for every vertex $v_i$, where $\val(v_i)$ is the valency of $v_i$ in $\mathcal{F}$ and $\delta_{v_i}^{\alpha},\delta_{v_i}^{\beta},\delta_{v_i}^{\alpha,P},\delta_{v_i}^{\beta, P},\delta_{v_i}^{\alpha,L},\delta_{v_i}^{\beta,L}$ are submultisets of $\delta_{v_i}$ such that
\begin{itemize}
\item
$\delta_{v_i}^{\alpha}$ (resp. $\delta_{v_i}^{\beta}$) are the ends that are associated to $\alpha$ (resp. $\beta$),
\item
$\delta_{v_i}^{\alpha,P}\subset\delta_{v_i}^{\alpha}$ (resp. $\delta_{v_i}^{\beta,P}\subset\delta_{v_i}^{\beta}$) are the ends that satisfy some codimension one tangency conditions (see Notation \ref{notation:convention_conditions_related_to_ends}),
\item
$\delta_{v_i}^{\alpha,L}\subset\delta_{v_i}^{\alpha}$ (resp. $\delta_{v_i}^{\beta,L}\subset\delta_{v_i}^{\beta}$) are the ends that satisfy some codimension two tangency conditions (see Notation \ref{notation:convention_conditions_related_to_ends}).
\end{itemize}
The leak function of $\mathcal{F}$ is given by
\begin{align*}
\textrm{\, \,}\leak &(v_i)=
\begin{cases}
0 &\textrm{, if $\delta_{v_i}^{(1,1,1)}=\emptyset$} \\
A(v_i) &\textrm{, else}\\
\end{cases}
\end{align*}
Notice that the leak function determines the condition flow of type $3$ on $\mathcal{F}$ uniquely by Lemma \ref{lemma:condition_flow_unique}.
\end{enumerate}
\end{definition}

\begin{example}\label{ex:cross-ratio_floor_diagram_from_floor_graph}
Let $\Delta^3_4\left( (4,1,0,\dots), (2,0,\dots) \right)$ denote a degree as in Notation \ref{notation:standard_directions_and_alpha_beta_degrees} whose labeling is:

\begin{align*}
\begin{array}{c|c|c|c|c|c}
\textrm{Ends of primitive direction\dots} & e_3 & -e_3 & -e_1 & -e_2 & e_0 \\
\hline
\textrm{its associated labels} & 3,9 & 4,5,6,7,8 & 10,11,12,19 & 13,14,15,20 & 16,17,18,21
\end{array},
\end{align*}
such that the end of weight two is labeled by $8$. Let $p_{[2]}, P_{[9]\backslash [2]}, \lambda_1=\lbrace 1,2,3,7 \rbrace$ general positioned conditions with notation from Definition \ref{def:general_pos}. Recall the floor graph from Example \ref{ex:floor_graph}. Equipping it with discrete data as below turns it into a cross-ratio floor diagram $\mathcal{F}$ that satisfies $p_{[2]}, P_{[9]\backslash [2]}, \lambda_1$.

\begin{figure}[H]
\centering\hspace*{1.5cm}
\def\svgwidth{230pt}
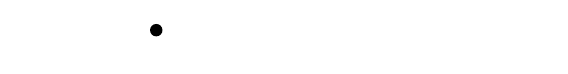
\label{Figure34}

\setlength{\tabcolsep}{3.5em}
\renewcommand{\arraystretch}{1.6} 
\begin{table}[H]
  \centering
      \centering
      \begin{tabular}{*{3}{|c}|}\hline
        $i$ & $1$ & $2$\\\hline
        $\delta_{v_i}$ & $[18]\backslash \lbrace 2,9 \rbrace$ & $\lbrace 2,9,19,20,21\rbrace$  \\\hline
        $\delta_{v_i}^{(1,1,1)}$& $\lbrace 16,17,18\rbrace$ & $\lbrace 21\rbrace$ \\\hline
        $\delta_{v_i}^\alpha$ & $\lbrace 4,5,6,7,8 \rbrace$ & $\emptyset$ \\\hline
        $\delta_{v_i}^\beta$ & $\lbrace 3 \rbrace$ & $\lbrace 9 \rbrace$ \\\hline
        $\lambda_{v_i}$ & $\lbrace \lambda_1 \rbrace$ & $\emptyset$ \\\hline
        $\delta_{v_i}^{\alpha,P}$ & $\lbrace 4,5,6,7,8 \rbrace$ & $\emptyset$ \\\hline
        $\delta_{v_i}^{\beta,P}$ & $\lbrace 3 \rbrace$ & $\lbrace 9 \rbrace$ \\\hline
        $\delta_{v_i}^{\alpha,L}$ & $\emptyset$ & $\emptyset$ \\\hline
        $\delta_{v_i}^{\beta,L}$ & $\emptyset$ & $\emptyset$ \\\hline
      \end{tabular}
\end{table}
\end{figure}

\end{example}

\begin{definition}[Multiplicity of a cross-ratio floor diagram]\label{def:mult_CR_floor_diagram}
Let $p_{[n]},L_{\underline{\kappa}^\alpha},L_{\underline{\kappa}^\beta},P_{\underline{\eta}^\alpha},P_{\underline{\eta}^\beta},\lambda_{[l]}$ be general positioned conditions and let $\mathcal{F}$ be a floor diagram as in Definition \ref{def:CR_floor_diagram} that satisfies the given conditions. Let $v_1<\cdots <v_n$ denote the totally ordered vertices of $\mathcal{F}$. For $\gamma=\alpha,\beta$, define the following:
Let $\underline{{2/0}_i}^\gamma$ be the $2/0$ edges adjacent to $v_i$ and $v_j$ (with $j\neq i$) in $\mathcal{F}$ such that $j<i$ if $\gamma=\alpha$ and $i<j$ if $\gamma=\beta$. Let $\underline{{1/1}_i}^\gamma$ be the $1/1$ edges adjacent to $v_i$, where $\gamma\in\lbrace\alpha,\beta\rbrace$ is defined analogously.

For a vertex $v_i$ of the cross-ratio floor diagram $\mathcal{F}$, its \textit{multiplicity} $\mult(v_i)$ is defined as
\begin{align*}
\mult(v_i):=N_{\Delta^3_{\#\delta_{v_i}^{(1,1,1)}}\left( \alpha^{i},\beta^{i} \right)}\left(p_i, L_{\delta_{v_i}^{\alpha,L} \cup \underline{{1/1}_i}^\alpha},L_{\delta_{v_i}^{\beta,L} \cup \underline{{1/1}_i}^\beta},P_{\delta_{v_i}^{\alpha,P} \cup \underline{{2/0}_i}^\alpha},P_{\delta_{v_i}^{\beta,P} \cup \underline{{2/0}_i}^\beta},\lambda^{\to}_{v_i} \right),
\end{align*}
with notation from \ref{def:CR_floor_diagram} and Notation \ref{notation:underlined_symbols}, where $\alpha^{i}$ (resp. $\beta^{i}$) of the degree $\Delta^3_{\#\delta_{v_i}^{(1,1,1)}}\left( \alpha^{i},\beta^{i} \right)$ arises from $\delta_{v_i}^\alpha$ (resp. $\delta_{v_i}^\beta$) and edges contributing to $\val(v_i)$ in $\mathcal{F}$, and where $L_{\underline{{1/1}_i}^\alpha}$ (resp. $L_{\underline{{1/1}_i}^\beta}$) are collections of tropical multi line conditions with ends of weight $1$. Moreover, the cross-ratios $\lambda_{v_i}$ are adapted to cutting the edges adjcent to $v_i$ similar to Construction \ref{constr:cutting_along_elevators} and Notation \ref{notation:adapt_CRs_to_multiple_cuts}. The \textit{multiplicity} of the entire cross-ratio floor diagram $\mathcal{F}$ is defined to be the product of the vertices' multiplicities times the edges' weights, i.e.
\begin{align*}
\mult(\mathcal{F}):=\prod_{\substack{\textrm{$e$ edge}\\\textrm{ of $\mathcal{F}$}}} \omega(e)\cdot\prod_{i=1}^n \mult(v_i).
\end{align*}
\end{definition}

\begin{example}\label{ex:multiplicity_cross-ratio_floor_diagram}
Let $\mathcal{F}$ be the cross-ratio floor diagram of Example \ref{ex:cross-ratio_floor_diagram_from_floor_graph}. Label its edge of weight two by $22$. The multiplicities of the vertices $v_1,v_2$ of $\mathcal{F}$ are
\begin{align*}
\mult(v_1)&=N_{\Delta^3_{3}\left( (4,1,0,\dots),(1,1,0,\dots) \right)}\left(p_1,L_{22},P_{[8]\backslash [2]}, \lambda_1^{\to 22} \right),\\
\mult(v_2)&=N_{\Delta^3_{1}\left( (0,1,0,\dots),(1,0,\dots) \right)}\left(p_2,L_{22},P_{9} \right),
\end{align*}
where $\lambda_1^{\to 22}=\lbrace 1,22,3,7\rbrace$ and $L_{22}$ is a tropical multi line with ends of weight $1$ (we use Notation \ref{notation:convention_conditions_related_to_ends}).
\end{example}

\begin{definition}
Given a set of general positioned conditions $p_{[n]},L_{\underline{\kappa}^\alpha},L_{\underline{\kappa}^\beta},P_{\underline{\eta}^\alpha},P_{\underline{\eta}^\beta},\lambda_{[l]}$ as in Definition \ref{def:CR_floor_diagram} with respect to the degree $\Delta^3_d(\alpha,\beta)$, the number of cross-ratio floor diagrams satisfying these conditions is defined by
\begin{align*}
N^{\floor}_{\Delta^3_d(\alpha,\beta)}\left( p_{[n]},L_{\underline{\kappa}^\alpha},L_{\underline{\kappa}^\beta},P_{\underline{\eta}^\alpha},P_{\underline{\eta}^\beta},\lambda_{[l]}\right)
:=
\sum_{\mathcal{F}} \mult(\mathcal{F}),
\end{align*}
where the sum goes over all cross-ratio floor diagrams that satisfy the given conditions.
\end{definition}

\subsection{Multiplicity of a cross-ratio floor diagram via curves in $\mathbb{R}^2$}

The following Proposition reduces the calculation of the multiplicity of a cross-ratio floor diagram to the enumeration of rational tropical stable maps to $\mathbb{R}^2$ satisfying point, multi line and cross-ratio conditions.

\begin{proposition}\label{prop:mult_CR_floor_diagram_2_dim}
For notation, see Notation \ref{notation:underlined_symbols}, \ref{notation:projection_moduli_spaces} and Definition \ref{def:mult_CR_floor_diagram}. The multiplicity $\mult(v_i)$ of a vertex $v_{i}$ of a cross-ratio floor diagram that satisfies general positioned conditions equals the degree of the cycle
\begin{align}\label{eq:proposition_mult_CR_floor_diagram_cycle}
\begin{split}
\prod_{k\in \underline{\kappa_{i}}^{\alpha}\cup \underline{\kappa_{i}}^{\beta}} \ev_k^{*}(L_k)
\cdot
\prod_{f\in \underline{\eta_{i}}^{\alpha}\cup \underline{\eta_{i}}^{\beta}} \ev_f^{*}(P_f)
\cdot
\prod_{\lambda_j\in \lambda_{v_i}} &\ft_{\lambda^{\to}_j}^*\left( 0\right)
\cdot
\ev_i^*\left( \pi(p_i)\right) \\
&\cdot
\mathcal{M}_{0,1+|\alpha^{i}|+|\beta^{i}|}\left(\mathbb{R}^2,\pi\left( \Delta^3_{\#\delta_{v_i}^{(1,1,1)}}\left( \alpha^{i},\beta^{i} \right) \right)\right),
\end{split}
\end{align}
where $\underline{\kappa_{i}}^{\gamma}:=\delta_{v_i}^{\gamma,L} \cup \underline{{1/1}_i}^\gamma$ and $\underline{\eta_{i}}^{\gamma}:=\delta_{v_i}^{\gamma,P} \cup \underline{{2/0}_i}^\gamma$ for $\gamma=\alpha,\beta$.
\end{proposition}

\begin{proof}
Notice that since the given conditions $p_i, L_{\underline{\kappa_{i}}^{\alpha}},L_{\underline{\kappa_{i}}^{\beta}},P_{\underline{\eta_{i}}^{\alpha}},P_{\underline{\eta_{i}}^{\beta}},\lambda^{\to}_{v_i}$ are in general position with respect to $\Delta^3_{\#\delta_{v_i}^{(1,1,1)}}\left( \alpha^{i},\beta^{i} \right)$ and there is only one point condition $p_i$, we can assume that the conditions $\pi(p_i), L_{\underline{\kappa_{i}}^{\alpha}},L_{\underline{\kappa_{i}}^{\beta}},P_{\underline{\eta_{i}}^{\alpha}},P_{\underline{\eta_{i}}^{\beta}},\lambda^{\to}_{v_i}$ are also in general position with respect to the degree $\pi\left( \Delta^3_{\#\delta_{v_i}^{(1,1,1)}}\left( \alpha^{i},\beta^{i} \right) \right)$. Using \eqref{eq:general_dimension_count_arbitrary_contracted_ends}, we see that the cycle \eqref{eq:proposition_mult_CR_floor_diagram_cycle} is indeed $0$-dimensional. Thus considering its degree makes sense.

Let $C$ be a tropical stable map contributing to $\mult(v_i)$. Applying the map $\tilde{\pi}$ from Notation \ref{notation:projection_moduli_spaces} induced by the projection $\pi:\mathbb{R}^3\to\mathbb{R}^2$ that forgets the $z$-coordinate leads to a tropical stable map $\tilde{\pi}(C)$ that contributes to \eqref{eq:proposition_mult_CR_floor_diagram_cycle}. The other way round, a tropical stable map $C'$ that contributes to \eqref{eq:proposition_mult_CR_floor_diagram_cycle} can be lifted uniquely to a tropical stable map $C$ that contributes to $\mult(v_i)$, because the $z$-coordinates of the directions of the edges can be recovered from the balancing condition and the overall $z$-position of $C$ is fixed by the $z$-coordinate of $p_i$. Hence $\pi$ induces a bijection between tropical stable maps $C$ that contribute to $\mult(v_i)$ and tropical stable maps $\tilde{\pi}(C)$ that contribute to \eqref{eq:proposition_mult_CR_floor_diagram_cycle}.

It remains to show that the multiplicities of $C$ and $\tilde{\pi}(C)$ coincide. For that notice that the cross-ratio multiplicities of every vertex $v\in C$ and its image $\pi(v_i)$ in $\tilde{\pi}(C)$ coincide. Thus is remains to show that the $\ev$-multiplicities coincide as well.
The $\ev$-multiplicity of $C$ (resp. $\tilde{\pi}(C)$) is given by the absolute value of the determinant of the $\ev$-Matrix $M(C)$ (resp. the $\ev$-matrix $M(\tilde{\pi}(C))$) associated to $C$ (resp. $\tilde{\pi}(C)$), see Definition \ref{def:ev_matrix}. We choose $p_i$ as base point for the local coordinates used for $M(C)$ (resp. $\pi(p_i)$ as base point for $M(\tilde{\pi}(C))$) which are the lengths of the edges of $C$ (resp. $\tilde{\pi}(C)$). The matrix $M(\tilde{\pi}(C))$ is obtained from $M(C)$ by erasing the third column and row which intersect in the $z$-coordinate of the base point, which is $1$ (see below). The matrices $M(C)$ and $M(\tilde{\pi}(C))$ look like follows
\begin{align*}
M(C)=
\begin{array}{c ccc| ccc c}
  & \multicolumn{3}{c}{\footnotesize \textrm{Base $p_i$}} & &&&\\
  \ldelim({6}{0.5em} &1&0&0& 0& \dots & 0& \rdelim){6}{0.5em} \\
  &0&1&0&0&\dots&0&\\
  &0&0&1&0&\dots&0&\\
  \cline{1-8} & \multicolumn{2}{c}{\multirow{3}*{$B$}} & \multicolumn{1}{|c|}{\multirow{1}*{$*$}} &  \multicolumn{3}{c}{\multirow{3}*{$M$}} & \\
&&&\multicolumn{1}{|c|}{\multirow{1}*{\vdots}}&&&&\\
&&&\multicolumn{1}{|c|}{\multirow{1}*{$*$}}&&&&\\
\end{array} \textrm{\quad and \quad}
M(\tilde{\pi}(C))=
\begin{array}{c cc| ccc c}
  & \multicolumn{2}{c}{\footnotesize \textrm{Base $\pi(p_i)$}} & &&&\\
  \ldelim({5}{0.5em} &1&0& 0& \dots & 0& \rdelim){5}{0.5em} \\
  &0&1&0&\dots&0&\\
  \cline{1-7} & \multicolumn{2}{c|}{\multirow{3}*{$B$}} &  \multicolumn{3}{c}{\multirow{3}*{$M$}} & \\
&&&&&&\\
&&&&&&\\
\end{array}.
\end{align*}
Recall that $\partial\ev$ is by definition $\ev\circ\pi$, i.e. the submatrices $B,M$ marked above are equal. Therefore
\begin{align*}
|\det(M(C))|=|\det(M(\tilde{\pi}(C))|
\end{align*}
follows from using Laplace expansion on the third row of $M(C)$.
\end{proof}

There is a \textit{general Kontsevich's formula} \cite{GeneralKontsevich} which recursively calculates the weighted number of rational tropical curves in $\mathbb{R}^2$ that satisfy point, multi line and cross-ratio conditions. As a consequence, the multiplicity of a cross-ratio floor diagram can be determined recursively.

\begin{corollary}\label{cor:mult_CR_floor_diagram_via_kontsevich}
The multiplicity $\mult(v_i)$ of a vertex $v_{i}$ of a cross-ratio floor diagram can be calculated recursively using the general Kontsevich's formula from \cite{GeneralKontsevich}.
\end{corollary}

\begin{example}\label{ex:multiplicity_cross-ratio_floor_diagram_apply_general_kontsevich}
Let $\mathcal{F}$ be the cross-ratio floor diagram of Example \ref{ex:cross-ratio_floor_diagram_from_floor_graph}.
Using Proposition \ref{prop:mult_CR_floor_diagram_2_dim} to express $\mult(v_1),\mult(v_2)$ of Example \ref{ex:multiplicity_cross-ratio_floor_diagram} yields
\begin{align*}
\mult(v_1)&=N_{\Delta^2_{3}}\left(\pi(p_1),\pi\left(P_{[8]\backslash [2]}\right), L_{22} \lambda_1^{\to 22} \right),\\
\mult(v_2)&=N_{\Delta^2_{1}}\left(\pi(p_2),\pi\left(P_{9}\right), L_{22} \right).
\end{align*}
This allows us to use Corollary \ref{cor:mult_CR_floor_diagram_via_kontsevich}, resp. general Kontsevich's formula \cite{GeneralKontsevich}. Hence
\begin{align*}
\mult(v_1)&=5,\\
\mult(v_2)&=1.
\end{align*}
Therefore $\mult(\mathcal{F})=10$.
\end{example}

\section{Counting tropical curves using cross-ratio floor diagrams}

The aim of this section is to prove the following theorem which is the main result of this paper. It reduces the count of spatial curves satisfying given conditions to a counting problem of cross-ratio floor diagrams. There are only finitely many cross-ratio floor diagrams to given conditions. Thus the weighted number of cross-ratio floor diagrams satisfying given conditions can be determined by going through all possible cross-ratio floor-diagrams, which then answers the initial counting problem.

\begin{theorem}\label{thm:counting_CR_floor_diagrams=counting_curves}
For notations, see Definition \ref{def:general_pos}, \ref{def:numbers_of_interest} and Notation \ref{notation:underlined_symbols}, \ref{notation:standard_directions_and_alpha_beta_degrees}.
Consider general positioned conditions $p_{[n]},L_{\underline{\kappa}^\alpha},L_{\underline{\kappa}^\beta},P_{\underline{\eta}^\alpha},P_{\underline{\eta}^\beta},\lambda_{[l]}$ with respect to a degree $\Delta_d^3(\alpha,\beta)$ such that each entry of a degenerated tropical cross-ratio is a label of a contracted end or a label of an end of primitive direction $(0,0,\pm 1)\in\mathbb{R}^3$. Then
\begin{align}\label{eq:thm:counting_CR_floor_diagrams=counting_curves}
N^{\floor}_{\Delta^3_d(\alpha,\beta)}\left( p_{[n]},L_{\underline{\kappa}^\alpha},L_{\underline{\kappa}^\beta},P_{\underline{\eta}^\alpha},P_{\underline{\eta}^\beta},\lambda_{[l]}\right)
=
N_{\Delta^3_d(\alpha,\beta)}\left( p_{[n]},L_{\underline{\kappa}^\alpha},L_{\underline{\kappa}^\beta},P_{\underline{\eta}^\alpha},P_{\underline{\eta}^\beta},\lambda_{[l]}\right)
\end{align}
holds, i.e. the weighted count of cross-ratio floor diagrams satisfying the given conditions equals the weighted count of rational tropical stable degree $\Delta_d^3(\alpha,\beta)$ maps to $\mathbb{R}^3$ satisfying the given conditions.
\end{theorem}

\begin{construction}[Floor decomposed tropical stable map $\mapsto$ cross-ratio floor diagram]\label{constr.:associating_CR_floor_diagram_to_curve}
Let $C$ be a floor decomposed tropical stable map contributing to $N_{\Delta_d^3(\alpha,\beta)}\left(p_{[n]},L_{\underline{\kappa}^\alpha},L_{\underline{\kappa}^\beta},P_{\underline{\eta}^\alpha},P_{\underline{\eta}^\beta},\lambda_{[l]} \right)$. We want to construct a cross-ratio floor diagram $\mathcal{F}$ as in Definition \ref{def:CR_floor_diagram} (that satisfies given conditions) with vertices $v_1<\cdots <v_n$.

Let $\mathcal{F}$ denote the floor graph associated to $C$, see Definition \ref{def:floor_graph}.
To make sure that $\mathcal{F}$ is indeed a cross-ratio floor diagram, properties \ref{item_1:def:cr_floor_diagram}, \ref{item_2:def:cr_floor_diagram}, \ref{item_3:def:cr_floor_diagram} and \ref{item_4:def:cr_floor_diagram} of Definition \ref{def:CR_floor_diagram} must be satisfied. For \ref{item_1:def:cr_floor_diagram}, define $\delta_{v_i}$ as the multiset of ends adjacent to the floor $C_i$ of $C$ which satisfies the point condition $p_i$. Property \ref{item_2:def:cr_floor_diagram} follows from balancing of $C$. For \ref{item_3:def:cr_floor_diagram}, define $\lambda_{v_i}$ as the union over all $\lambda_{u_j}$ (the set of cross-ratios satisfied at $u_j$), where $u_j$ is a vertex of the floor $C_i$, and use the path criterion (Remark \ref{remark:path_criterion}) to verify that $\mathcal{F}$ satisfies the cross-ratios $\lambda_{[l]}$ if $C$ does. Property \ref{item_4:def:cr_floor_diagram} is more technical:
If the floor $C_i$ does not contain an end of direction $(1,1,1)$ (i.e. if $\delta_{v_i}^{(1,1,1)}=\emptyset$ with the notation from Definition \ref{def:CR_floor_diagram}), then $\flow(v_i)=0$ since $C_i$ consists of a single vertex satisfying the point condition $p_i$ that gains all its flow in $C_i$ via a contracted end which is not contained in $\mathcal{F}$. Let now $\delta_{v_i}^{(1,1,1)}\neq\emptyset$ and let the \textit{elevator flow} $\flow_{\operatorname{elevator}}(C_i)$ of $C_i$ be the total flow incoming to vertices of $C_i$ via elevators. Let the \textit{end flow} $\flow_{\operatorname{end}}(C_i)$ of $C_i$ be the total flow incoming to vertices of $C_i$ via non-contracted ends (notice that notation is abused here as indicated in Definition \ref{def:floor_decomposed}). Since $C_i$, that is of degree $\Delta^3_{s_i}\left( \alpha^{i},\beta^{i} \right)$ (notation of Construction \ref{constr:cutting_along_elevators}), is fixed by all restrictions imposed to it via its ends and edges, we can use Equation \eqref{eq:general_dimension_count} and the notation of Definition \ref{def:CR_floor_diagram} to obtain
\begin{align*}
\#\Delta^3_{s_i}\left( \alpha^{i},\beta^{i} \right)-2-\#\lambda_{v_i}-\flow_{\operatorname{end}}(C_i)=\flow_{\operatorname{elevator}}(C_i),
\end{align*}
where $-2$ comes from the point condition $p_i$ satisfied by $C_i$. Using
\begin{align*}
\#\Delta^3_{s_i}\left( \alpha^{i},\beta^{i} \right)=3\cdot\#\delta_{v_i}^{(1,1,1)}+ \#\delta_{v_i}^{\alpha}+ \#\delta_{v_i}^{\beta}+\val(v_i)
\end{align*}
and
\begin{align*}
\flow_{\operatorname{end}}(C_i)=2\cdot \#\delta_{v_i}^{\alpha,P} + 2\cdot\#\delta_{v_i}^{\beta, P}+ \#\delta_{v_i}^{\alpha,L}+ \#\delta_{v_i}^{\beta,L}
\end{align*}
turns the induced flow on $\mathcal{F}$ into a condition flow of type $3$ if the leak function is defined as $\leak(v_i):=\flow_{\operatorname{elevator}}(C_i)$ for all $i=1,\dots,n$. Notice that this leak function coincides with the one of Definition \ref{def:CR_floor_diagram}. In this case, the condition flow is uniquely determined by its leak function (see Lemma \ref{lemma:condition_flow_unique}). Hence $\mathcal{F}$ is a cross-ratio floor diagram satisfying the given conditions.
We say that $C$ \textit{degenerates} to $\mathcal{F}$ and denote it by $C\to \mathcal{F}$.
\end{construction}

\begin{example}\label{ex:two_floor_decomposed_stable_maps_degenerating_to_one_CR_floor_diagram}
Let $\mathcal{F}$ be the cross-ratio floor diagram of Example \ref{ex:cross-ratio_floor_diagram_from_floor_graph}. Observe that the floor decomposed tropical stable map $C$ of Example \ref{ex:floor_decomposed_tropical_stable_map} degenerates to $\mathcal{F}$ if the labels of ends that are not shown in Figure \ref{Figure32} are chosen appropriately, i.e. to fit Example \ref{ex:cross-ratio_floor_diagram_from_floor_graph}.

Figure \ref{Figure35} shows another floor decomposed tropical stable map $D$. It satisfies the degenerated tropical cross-ratio $\lambda_1=\lbrace 1,2,3,7\rbrace$. Moreover, we claim that it is possible to assign lengths to the bounded edges of $D$ in its schematic representation in Figure \ref{Figure35} in such a way that $D$ satisfies the same conditions as $C$. The conditions in question are point conditions $p_1,p_2$ and the codimension one tangency conditions $P_{[9]\backslash [2]}$. Cut the elevators of $C$ and $D$, project the floors to $\mathbb{R}^2$ using $\pi$ as in the proof of Proposition \ref{prop:mult_CR_floor_diagram_2_dim}. Notice that it is sufficient to check whether the projections of the floors $C_1$ and $D_1$ satisfy the same conditions. For that, use the cross-ratio lattice path algorithm from \cite{CR1} with the degenerated cross-ratio $\lambda_1$ to obtain the projections $\pi(C_1)$ and $\pi(D_1)$ that then satisfy $\pi\left(p_{[2]}\right)$ and $\pi\left(P_{[9]\backslash [2]}\right)$. Lifting $\pi(C_1)$ and $\pi(D_1)$ yields the desired lengths. The lattice path calculation can be found in Example 3.15, Figure 9 of \cite{CR1}. More precisely, $\pi(C_1)$ corresponds to the entry (read as a matrix) $(6,2)$, and $\pi(D_1)$ corresponds to the entry $(4,2)$ of Figure 9 there.

Thus $D$ degenerates to $\mathcal{F}$ as well if the missing labels in Figure \ref{Figure35} are chosen appropriately.

\begin{figure}[H]
\centering
\def\svgwidth{450pt}
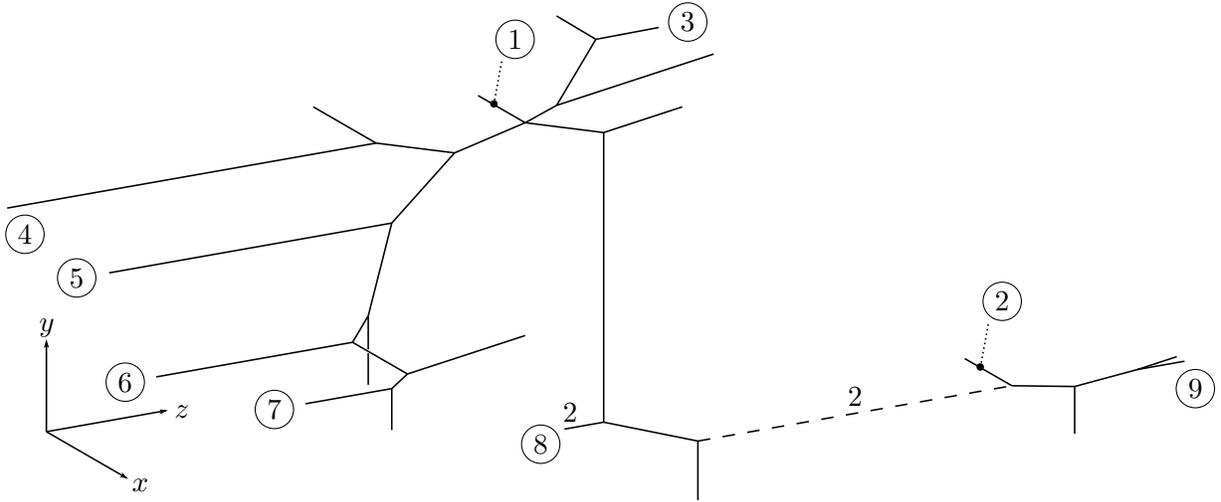
\caption{The tropical stable map $D$ from Example \ref{ex:two_floor_decomposed_stable_maps_degenerating_to_one_CR_floor_diagram} which is floor decomposed. It has two floors $D_i$ for $i=1,2$. The dashed edge is the elevator of weight two of $D$.}
\label{Figure35}
\end{figure}
\end{example}

\begin{proof}[Proof of Theorem \ref{thm:counting_CR_floor_diagrams=counting_curves}]
Since the numbers in question are independent of the exact positions of the given conditions, we can assume that all given conditions are in stretched configuration (see Definition \ref{def:stretched_config}). Thus (by Proposition \ref{prop:floor_decomposed}) every tropical stable map contributing to the right-hand side of \eqref{eq:thm:counting_CR_floor_diagrams=counting_curves} is floor decomposed. Hence Construction \ref{constr.:associating_CR_floor_diagram_to_curve} associates a cross-ratio floor diagram to every tropical stable map contributing to the right-hand side of \eqref{eq:thm:counting_CR_floor_diagrams=counting_curves}. Therefore it is sufficient to show for a fixed cross-ratio floor diagram $\mathcal{F}$ (that satisfies the given conditions) that
\begin{align}\label{eq:claim2:thm:counting_CR_floor_diagrams=counting_curves}
\mult(\mathcal{F})=\sum_{C \to\mathcal{F}} \mult(C)
\end{align}
holds (where the sum goes over all $C$ degenerating to $\mathcal{F}$), i.e. that the multiplicity with which a cross-ratio floor diagram $\mathcal{F}$ is counted equals the sum of the multiplicities of all tropical stable maps $C$ contributing to the right-hand side of \eqref{eq:thm:counting_CR_floor_diagrams=counting_curves} such that $C$ degenerates to $\mathcal{F}$. So fix a cross-ratio floor diagram $\mathcal{F}$.

To shorten notation, let $B:=\lbrace p_{[n]},L_{\underline{\kappa}^\alpha},L_{\underline{\kappa}^\beta},P_{\underline{\eta}^\alpha},P_{\underline{\eta}^\beta},\lambda_{[l]} \rbrace$ be the set of conditions that $\mathcal{F}$ satisfies, and let $N^{\floor}(B)$ (resp. $N(B)$) denote the number on the left-hand side (resp. the right-hand side) of \eqref{eq:thm:counting_CR_floor_diagrams=counting_curves}. Assume that $\mathcal{F}$ has more than $1$ vertex, because otherwise there is nothing to show. Since $\mathcal{F}$ is a tree, there is a $1$-valent vertex $v_i$ of $\mathcal{F}$ adjcacent to a vertex $v_j$ with $i<j$ via an edge $q$. There are two cases: $q$ is either a $1/1$ edge or a $2/0$ edge. First, assume that $q$ is a $1/1$ edge. Let $L^{(j)}$ (resp. $L^{(i)}$) be the codimension two condition from Corollary \ref{cor:cut_1/1_elevators_yield_standard_directions} which $v_i$ passes to $v_j$ via $q$ (resp. $v_j$ passes to $v_i$).

We follow the idea of recursively moving conditions in such a way that $\mult(C)$ can be calculated using a single graphical contribution, namely one similar to $G_1$ in Example \ref{ex:graphical_contributions}.  Cut $q$ to obtain two new cross-ratio floor diagrams $\mathcal{F}_i$ and $\mathcal{F}_j$, where $\mathcal{F}_i$ consists of a single vertex $v_i$ and $\mathcal{F}_j$ is given by $\mathcal{F}$ without $v_i$. Decomposing $\mathcal{F}$ into $\mathcal{F}_i$ and $\mathcal{F}_j$ decomposes $B$ into $B_i$ and $B_j$ as well, more precisely, let $B_i\subset B$ (resp. $B_j\subset B$) be the subset of conditions $\mathcal{F}_i\subset \mathcal{F}$ (resp. $\mathcal{F}_j$) satisfies. Notice that the set of all conditions $\mathcal{F}_i$ (resp. $\mathcal{F}_j$) satisfies is $B_i\cup L^{(i)}$ (resp. $B_j\cup L^{(j)}$).

If we change the $x$- and $y$-coordinates of the conditions in $B_i$ and $B_j$ in such a way that all conditions are still in a stretched configuration, then any tropical stable map $C$ with $C\to\mathcal{F}$ is still floor decomposed and still degenerates to $\mathcal{F}$. Notice that moving conditions as above moves $L^{(i)}$ and $L^{(j)}$ accordingly.
Hence we can achieve that $L^{(i)}$ and $L^{(j)}$ intersect in the following way: all points of the intersection of $L^{(i)}$ and $L^{(j)}$ are on the ends of $L^{(i)}$ that are of primitive direction $(0,-1)\in\mathbb{R}^2$, and on the ends of $L^{(j)}$ that are of primitive direction $(-1,0)\in\mathbb{R}^2$, see Figure \ref{Figure_17}.
Thus, using part (b) of Lemma \ref{lemma:weights_cutting_elevators}, Notation \ref{notation:replacing_tangency_conditions_1/1-edge} and notation from Definition \ref{def:graphical_contribution}, we have
\begin{align*}
\sum_{C \to\mathcal{F}} \mult(C)=\omega(q) \sum_{C \to\mathcal{F}}\mult(C_{i,\emptyset;10})\mult(C_{j,01;\emptyset}),
\end{align*}
where $\omega(q)$ is the weight of our cut edge $q$ and $C_i,C_j$ are the pieces obtained from $C$ by cutting $e$.

\begin{figure}[]
\centering
\def\svgwidth{270pt}
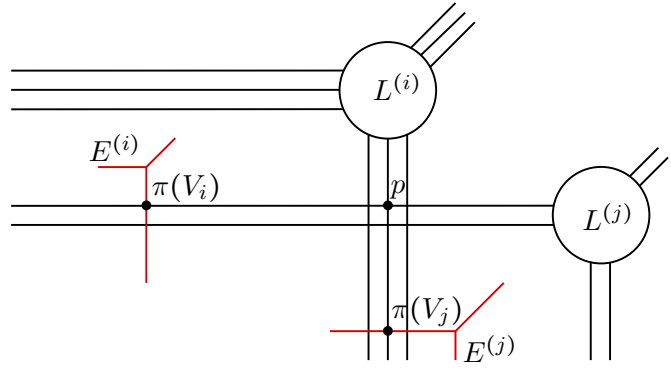
\caption{Codimension two tangency conditions $L^{(i)}$ and $L^{(j)}$ after movement, together with the codimension two tangency conditions $E^{(i)}$ and $E^{(j)}$, where $p\in L^{(i)}\cap L^{(j)}$ is the point associated to $\pi(V_i)$ and $\pi(V_j)$.}
\label{Figure_17}
\end{figure}

We claim that
\begin{align}\label{eq:claim1:thm:counting_CR_floor_diagrams=counting_curves}
\sum_{C \to\mathcal{F}}\mult(C_{i,\emptyset;10})\mult(C_{j,01;\emptyset})=N(B_i\cup \lbrace E^{(i)} \rbrace )\sum_{C_j \to\mathcal{F}_j} \mult(C_j),
\end{align}
where $\mathcal{F}_j$ is understood as a cross-ratio floor diagram that satisfies the conditions $B_j\cup \lbrace E^{(j)} \rbrace
$, and where $E^{(i)}$ and $E^{(j)}$ are codimension two tangency conditions that are tropical multi lines in $\mathbb{R}^2$ with ends of weight $1$.
To see this, let $E^{(i)}$ and $E^{(j)}$ be two tropical lines with ends of weight $1$ whose positions are chosen according to Figure \ref{Figure_17}, i.e. choose $E^{(i)}$ (resp. $E^{(j)}$) in such a way that it intersects $L^{(j)}$ (resp. $L^{(i)}$) only in its rays of primitive direction $(-1,0)$ (resp. $(0,-1)$) and such that each point of intersection lokally looks like the $x$ and $y$ axes' intersection.

Each tropical stable map $C_i$ contributing to $N(B_i\cup \lbrace E^{(i)} \rbrace )$ has an end $q$ parallel to the $z$-axis whose adjacent vertex $V_i$ is $3$-valent and satisfies $\pi(V_i)\in L^{(j)}\cap E^{(i)}$, where $\pi$ is the projection that forgets the $z$-coordinate. This is true due to Proposition \ref{prop:push-forward} and since $C_i$ satisfies $L^{(j)}$ by definition of $L^{(j)}$. Analogously, by definition of $L^{(i)}$ and Proposition \ref{prop:push-forward}, each tropical stable map $C_j$ from the right-hand side of \eqref{eq:claim1:thm:counting_CR_floor_diagrams=counting_curves} has an end $q$ parallel to the $z$-axis whose adjacent vertex $V_j$ is $3$-valent and satisfies $\pi(V_j)\in L^{(i)}\cap E^{(j)}$.
Since $V_i$ (resp. $V_j$) is by Proposition \ref{prop:push-forward} adjacent to an end of $C_i$ (resp. $C_j$), we can move $V_i$ and $V_j$ as in Figure \ref{Figure_17} to the corresponding point of intersection of $L^{(j)}\cap L^{(i)}$ such that the combinatorials types of $C_i$ and $C_j$ do not change and such that the multiplicities of $C_i$ and $C_j$ understood as tropical stable maps contributing to the right-hand side of \eqref{eq:claim1:thm:counting_CR_floor_diagrams=counting_curves} do not change. Since we moved $V_i$ and $V_j$ to one point, we can glue $C_i$ and $C_j$ to obtain a tropical stable map $C$ such that $C\to\mathcal{F}$ and the multiplicities of $C_i$ and $C_j$ (understood as tropical stable maps contributing to the right-hand side of \eqref{eq:claim1:thm:counting_CR_floor_diagrams=counting_curves}) are equal to $\mult(C_{i,\emptyset;10})$ and $\mult(C_{j,01;\emptyset})$ by our special choice of the positions of $L^{(i)}$ and $L^{(j)}$. Reversing the process of glueing yields a bijection between factors of the left and factors of the right-hand side of \eqref{eq:claim1:thm:counting_CR_floor_diagrams=counting_curves}.

The multiplicity of the vertex $v_i$ of $\mathcal{F}$ equals $N(B_i\cup \lbrace E^{(i)} \rbrace )$ by Definition \ref{def:mult_CR_floor_diagram}. Moreover, if $q$ is a $2/0$ edge instead, then part (a) of Lemma \ref{lemma:weights_cutting_elevators} guarantees that multiplicities splite nicely if edges are cut, so in total \eqref{eq:claim1:thm:counting_CR_floor_diagrams=counting_curves} gives rise to a recursion that eventually yields
\begin{align*}
\sum_{C \to\mathcal{F}}\mult(C_{i,\emptyset;10})\mult(C_{j,01;\emptyset})=\prod_{i=1}^n \mult(v_i)
\end{align*}
since $\mathcal{F}$ is a tree. Hence \eqref{eq:claim2:thm:counting_CR_floor_diagrams=counting_curves} holds.

Notice that $L^{(i)}$ and $L^{(j)}$ depend on the choice of the floor diagram $\mathcal{F}$. So we should use the notation $L_\mathcal{F}^{(i)}$ and $L_\mathcal{F}^{(j)}$ instead.
It remains to show that we can bring $L_\mathcal{F}^{(i)}$ and $L_\mathcal{F}^{(j)}$ in a position as above for each choice of floor diagram $\mathcal{F}$ without effecting the overall weighted count of cross-ratio floor diagrams $N^{\floor}_{\Delta^3_d(\alpha,\beta)}\left( p_{[n]},L_{\underline{\kappa}^\alpha},L_{\underline{\kappa}^\beta},P_{\underline{\eta}^\alpha},P_{\underline{\eta}^\beta},\lambda_{[l]}\right)$. Moving conditions as above does not lead to a tropical stable map degenerating to another cross-ratio floor diagram then it initially did, and the cycle obtained by moving conditions (i.e. by relaxing some of the initially given conditions) as above is balanced. Therefore we can assume that $L_\mathcal{F}^{(i)}$ and $L_\mathcal{F}^{(j)}$ are always in a position as shown in Figure \ref{Figure_17}.
\end{proof}

\begin{corollary}\label{cor:count_CR_floor_diag=classical_count}
Notation of Theorem \ref{thm:correspondence_thm} is used. Let $\mu_{[l]}$ be non-tropical cross-ratios tropicalizing to $\lambda'_{[l]}$ and let $p_{[n]}$ be point conditions such that $p_{[n]},\lambda_{[l]}$, where $\lambda'_j$ degenerates to $\lambda_j$ for $j\in [l]$, are in general position with respect to the degree $\Delta_d^3(\alpha,\beta)$ and such that each entry of a degenerated tropical cross-ratio is a label of a contracted end or a label of an end of primitive direction $(0,0,\pm 1)\in\mathbb{R}^3$. Then
\begin{align*}
N^{\operatorname{alg}}_{\Delta_d^3(\alpha,\beta)}\left(p_{[n]},\mu_{[l]}\right)
=
N^{\floor}_{\Delta^3_d(\alpha,\beta)}\left( p_{[n]},\lambda_{[l]}\right)
\end{align*}
holds. Thus rational algebraic space curves that satisfy non-tropical cross-ratio conditions and point conditions can be enumerated using cross-ratio floor diagrams.
\end{corollary}

\begin{proof}
Combine correspondence theorem \ref{thm:correspondence_thm}, Proposition \ref{prop:zsfssg_int_theory} and Theorem \ref{thm:counting_CR_floor_diagrams=counting_curves}.
\end{proof}

\hyphenation{Kaisers-lautern} 
\bibliographystyle{alpha}
\bibliography{literatur}

\end{document}